\documentclass[a4paper,10pt]{amsart}
\usepackage{color}
\definecolor{Chocolat}{rgb}{0.36, 0.2, 0.09}
\definecolor{BleuTresFonce}{rgb}{0.215, 0.215, 0.36}
\usepackage[colorlinks,final,hyperindex]{hyperref}
\hypersetup{citecolor=BleuTresFonce, linkcolor=Chocolat}

\usepackage{graphicx}
\usepackage{epstopdf}
\usepackage[all]{xy}
\usepackage{amsmath,amsfonts,amssymb,MnSymbol}
\usepackage[centering,vscale=0.68,hscale=0.68]{geometry}
\usepackage{latexsym,amscd}
\usepackage{tikz}
\usetikzlibrary{arrows,decorations.markings,shapes.geometric}
\tikzset{>=latex}
\usepackage{url}
\usepackage{dsfont}
\usepackage{microtype}
\usepackage{todonotes}
\makeatletter
\providecommand\@dotsep{5}
\renewcommand{\listoftodos}[1][\@todonotes@todolistname]{%
  \@starttoc{tdo}{#1}}
\makeatother

\newtheorem{theorem}{Theorem}[section]
\newtheorem{lemma}[theorem]{Lemma}
\newtheorem{proposition}[theorem]{Proposition}
 
\theoremstyle{definition}  
\newtheorem{definition}[theorem]{Definition}
\newtheorem{example}[theorem]{Example}

\newtheorem{remark}[theorem]{Remark}

\newcommand{\Der}{{\operatorname{Der}}}
 
\newcommand{\x}{\operatorname{x}} 
\newcommand{\y}{\operatorname{y}} 
 
\newcommand{\Ad}{\operatorname{Ad}}
\newcommand{\ad}{{\operatorname{ad}}}
\newcommand{\<}{\langle} 
\renewcommand{\>}{\rangle} 
\newcommand{\Aut}{\operatorname{Aut}}
\newcommand{\on}{\operatorname}

\newcommand{\kk}{\mathbf{k}}
\newcommand{\Z}{{\mathbb Z}}
\newcommand{\Q}{{\mathbb Q}}
\newcommand{\C}{{\mathbb C}}

\newcommand{\NN}{{\mathbb N}}
\newcommand{\KK}{\mathbf{k}}
\newcommand{\0}{\mathbf{0}}

\newcommand{\PaB}{\mathbf{PaB}}
\newcommand{\PaCD}{\mathbf{PaCD}}

\renewcommand{\sl}{\mathfrak{sl}}

\newcommand{\pb}{{\mathfrak{pb}}}
\renewcommand{\t}{{\mathfrak{t}}}

\renewcommand{\d}{{\mathfrak d}}
\newcommand{\g}{{\mathfrak{g}}}

\newcommand{\HH}{{\mathfrak H}}
\renewcommand{\u}{{\mathfrak{u}}}

\newcommand{\h}{{\mathfrak{h}}}
\renewcommand{\l}{{\mathfrak{l}}}
\newcommand{\n}{{\mathfrak{n}}}
\newcommand{\m}{{\mathfrak{m}}}

\newcommand{\SG}{{\mathfrak{S}}}
\newcommand{\f}{{\mathfrak{f}}}

\renewcommand{\i}{\on{i}}
\newcommand{\Diff}{\on{Diff}}
\newcommand{\PB}{\on{PB}}

\newcommand{\zz}{\mathbf{z}}

\newcommand{\GT}{\mathbf{GT}}
\newcommand{\GRT}{\mathbf{GRT}}

\newcommand{\Lie}{{\text{Lie}}}

\newcommand{\cO}{{\mathcal O}}

\newcommand{\mP}{{\mathcal P}}

\newcommand{\M}{{\mathcal M}}

\newcommand{\cM}{{\mathcal M}}
\newcommand{\ddd}{\mathbf{d}}
\newcommand{\XXX}{\mathbf{X}}

\tikzstyle cross=[preaction={draw=white, -, line width=4pt}, thick]
\tikzstyle normal=[thick]
\tikzstyle chord=[densely dotted, thick]
\tikzstyle zero=[ultra thick, gray]
\tikzstyle zell=[ultra thick, white]
\tikzstyle zerocross=[preaction={draw=white, -, line width=4pt}, ultra thick, gray]
\tikzstyle point=[draw,circle,inner sep=1,fill=black]
\tikzstyle diam=[draw,diamond,inner sep=1,fill=black]
\tikzstyle petitpoint=[draw,circle,inner sep=0.3,fill=black]

\newcommand{\noplus}{}
\newcommand{\tmmathbf}[1]{\ensuremath{\boldsymbol{#1}}}
\newcommand{\tmop}[1]{\ensuremath{\operatorname{#1}}}

\usepackage[toc,section=subsection]{glossaries}

\newglossary*{notation}{List of Notation}
\newglossary*{operads}{Operads}
\newglossary*{groups}{Groups}
\newglossary*{spaces}{Spaces}
\newglossary*{liealgebras}{Lie and associative algebras}
\newglossary*{torsors}{Torsor sets}
\newglossary*{bundles}{Bundles}
\newglossary*{associators}{Series}

 \makenoidxglossaries

\newglossaryentry{PaB}{type=operads, name=$\PaB$, 
  description={Operad of parenthesized braids} , sort={S}}

\newglossaryentry{PaBe}{type=operads, name=$\PaB_{e\ell\ell}$, 
  description={$\PaB$-module of elliptic parenthesized braids} , sort={S}}
  
\newglossaryentry{PaB1}{type=operads, name=$\PaB^{1}$, 
  description={$\PaB$-moperad of parenthesized braids with a frozen strand}, sort={S}}  
  
\newglossaryentry{PaBg}{type=operads, name=$\PaB^{\Gamma}$, 
  description={$\PaB$-moperad of cyclotomic parenthesized braids}, sort={S}}
  
\newglossaryentry{PaBeg}{type=operads, name=$\PaB_{e\ell\ell}^{\Gamma}$, 
  description={$\PaB$-module of ellipsitomic parenthesized braids}, sort={S}}
  
\newglossaryentry{PaBf}{type=operads, name=$\PaB^f$, 
  description={Operad of framed parenthesized braids}, sort={S}}

\newglossaryentry{PaBfg}{type=operads, name=$\PaB^f_g$, 
  description={$\PaB^f$-module of framed genus $g$ parenthesized braids}, sort={S}}

\newglossaryentry{PaCD}{type=operads, name=$\PaCD(\KK)$, 
  description={Operad of parenthesized chord diagrams}, sort={S}}

\newglossaryentry{PaCDe}{type=operads, name=$\PaCD_{e\ell\ell}(\KK)$, 
  description={$\PaCD(\KK)$-module of ellitpic parenthesized chord diagrams}, sort={S}}
  
\newglossaryentry{PaCDg}{type=operads, name=$\PaCD^{\Gamma}(\KK)$, 
  description={$\PaCD(\KK)$-moperad of cyclotomic parenthesized chord diagrams}, sort={S}}
  
\newglossaryentry{PaCDeg}{type=operads, name=$\PaCD_{e\ell\ell}^{\Gamma}(\KK)$, 
  description={$\PaCD(\KK)$-module of ellipsitomic parenthesized chord diagrams}, sort={S}}
  
\newglossaryentry{PaCDf}{type=operads, name=$\PaCD^f(\KK)$, 
  description={Operad of framed parenthesized chord diagrams}, sort={S}}

\newglossaryentry{PaCDfg}{type=operads, name=$\PaCD^f_g(\KK)$, 
  description={$\PaCD^f(\KK)$-module of framed genus $g$ parenthesized chord diagrams}, sort={S}}
  
\newglossaryentry{GT}{type=groups, name=$\GT$, 
  description={Operadic Grothendieck-Teichm\"uller group}, sort={S}}

\newglossaryentry{kGTe}{type=groups, name=$\widehat{\GT}_{e\ell\ell}(\kk)$, 
  description={Operadic $\kk$-pro-unipotent elliptic Grothendieck-Teichm\"uller group}, sort={S}}
  
\newglossaryentry{kGTg}{type=groups, name=$\widehat{\GT}^{\Gamma}(\kk)$, 
  description={Operadic $\kk$-pro-unipotent cyclotomic Grothendieck-Teichm\"uller group}, sort={S}}
  
\newglossaryentry{kGTeg}{type=groups, name=$\widehat{\GT}_{e\ell\ell}^{\Gamma}(\kk)$, 
  description={Operadic $\kk$-pro-unipotent ellipsitomic Grothendieck-Teichm\"uller group}, sort={S}}
  
\newglossaryentry{kGTf}{type=groups, name=$\widehat{\GT}^f(\kk)$, 
  description={Operadic $\kk$-pro-unipotent framed Grothendieck-Teichm\"uller group}, sort={S}}

\newglossaryentry{kGTfg}{type=groups, name=$\widehat{\GT}^f_g(\kk)$, 
  description={Operadic $\kk$-pro-unipotent genus $g$ Grothendieck-Teichm\"uller group}, sort={S}}
  
  \newglossaryentry{GRT}{type=groups, name=$\GRT(\KK)$, 
  description={Operadic graded Grothendieck-Teichm\"uller group}, sort={S}}

\newglossaryentry{GRTe}{type=groups, name=$\GRT_{e\ell\ell}(\KK)$, 
  description={Operadic graded elliptic Grothendieck-Teichm\"uller group}, sort={S}}
  
\newglossaryentry{GRTg}{type=groups, name=$\GRT^{\Gamma}(\KK)$, 
  description={Operadic graded cyclotomic Grothendieck-Teichm\"uller group}, sort={S}}
  
\newglossaryentry{GRTeg}{type=groups, name=$\GRT_{e\ell\ell}^{\Gamma}(\KK)$, 
  description={Operadic graded ellipsitomic Grothendieck-Teichm\"uller group}, sort={S}}
  
\newglossaryentry{GRTf}{type=groups, name=$\GRT^f(\KK)$, 
  description={Operadic graded framed Grothendieck-Teichm\"uller group}, sort={S}}

\newglossaryentry{GRTfg}{type=groups, name=$\GRT^f_g(\KK)$, 
  description={Operadic graded genus $g$ Grothendieck-Teichm\"uller group}, sort={S}}
  
\newglossaryentry{bGT}{type=groups, name=$\widehat{\on{GT}}(\kk)$, 
  description={$\kk$-pro-unipoten Grothendieck-Teichm\"uller group}, sort={S}}

\newglossaryentry{bkGTe}{type=groups, name=$\widehat{\on{GT}}_{e\ell\ell}(\kk)$, 
  description={$\kk$-pro-unipotent elliptic Grothendieck-Teichm\"uller group}, sort={S}}
  
\newglossaryentry{bkGTg}{type=groups, name=$\widehat{\on{GT}}^{\Gamma}(\kk)$, 
  description={$\kk$-pro-unipotent cyclotomic Grothendieck-Teichm\"uller group}, sort={S}}
  
\newglossaryentry{bkGTeg}{type=groups, name=$\widehat{\on{GT}}_{e\ell\ell}^{\Gamma}(\kk)$, 
  description={$\kk$-pro-unipotent ellipsitomic Grothendieck-Teichm\"uller group}, sort={S}}
  
\newglossaryentry{bkGTf}{type=groups, name=$\widehat{\on{GT}}^f(\kk)$, 
  description={$\kk$-pro-unipotent framed Grothendieck-Teichm\"uller group}, sort={S}}

\newglossaryentry{bkGTfg}{type=groups, name=$\widehat{\on{GT}}^f_g(\kk)$, 
  description={$\kk$-pro-unipotent genus $g$ Grothendieck-Teichm\"uller group}, sort={S}}
  
\newglossaryentry{bGRT}{type=groups, name=$\on{GRT}(\KK)$, 
  description={Graded Grothendieck-Teichm\"uler group}, sort={S}}

\newglossaryentry{bGRTe}{type=groups, name=$\on{GRT}_{e\ell\ell}(\KK)$, 
  description={Graded elliptic Grothendieck-Teichm\"uller group}, sort={S}}
  
\newglossaryentry{bGRTg}{type=groups, name=$\on{GRT}^{\Gamma}(\KK)$, 
  description={Graded cyclotomic Grothendieck-Teichm\"uller group}, sort={S}}
  
\newglossaryentry{bGRTeg}{type=groups, name=$\on{GRT}_{e\ell\ell}^{\Gamma}(\KK)$, 
  description={Graded ellipsitomic Grothendieck-Teichm\"uller group}, sort={S}}
  
\newglossaryentry{bGRTfg}{type=groups, name=$\on{GRT}^f_g(\KK)$, 
  description={Graded genus $g$ Grothendieck-Teichm\"uller group}, sort={S}}
  
\newglossaryentry{ConfCn}{type=spaces, name=\ensuremath{\on{Conf}(\C,n)}, 
  description={Configuration space of $n$ points in $\C$}, sort={S}}
  
\newglossaryentry{CCn}{type=spaces, name=\ensuremath{\on{C}(\C,I)}, 
  description={Reduced configuration space of $I$-indexed points in $\C$}, sort={S}}
  
\newglossaryentry{FMCn}{type=spaces, name=\ensuremath{\bar{\on{C}}(\C,I)}, 
  description={Fulton-MacPherson compactification of \ensuremath{\on{C}(\C,I)}}, sort={S}}

\newglossaryentry{ConfCn1}{type=spaces, name=\ensuremath{\on{Conf}(\C^{\times},n)}, 
  description={Configuration space of $n$ points in $\C^{\times}$}, sort={S}}
  
  \newglossaryentry{CCI1}{type=spaces, name=\ensuremath{\on{C}(\C^{\times},I)}, 
  description={Reduced configuration space of $I$-indexed points in $\C^{\times}$}, sort={S}}

\newglossaryentry{ConfCnM}{type=spaces, name=\ensuremath{\on{Conf}(\C^{\times},n,M)}, 
  description={$M$-decorated configuration space of $n$ points in $\C^{\times}$}, sort={S}}
  
  \newglossaryentry{CCIM}{type=spaces, name=\ensuremath{\on{C}(\C^{\times},I,\Gamma)}, 
  description={Reduced $\Gamma$-decorated configuration space of $I$-indexed points in $\C^{\times}$}, sort={S}}
  
  \newglossaryentry{ConfTn}{type=spaces, name=\ensuremath{\on{Conf}(E_\tau,n)}, 
  description={Configuration space of $n$ points in \ensuremath{E_\tau}}, sort={S}}

  \newglossaryentry{CTn}{type=spaces, name=\ensuremath{\on{C}(\mathbb{T},I)}, 
  description={Reduced configuration space of $I$-indexed points in \ensuremath{\mathbb{T}}}, sort={S}}

\newglossaryentry{FMTn}{type=spaces, name=\ensuremath{\bar{\on{C}}(\mathbb{T},I)}, 
  description={Fulton-MacPherson compactification of \ensuremath{\on{C}(\mathbb{T},I)}}, sort={S}}

  \newglossaryentry{ConfTnG}{type=spaces, name=\ensuremath{\on{Conf}(E,n,\Gamma)}, 
  description={$\Gamma$-decorated configuration space of $n$ points in \ensuremath{E}}, sort={S}}
  
    \newglossaryentry{CTnG}{type=spaces, name=\ensuremath{\on{C}(E,n,\Gamma)}, 
  description={Reduced $\Gamma$-decorated configuration space of $n$ points in \ensuremath{E}}, sort={S}}
  
  \newglossaryentry{FMTIG}{type=spaces, name=\ensuremath{\bar{\on{C}}(\mathbb{T},I,\Gamma)}, 
  description={Fulton-MacPherson compactification of \ensuremath{\on{C}(\mathbb{T},I,\Gamma)}}, sort={S}}
  
\newglossaryentry{PBn}{type=groups, name=$\on{PB}_n$, 
  description={Pure braid group on the complex plane}, sort={S}}
  
\newglossaryentry{PB1n}{type=groups, name=$\on{PB}_{1,n}$, 
  description={Pure braid group on the torus}, sort={S}}
  
\newglossaryentry{B1n}{type=groups, name=$\on{B}_{1,n}$, 
  description={Braid group on the torus}, sort={S}}  
  
\newglossaryentry{PBnM}{type=groups, name=$\on{PB}_n^M$, 
  description={$M$-decorated pure braid group on the cylinder}, sort={S}}
  
\newglossaryentry{PB1nG}{type=groups, name=$\on{PB}_{1,n}^\Gamma$, 
  description={$\Gamma$-decorated pure braid group on the torus}, sort={S}}

\newglossaryentry{tn}{type=liealgebras, name=$\t_n(\C)$, 
  description={Rational Kohno-Drinfeld Lie $\C$-algebra}, sort={S}}
  
\newglossaryentry{t1n}{type=liealgebras, name=$\t_{1,n}$, 
  description={Elliptic Kohno-Drinfeld Lie $\C$-algebra}, sort={S}}
  
\newglossaryentry{tnM}{type=liealgebras, name=$\t_n^M$, 
  description={$M$-cyclotomic Kohno-Drinfeld Lie $\C$-algebra}, sort={S}}
  
\newglossaryentry{t1IG}{type=liealgebras, name=$\t_{1,n}^\Gamma(\kk)$, 
  description={$\Gamma$-ellipsitomic Kohno-Drinfeld Lie $\kk$-algebra}, sort={S}}
  
\newglossaryentry{Ptng}{type=bundles, name=$\mathcal{P}_{\tau,n,\Gamma}$, 
  description={Principal ${\rm exp}(\hat{\t}_{1,n}^\Gamma)$-bundle over $\on{Conf}(E,n,\Gamma)$}, sort={S}}  
  
\newglossaryentry{Ptnng}{type=bundles, name=$\mathcal{P}_{\tau,[n],\Gamma}$, 
  description={Principal ${\rm exp}(\hat{\bar{\t}}_{1,n}^\Gamma)$-bundle over $\on{Conf}(E,[n],\Gamma)$}, sort={S}}  

\newglossaryentry{bPtnng}{type=bundles, name=$\bar{\mathcal{P}}_{\tau,[n],\Gamma}$, 
  description={Principal $\exp(\hat{\t}_{1,n}^\Gamma)\rtimes \mathfrak{S}_n$-bundle over $\on{C}(E,[n],\Gamma)$}, sort={S}}  

\newglossaryentry{Ptn}{type=bundles, name=$\bar{\mathcal{P}}_{(\tau,\Gamma),n}$, 
  description={Principal $\exp(\hat{\t}_{1,n}^\Gamma)\rtimes \Gamma^n$-bundle over $\on{Conf}(E,n)$}, sort={S}}  

\newglossaryentry{Gng}{type=groups, name=$\mathbf{G}_n^\Gamma$, 
  description={Structure group of the principal bundle over $\M_{1,n}^\Gamma$}, sort={S}}  

\newglossaryentry{bGng}{type=groups, name=$\bar{\mathbf{G}}_n^\Gamma$, 
  description={Structure group of the principal bundle over $\bar{\M}_{1,n}^\Gamma$}, sort={S}}  

\newglossaryentry{dtg}{type=liealgebras, name=$\tilde\d^\Gamma$, 
  description={Intermediate twisted derivations Lie algebra}, sort={S}}  

\newglossaryentry{dg}{type=liealgebras, name=$\d^\Gamma$, 
  description={Twisted derivations Lie algebra}, sort={S}}  

\newglossaryentry{SL2g}{type=groups, name=$\on{SL}^\Gamma_2(\Z)$, 
  description={$\Gamma$-level principal congruence subgroup of $\on{SL}_2(\Z)$}, sort={S}}  

\newglossaryentry{bM1ng}{type=spaces, name=$\bar\cM_{1,n}^\Gamma$, 
  description={Reduced moduli space of $\Gamma$-structured $n$-marked elliptic curves}, sort={S}}  

\newglossaryentry{M1ng}{type=spaces, name=$\cM_{1,n}^\Gamma$, 
  description={Non-reduced moduli space of $\Gamma$-structured $n$-marked elliptic curves}, sort={S}}  

\newglossaryentry{M1nng}{type=spaces, name=$\M_{1,[n]}^\Gamma$, 
  description={Non-reduced moduli space of $\Gamma$-structured unorderly $n$-marked elliptic curves}, sort={S}}    
  
\newglossaryentry{bM1nng}{type=spaces, name=$\bar\cM_{1,[n]}^\Gamma$, 
  description={Reduced moduli space of $\Gamma$-structured unorderly $n$-marked elliptic curves}, sort={S}}    

\newglossaryentry{Png}{type=bundles, name=$\mathcal{P}_{\n,\Gamma}$, 
  description={Principal $\mathbf{G}_n^\Gamma$-bundle over $\M_{1,n}^\Gamma$}, sort={S}}  
  
\newglossaryentry{Pnng}{type=bundles, name=$\mathcal{P}_{[n],\Gamma}$, 
  description={Principal $\mathbf{G}_{[n]}^\Gamma$-bundle over $\M_{1,[n]}^\Gamma$}, sort={S}}  

\newglossaryentry{bPng}{type=bundles, name=$\bar{\mathcal{P}}_{n,\Gamma}$, 
  description={Principal $\bar{\mathbf{G}}_n^\Gamma$-bundle over $\bar{\M}_{1,n}^\Gamma$}, sort={S}}  

\newglossaryentry{Pn}{type=bundles, name=$\mathcal{P}_{(\Gamma),n}$, 
  description={Principal ${\mathbf G}^\Gamma_{n}\rtimes \Gamma^n$-bundle over ${\mathcal M}^\Gamma_{1,n}/\Gamma^n$}, sort={S}}  

  \newglossaryentry{Hgl}{type=liealgebras, name=\ensuremath{H_{n}(\g,\l^*)}, 
  description={Hecke algebra of the pair $(\g,\l)$}, sort={S}  }

  \newglossaryentry{Hgh}{type=liealgebras, name=\ensuremath{H_{n}(\g,\h^*_{reg})}, 
  description={Reduced Hecke algebra of the pair $(\g,\h)$}, sort={S}  }

\title{On the universal ellipsitomic KZB connection}
\author{Damien Calaque and Martin Gonzalez}
\address{Damien CALAQUE \newline \indent IMAG, Univ Montpellier, CNRS, Montpellier, France \& Institut Universitaire de France}
\email{damien.calaque@umontpellier.fr}
\address{Martin GONZALEZ \newline \indent Max-Planck Institut f\"ur Mathematik, Vivatsgasse 7, Bonn, Germany}
\email{marting@mpim-bonn.mpg.de}

\begin{document}

\begin{abstract} 
We construct a twisted version of the genus one universal\break Knizhnik--Zamolodchikov--Bernard 
(KZB) connection introduced by Calaque--Enriquez--Etingof, that we call the \textit{ellipsitomic} 
KZB connection. 
This is a flat connection on a principal bundle over the moduli space of $\Gamma$-structured 
elliptic curves with marked points, where $\Gamma=\Z/M\Z\times\Z/N\Z$, and $M,N\geq1$ are two integers. 
It restricts to a flat connection on $\Gamma$-twisted configuration spaces of points on elliptic 
curves, which can be used to construct a filtered-formality isomorphism for some interesting subgroups 
of the pure braid group on the torus. 
We show that the universal ellipsitomic KZB connection realizes as the usual KZB connection 
associated with elliptic dynamical $r$-matrices with spectral parameter, and finally, also 
produces representations of cyclotomic Cherednik algebras.
\end{abstract}

\maketitle

\setcounter{tocdepth}{1}

\tableofcontents


\section*{Introduction}


In this paper, which fits in a series of works about universal Knizhnik--Zamolodchikov--Bernard (KZB) 
connections by different authors \cite{CEE,En4}, we focus on a twisted version of the genus 1 situation. 
In his seminal work \cite{DrGal}, Drinfeld considers the monodromy representation of the universal 
Knizhnik--Zamolodchikov (KZ) equation which leads to the formality of the pure braid group (see reminder below) 
and the so-called theory of associators that makes the link between rich algebraic structures 
(such as braided monoidal categories) and the Grothendieck--Teichm\"uller group $\on{GT}$. 

Enriquez generalizes in \cite{En} Drinfeld's work to the twisted (a-k-a trigonometric, or cyclotomic) situation 
and relates it to multiple polylogarithms at roots of unity. Namely, he uses the universal trigonometric 
KZ system to prove the formality of some subgroups of the pure braid group on $\C^\times$ and to move 
relations between suitable algebraic structures (quasi-reflection algebras, or braided module categories)  
and analogues of the group $\on{GT}$. 

The next step has been made by Enriquez, Etingof and the first author in \cite{CEE}, where a universal 
version of the elliptic KZB system (see \cite{B1}) is defined and used to: 
\begin{itemize}
\item give a new proof (see \cite{BZ} for the original one) of the filtered formality of the pure braid group 
on the torus,
\item find a relation between the KZ associator and a generating series for 
iterated integrals of Eisenstein series (see also \cite{En3}),
\item provide examples of elliptic structures on braided monoidal categories (see also \cite{En2}). 
\end{itemize}

\medskip

The main goal of the present paper is to introduce a twisted version of the universal elliptic 
KZB system, called the \textit{ellipsitomic} KZB connection, and to derive from it the formality of 
some subgroups of the pure braid group on the torus. In a subsequent work \cite{CG-toappear}, we use it 
to emphasize a relation between generating series for values of multiple polylogarithms at roots of unity 
and values of elliptic multiple polylogarithms at torsion points. 

\medskip

Throughout the paper, $\kk$ is a field of characteristic zero, $M,N$ are 
fixed positive integers, and $\Gamma:=\Z/M\Z\times\Z/N\Z$.

\medskip	

\noindent\textbf{Genus zero situation (rational KZ).} First recall from \cite{Kohno} that the holonomy Lie algebra of 
the configuration space 
\[
\text{\gls{ConfCn}}:=\{\zz=(z_1,\dots,z_n)\in \C^n|z_i\neq z_j\textrm{ if }i\neq j\}
\]
of $n$ points on the complex line is isomorphic to the graded Lie $\mathbb{C}$-algebra $\t_n$ 
generated by $t_{ij}$, $1\leq i\neq j\leq n$, with relations
\begin{flalign}
& t_{ij}=t_{ji}\,,                			                 			\tag{S} \label{eqn:S} \\
& [t_{ij},t_{kl}]=0\quad\textrm{if }\#\{i,j,k,l\}=4 \,,				\tag{L} \label{eqn:L} \\
& [t_{ij},t_{ik}+t_{jk}]=0\quad\textrm{if }\#\{i,j,k\}=3\,. &	\tag{4T} \label{eqn:4T} 
\end{flalign}
Then, on the one hand, denote by $\text{\gls{PBn}}$ the fundamental group of $\on{Conf}(\C,n)$, also known as the pure 
braid group with $n$ strands, and by $\pb_n$ its Malcev Lie algebra (which is filtered by its lower central 
series, and complete). 
One can easily check that $\PB_n$ is generated by elementary pure braids $P_{ij}$, $1\leq i<j\leq n$, 
which satisfy the following relations: 
\begin{flalign}
& (P_{ij},P_{kl})=1\quad\textrm{if }\{i,j\}\textrm{ and }\{k,l\}~\textrm{are non crossing}\,,  	\tag{PB1}\label{eqn:PB1} \\
& (P_{kj}P_{ij}P_{kj}^{-1},P_{kl})=1\quad\textrm{if }i<k<j<l\,, 																\tag{PB2}\label{eqn:PB2} \\
& (P_{ij},P_{ik}P_{jk})=(P_{jk},P_{ij}P_{ik})=(P_{ik},P_{jk}P_{ij})=1\quad\textrm{if }i<j<k\,. &\tag{PB3}\label{eqn:PB3}  
\end{flalign}
We can depict the generator $P_{ij}$ in the following two equivalent ways: 
\begin{center}
\begin{tikzpicture}[baseline=(current bounding box.center)]
\tikzstyle point=[circle, fill=black, inner sep=0.05cm]
 \node[point, label=above:$1$] at (0,1) {};
 \node[point, label=below:$1$] at (0,-1.5) {};
 \draw[->,thick, postaction={decorate}] (0,1) -- (0,-1.45);
 \node[point, label=above:$i$] at (1,1) {};
 \node[point, label=below:$i$] at (1,-1.5) {};
  \node[point, label=above:$...$] at (2,1) {};
 \node[point, label=below:$...$] at (2,-1.5) {};
 \draw[thick] (1,1) .. controls (1,0.25) and (3.5,0.5).. (3.5,-0.25);
 \node[point, ,white] at (3,0.2) {};
 \node[point, label=above:$j$] at (3,1) {};
 \node[point, label=below:$j$] at (3,-1.5) {};
   \draw[->,thick, postaction={decorate}] (3,1) -- (3,-1.45);

   \draw[->,thick, postaction={decorate}] (3,0) -- (3,-1.45);
  \node[point, ,white] at (3,-0.7) {};
 \draw[->,thick, postaction={decorate}] (3.5,-0.25) .. controls (3.5,-1) and (1,-0.75).. (1,-1.45);


 \node[point, label=above:$n$] at (4,1) {};
 \node[point, label=below:$n$] at (4,-1.5) {};
 \draw[->,thick, postaction={decorate}] (4,1) -- (4,-1.45);
   \node[point, ,white] at (2,0.4) {};
    \node[point, ,white] at (2,-0.9) {};
   \draw[->,thick, postaction={decorate}] (2,1) -- (2,-1.45);
\end{tikzpicture}
$\qquad\longleftrightarrow\quad$
\begin{tikzpicture}[baseline=(current bounding box.center)] 
\tikzstyle point=[circle, fill=black, inner sep=0.05cm]
 \draw[loosely dotted] (1.5,-1.5) -- (-1.5,1.5); 
 \node[point, label=above:$n$] at (1.5,-1.5) {}; 
  \node[point, label=left:$i$] at (-0.5,0.5) {};
 \node[point, label=below:$j$] at (0.5,-0.5) {};
 \node[point, label=below:$1$] at (-1.5,1.5) {};
 \draw[->, thick, postaction={decorate}] (-0.5,0.5) .. controls (-0.5,0.5) and (0.65,0.7).. (0.65,-0.5); 
 \draw[thick] (0.65,-0.49) .. controls (0.5,-0.7)  .. (0.4,-0.5) ;
 \draw[thick, postaction={decorate}] (0.4,-0.5) .. controls (0.5,0.25) .. (-0.5,0.5);
\end{tikzpicture}
\end{center}
Therefore one has a surjective morphism of graded Lie algebras $p_n:\t_n\twoheadrightarrow{\rm gr}(\pb_n)$ 
sending $t_{ij}$ to $\sigma({\rm log}(P_{ij}))$, $i<j$ and $\sigma:\pb_n\to{\rm gr}(\pb_n)$ being the 
symbol map. 

On the other hand, denote by $\exp(\hat{\mathfrak{t}}_n)$ the exponential group associated with the 
degree completion $\hat{\mathfrak{t}}_n$ of $\t_n$.
The universal KZ connection on the trivial $\exp(\hat{\mathfrak{t}}_n)$-principal bundle over 
$\tmop{Conf} (\C, n)$ is then given by the holomorphic 1-form
\[
		w^{\tmop{KZ}}_n : = \underset{1 \leqslant i < j \leqslant n}{\sum}
		\frac{d z_i - d z_j}{z_i - z_j} t_{i j} \in \Omega^1 (\tmop{Conf}
		(\C, n), \mathfrak{t}_n)\,, 
\]
which takes its values in $\mathfrak{t}_n$. It is a fact that the connection
associated with this 1-form is flat, and descends to a flat connection on the moduli space
$\mathcal{M}_{0, n + 1} \simeq \tmop{Conf} (\C, n) / \tmop{Aff}
(\C)$ of rational curves with $n+1$ marked points.

Firstly, the regularized holonomy of this connection along the real straight path from 0 to 1 in 
$\mathcal{M}_{0, 4} \simeq \mathbb{P}^1-\{0,1,\infty\}$ gives a formal power series $\Phi_{\on{KZ}}$ 
in two non-commuting variables, called the KZ associator, that is a generating series for values 
at $0$ and $1$ of multiple polylogarithms.
Secondly, using the monodromy representation of the universal KZ connection, one obtains:
\begin{enumerate}
\item A morphism of filtered Lie algebras $\mu_n:\pb_n\to\hat{\t}_n$ such that 
${\rm gr}(\mu_n)\circ p_n={\rm id}$. Hence one 
concludes that $p_n$ and $\mu_n$ are bijective. 
This provides a filtered isomorphism from $\pb_n$ to the degree completion of its associated graded, 
which is actually $\hat\t_n$. This recovers the known fact that the group $\PB_n$ is \textit{$1$-formal}, 
meaning that its Malcev Lie algebra is isomorphic to the degree completion of a quadratic Lie algebra. 
\item A system of relations (called Pentagon ($P$) and two Hexagons ($H_{\pm}$)) satisfied by the KZ associator. 
Then, if $\kk$ is a field of characteristic 0, one can define a set of $\kk$-associators $\on{Ass}(\kk)$, 
for which the KZ associator will be a $\C$-point (showing at the same time that the set of such abstract 
$\C$-associators is indeed non-empty).  
\end{enumerate}

\medskip	
\noindent\textbf{A twisted variant (trigonometric/cyclotomic KZ).} Similarly, one can consider the configuration space 
$$ 
\text{\gls{ConfCn1}}:=\{\zz=(z_1,\dots,z_n)\in(\C^\times)^n|z_i\neq z_j\textrm{ if }i\neq j\}
$$
of $n$ points on $\C^\times$. Then $\on{Conf}(\C^\times,n)\simeq\on{Conf}(\C,n+1)/\C$ and thus its 
fundamental group $\PB_n^{1}$ is isomorphic to $\PB_{n+1}$. 
More generally, for any $M\in\Z-\{0\}$ one can consider an $M$-twisted configuration space 
$$
\text{\gls{ConfCnM}}:=\{\zz=(z_1,\dots,z_n)\in (\C^\times)^n|z_i^M\neq z_j^M\textrm{ for }i\neq j\}.
$$ 
In \cite{En} Enriquez exhibits, using the so-called universal trigonometric KZ connection, 
an isomorphism $\pb_n^{M}\to\exp(\hat{\t}_n^{M})$, where $\pb_n^{M}$ is the Malcev Lie algebra 
of the fundamental group $\text{\gls{PBnM}}\subset \PB_n^{1}$ of $\on{Conf}(\C^\times,n,M)$, and 
\gls{tnM} is the holonomy Lie algebra of $\on{Conf}(\C^\times,n,M)$. 
The monodromy of this connection along a suitable (non closed) path gives a universal pseudotwist 
$\Psi^{M}_{\on{KZ}}\in\exp(\bar\t_2^{M})$ that is a generating series for values of multiple 
polylogarithms at $M$th roots of unity, and satisfies relations with $\Phi_{\on{KZ}}$. 

\medskip	
\noindent\textbf{Genus one situation (elliptic KZB).} The genus one universal 
Knizhnik--Zamolodchikov--Bernard (KZB) connection $\nabla^{\on{KZB}}_{1,n}$ was introduced in 
\cite{CEE}. This is a flat connection over the moduli space of elliptic curves with $n$ marked points 
$\mathcal{M}_{1,n}$, which was independently discovered by Levin--Racinet \cite{LR} in the 
specific cases $n=1,2$. 
It restricts to a flat connection over the configuration space 
$$
\text{\gls{ConfTn}}:= \Lambda_\tau^n \backslash \{\zz=(z_1,\dots,z_n)\in \C^n | z_i - z_j \notin \Lambda_\tau \text{ if } i \neq j\}
$$
of $n$ points on an (uniformized) elliptic curve $E_{\tau}:=\Lambda_\tau \backslash \C$, for $\tau\in\h$ 
and $\Lambda_\tau=\Z + \tau \Z$.  More precisely, this connection is defined on a $G$-principal 
bundle over $\mathcal{M}_{1,n}$ where the Lie algebra associated with $G$ has as components:
\begin{enumerate}
\item a Lie algebra \gls{t1n} related to $\on{Conf}(\mathbb{T},n)$, somehow controlling 
the variations of the marked points: it has generators $x_i,y_i$, for $i=1,...,n$, corresponding 
to moving $z_i$ along the topological cycles generating $H_1(E_{\tau})$;
\item a Lie algebra $\mathfrak{d}$ with as components the Lie algebra $\mathfrak{sl}_2$ with standard 
generators $e,f,h$ and a Lie algebra $\mathfrak{d}_+:=\on{Lie}(\{\delta_{2m}| m\geq 1 \})$ such 
that each $\delta_{2m}$ is a highest weight element for $\mathfrak{sl}_2$. The Lie algebra 
$\mathfrak{d}$ somehow controls the variation of the curve in $\mathcal{M}_{1,n}$ and is closely related to the one defined in \cite{Ts}.
\end{enumerate}
Now, the connection $\nabla^{\on{KZB}}_{1,n}$ can be locally expressed as 
$\nabla^{\on{KZB}}_{1,n}:=d-\Delta(\zz|\tau)d\tau-\sum_iK_i(\zz|\tau)dz_i$ where 
\begin{enumerate}
\item the term $K_{i}(-|\tau) : \C^{n} \to \hat\t_{1,n}$ is meromorphic on $\C^n$, having only simple poles on
$$
\on{Diag}_{n,\tau}:=\bigcup_{i\neq j}\{\zz=(z_1,\dots,z_n)\in \C^n | z_i-z_j\in\Lambda_\tau\}\,.
$$
It is constructed out of a function
$$
k(x,z|\tau) := \frac{\theta(z+x|\tau)}{\theta(z|\tau)\theta(x|\tau)} - {\frac{1}{x}}. 
$$
This relates directly the connection $\nabla^{\on{KZB}}_{1,n}$ with Zagier's work \cite{Zag} on 
Jacobi forms (see Weil's book \cite{Weil}) and to Brown and Levin's work \cite{BL}.
\item the term $\Delta(\zz|\tau)$ is a meromorphic function $\C^{n}\times\h\to \on{Lie}(G)$, 
with only simple poles on $\on{Diag}_{n}:=\{(\zz,\tau)\in\C^{n}\times \h| \zz\in\on{Diag_{n,\tau}}\}$. 
The coefficients of $\delta_{2m}$ in $\Delta(\zz|\tau)$ are Eisenstein series.
\end{enumerate}
We also refer to Hain's survey \cite{Hain} and references therein for the Hodge theoretic and motivic aspects of the story. 

Then, one can construct a holomorphic map sending each $\tau\in\h$ to a couple 
$e(\tau):=(A(\tau),B(\tau))$ where $A(\tau)$ (resp. $B(\tau)$) is the regularized holonomy 
of the universal elliptic KZB connection along the straight path from 0 to 1 (resp. from 0 to 
$\tau$) in the once punctured elliptic curve $\Lambda_\tau \backslash (\C-\Lambda_\tau) \simeq 
E_{\tau} \backslash \on{Conf}(E_{\tau},2)$. Enriquez developed in \cite{En2} the general theory 
of elliptic associators, whose scheme is denoted $\on{Ell}$ and for which the couple $e(\tau)$ 
is an example of a $\C$-point. Some of the main features of the so-called elliptic KZB associators $e(\tau)$ are the following:
\begin{itemize}
\item They satisfy algebraic and modularity relations.
\item They satisfy a differential equation in the variable $\tau$ expressed only in terms of iterated 
integrals of Eisenstein series, which will be called iterated Eisenstein integrals.
\item When taking $\tau$ to $i\infty$ (which consists in computing the constant term of the 
$q$-expansion of the series $A(\tau)$ and $B(\tau)$, where $q=e^{2i\pi\tau}$), they can be 
expressed only in terms of the KZ associator $\Phi_{\on{KZ}}$.
\item They provide isomorphisms between the Malcev Lie algebra of the fundamental group 
$\PB_{1,n}$ of $\on{Conf}(\mathbb{T},n)$ and the degree completion of its associated Lie algebra $\t_{1,n}$. 
\end{itemize}
Observe that, contrary to what happens in genus $0$, $\PB_{1,n}$ (also known as the 
\textit{pure elliptic braid group}) is not $1$-formal (as $\t_{1,n}$ is not quadratic), but 
only \textit{filtered-formal} according to the terminology of \cite{SW}. 

\medskip
\noindent\textbf{Ellipsitomic KZB.} As we wrote above, the purpose of the present work is to define a twisted version 
of the genus one KZB connection introduced in \cite{CEE}. This is a flat connection on a principal bundle over the moduli space of elliptic curves with a $\Gamma$-structure and $n$ marked points. It restricts 
to a flat connection on the so-called $\Gamma$-twisted configuration space of points on an elliptic curve, which can be 
used for constructing a filtered-formality isomorphism for some interesting subgroups of the pure braid group on the torus. 

In a subsequent work \cite{CG-toappear}, we will define ellipsitomic KZB associators as renormalized holonomies 
along certain paths on a once punctured elliptic curve with a $\Gamma$-structure, and exhibit a relation between 
ellipsitomic KZB associators, the KZ associator \cite{DrGal} and the cyclotomic KZ associator \cite{En}. Moreover, 
ellipsitomic associators can be regarded as a generating series for iterated Eisenstein integrals whose coefficients 
are elliptic multiple zeta values at torsion points. In the case $M=N$ these coefficients are related to Goncharov's 
work \cite{Go}, and also to the recent work \cite{M2} of Broedel--Matthes--Richter--Schlotterer. 

We finally prove that the universal ellipsitomic KZB connection realizes as the usual KZB connection associated with 
certain elliptic dynamical $r$-matrices with spectral parameter, that should be compared with \cite{ES,FICM}. 

It is worth mentioning the recent work \cite{Tol}, where Toledano-Laredo and Yang define a similar KZB 
connection. More precisely, they construct a flat KZB connection on moduli spaces of elliptic curves 
associated with crystallographic root systems. The type $A$ case coincides with the universal elliptic 
KZB connection defined in \cite{CEE}, and we suspect that the type $B$ case coincides with the 
connection of the present paper for $M=N=2$. It is interesting to point out that a common generalization 
of their work and ours (for $M=N$) could be obtained by constructing a universal KZB connection 
associated with arbitrary complex reflection groups. 

\medskip
\noindent\textbf{Plan of the paper.} The paper is organized as follows: 
\begin{itemize}
\item In Section \ref{Bundles with flat connections on twisted configuration spaces}, we introduce $\Gamma$-twisted configuration spaces on an elliptic curve and define the universal ellipsitomic KZB connection on them. 
It takes values in a the Lie algebra $\t_{1,n}^\Gamma$ of infinitesimal ellipsitomic (pure) braids, that we also define. 
\item As in \cite{CEE}, the connection extends from the configuration space to the moduli space 
$\bar\cM_{1,[n]}^\Gamma$ of elliptic curves with a $\Gamma$-level structure and unordered marked points. 
This is proven in Section \ref{Bundles with flat connections on moduli spaces} using some technical definitions introduced in Section \ref{Lie algebras of derivations and associated groups}, involving 
derivations of the Lie algebra $\t_{1,n}^\Gamma$ related to the twisted configuration space in genus $1$. 
As in the untwisted case, the results of this section also apply to the ``unordered marked points'' situation as well.
\item In Section \ref{Realizations}, we provide a notion of realizations for the Lie algebras previously introduced, 
and show that the universal ellipsitomic KZB connection realizes to a flat connection intimately related to 
elliptic dynamical $r$-matrices with spectral parameter.
\item In Section \ref{Formality of subgroups of the pure braid group on the torus}, we derive from the monodromy representation the filtered-formality of the fundamental group of the twisted configuration space of the torus, 
which is a subgroup of $\PB_{1,n}$. 
As in the cyclotomic case, it extends to a relative filtered-formality result for the map 
$\on{B}_{1,n} \to \Gamma^n \rtimes \mathfrak{S}_n$. 
\item Finally, in Section \ref{Representations of Cherednik algebras}, we construct a homomorphism from the Lie algebra 
$\bar {\mathfrak {t}}^\Gamma_{1,n}\rtimes {\mathfrak {d}^\Gamma}$ to the twisted Cherednik algebra $H^\Gamma_n(k)$. 
This allows us to consider the twisted elliptic KZB connection with values in representations of 
the twisted Cherednik algebra. This study shall be closely related to the recent paper \cite{BE}.
\item We also include an appendix that summarizes our conventions for fundamental groups, covering maps, 
principal bundles, and monodromy maps. 
\end{itemize}

\medskip

\noindent\textbf{Acknowledgements.} Both authors are grateful to Benjamin Enriquez, 
Richard Hain and Pierre Lochak for numerous conversations and suggestions, as well as 
Adrien Brochier, Richard Hain, and Eric Hopper for their very valuable comments on an 
earlier version, which helped us correcting inaccuracies, and improved the exposition. 
We also thank Nils Matthes for discussions about twisted elliptic MZVs. 
The first author acknowledges the financial support of the \textit{ANR project SAT} and of 
the \textit{Institut Universitaire de France}. 
This paper is extracted from the second author's PhD thesis \cite{G1} at Sorbonne Universit\'e, 
and part of this work has been done while the second author was visiting the \textit{Institut 
Montpelli\'erain Alexander Grothendieck}, thanks to the financial support of the 
\textit{Institut Universitaire de France}. The second author warmly thanks the Max-Planck 
Institute for Mathematics in Bonn, for its hospitality and excellent working conditions.

\newpage

\section{Bundles with flat connections on $\Gamma$-twisted configuration spaces}
\label{Bundles with flat connections on twisted configuration spaces}


\subsection{The Lie algebra of infinitisemal ellipsitomic braids}\label{sec:deft1n}

In this paragraph, $\Gamma$ can be replaced by any finite abelian group (with the additive notation). 

For any positive integer $n$ we define \gls{t1IG} to be the bigraded $\kk$-Lie algebra with generators 
$x_i$ ($1\leq i\leq n$) in degree $(1,0)$, $y_i$ ($1\leq i\leq n$) in degree $(0,1)$, and $t^{\alpha}_{ij}$ 
($\alpha\in\Gamma$, $i\neq j$) in degree $(1,1)$, and relations 
\begin{flalign}
& t^{\alpha}_{ij} = t^{-\alpha}_{ji}\,, 																					\tag{tS$_{e\ell\ell}1$}\label{eqn:etS} \\
& [x_i,y_j] = [x_j,y_i] = \sum_{\alpha\in\Gamma} t^{\alpha}_{ij}\,,	 										\tag{tS$_{e\ell\ell}2$}\label{eqn:etSbis} \\
& [x_i,x_j] = [y_i,y_j] = 0\,, 																									\tag{tN$_{e\ell\ell}$}\label{eqn:etN} \\
& [x_i,y_i] = - \sum_{j:j\neq i} \sum_{\alpha\in\Gamma} t^{\alpha}_{ij},									\tag{tT$_{e\ell\ell}$}\label{eqn:etT} \\
& [t^{\alpha}_{ij},t^{\beta}_{kl}]=0\,,							\tag{tL$_{e\ell\ell}1$}\label{eqn:etL1} \\
& [x_i,t^{\alpha}_{jk}] = [y_i,t^{\alpha}_{jk}] = 0,  																\tag{tL$_{e\ell\ell}2$}\label{eqn:etL2}   \\
& [t^{\alpha}_{ij},t^{\alpha+\beta}_{ik}+t^{\beta}_{jk}] =0\,, 													\tag{t4T$_{e\ell\ell}1$}\label{eqn:et4T1} \\
& [x_i + x_j,t^{\alpha}_{ij}] = [y_i+y_j,t^{\alpha}_{ij}] = 0\,, &												\tag{t4T$_{e\ell\ell}2$}\label{eqn:et4T2} 
\end{flalign}
where $1\leq i,j,k,l\leq n$ are pairwise distinct and $\alpha, \beta \in\Gamma$. 
We will call $\t_{1,n}^\Gamma(\kk)$ the $\kk$-Lie algebra of \textit{infinitesimal ellipsitomic braids}.
Observe that $\sum_i x_i$ and $\sum_i y_i$ are central in $\t_{1,n}^\Gamma$. Then we denote by 
$\bar\t_{1,n}^\Gamma(\kk)$ the quotient of $\t_{1,n}^\Gamma(\kk)$ by $\sum_i x_i$ and $\sum_i y_i$, 
and the quotient morphism $\t_{1,n}^\Gamma(\kk)\to\bar\t_{1,I}^\Gamma(\kk)$ by $u\mapsto\bar u$.

\medskip

When $\kk=\C$ we write $\t_{1,n}^\Gamma:=\t_{1,n}^\Gamma(\C)$, 
and $\bar\t_{1,n}^\Gamma:=\bar\t_{1,n}^\Gamma(\C)$.

\medskip

There is an alternative presentation of $\t_{1,n}^\Gamma(\kk)$ and $\bar\t_{1,n}^\Gamma(\kk)$: 
\begin{lemma}\label{lem:pres1}
The Lie $\kk$-algebra $\t_{1,n}^\Gamma(\kk)$ (resp.~$\bar\t_{1,n}^\Gamma(\kk)$) can equivalently be 
presented with the same generators, and the following relations: 
\eqref{eqn:etS}, \eqref{eqn:etSbis}, \eqref{eqn:etN}, \eqref{eqn:etL1}, \eqref{eqn:etL2}, \eqref{eqn:et4T1}, and, 
for every $1\leq i\leq n$, 
\[
[\sum_jx_j,y_i]=[\sum_jy_j,x_i]=0
\]
(resp.~$\sum_jx_j=\sum_jy_j=0$). 
\end{lemma}
\begin{proof}
If $x_i,y_i$ and $t^{\alpha}_{ij}$ satisfy the initial relations, then 
$$[\underset{j}{\sum} x_j, y_i]=[ x_i, y_i]+[\underset{j\neq i}{\sum} x_j, y_i]=- \sum_{j:j\neq i} 
\sum_{\alpha\in\Gamma} t^{\alpha}_{ij}+ \sum_{j:j\neq i} \sum_{\alpha\in\Gamma} t^{\alpha}_{ij}=0.$$ 
Now, if $x_i,y_i$ and $t^{\alpha}_{ij}$ satisfy the above relations, then relations 
$[\underset{j}{\sum} x_j, y_i]=0$ and
 $[x_j,y_i] = \sum_{\alpha\in\Gamma} t^{\alpha}_{ij}$, for $i\neq j$, imply that 
 $[ x_i, y_i]=- \sum_{j:j\neq i} \sum_{\alpha\in\Gamma} t^{\alpha}_{ij}$. 
 Now, relations $[\underset{k}{\sum} x_k, y_j]=0$ and $[\underset{k}{\sum} x_k, x_i]=0$ 
 imply that $[\underset{k}{\sum} x_k, \sum_{\alpha\in\Gamma} t^{\alpha}_{ij}]=0$. 
 Thus, as $[x_i,t^{\alpha}_{jk}] = 0 $ if $ \on{card}\{i,j,k\}=3, $ we obtain relation 
 $[x_i + x_j,t^{\alpha}_{ij}] = 0 $, for $i\neq j. $ 
 In the same way we obtain $[y_i+y_j,t^{\alpha}_{ij}] = 0$, for $i\neq j$.
\end{proof}

There is an action $\Gamma^n \to \Aut(\t_{1,n}^{\Gamma}(\kk))$ defined as follows: 
\begin{itemize}
\item it leaves $x_i$'s and $y_i$'s invariant. 
\item for every $i$ and every $\alpha\in\Gamma$, $\alpha_i$ leaves $t^{\beta}_{kl}$'s invariant if $k,l\neq i$, 
and sends $t^{\beta}_{ij}$ to $t^{\beta+\alpha}_{ij}$. 
Here $\alpha_i$ denotes the element of $\Gamma^n$ whose only nonzero component is the $i$th one and is $\alpha$.  
\end{itemize}
This action descends to an action on $\bar\t_{1,n}^\Gamma(\kk)$. 

\begin{proposition}
For a group morphism $\rho:\Gamma_1\to\Gamma_2$, scalars $a,b,c,d\in\kk$ such that $ad-bc=|\ker(\rho)|$, 
and a (set theoretical) section ${\rm coker}(\rho)\to\Gamma_2$, there is a \textit{comparison morphism} 
$\phi_\rho: \t_{1,n}^{\Gamma_1}(\kk)\to\t_{1,n}^{\Gamma_2}(\kk)$ defined by 
$$
x_i\mapsto ax_i+cy_i\,,\quad y_i\mapsto b x_i+dy_i\,,\quad t_{ij}^\alpha\mapsto\sum_{\beta\in{\rm coker}(\rho)}t_{ij}^{\rho(\alpha)+\beta}\,.
$$
\end{proposition}
\begin{proof}
Let us prove that the relation $[x_i,y_j] = \sum_{\alpha\in\Gamma} t^{\alpha}_{ij} $, where $i\neq j $, is preserved 
by $\phi$. On the one hand $[\phi(x_i),\phi(y_j)]=|\ker(\rho)|\sum_{\alpha\in\Gamma_2} t^{\alpha}_{ij}$. On the other hand 
\begin{eqnarray*}
\phi([x_i,y_j]) & = & \sum_{\alpha\in\Gamma_1}  \phi(t^{\alpha}_{ij})= \sum_{\alpha\in\Gamma_1} 
						\sum_{\beta\in{\rm coker}(\rho)}t_{ij}^{\rho(\alpha)+\beta}
						=|\ker(\rho)|\sum_{\alpha\in\Gamma_2} t^{\alpha}_{ij}.
\end{eqnarray*}
The fact that the remaining relations are preserved is immediate.
\end{proof}
Comparison morphisms are bigraded, and pass to the quotient by $\sum_ix_i$, $\sum_iy_i$. 
When $\rho$ is surjective, they also are compatible with the operadic module structure of $\t_{1,\bullet}^{\Gamma}(\kk)$ 
from \cite{CG-toappear} (see Proposition 5.2 in \textit{loc. cit.}). 

\subsection{Principal bundles over $\Gamma$-twisted configuration spaces}\label{sec-confspaces}

Let $E$ be an elliptic curve over $\C$ and consider the connected unramified 
$\Gamma$-covering $p:\tilde E\to E$ corresponding to the canonical surjective 
group morphism $\rho:\pi_1(E)\cong \Z^2\to\Gamma$ where $\pi_1(E)\cong \Z^2$ 
is the natural choice of such an isomorphism. 
Let us then define the \textit{twisted configuration space} 
\[
\gls{ConfTnG}:=
\{\mathbf{z}=(z_1,\dots,z_n)\in\tilde E^n|p(z_i)\neq p(z_j)\textrm{ if }i\neq j\}\,,
\]
and $\gls{CTnG}:=\textrm{Conf}(E,n,\Gamma)/\tilde{E}$ its reduced version. 
Notice that $\textrm{C}(E,n,\Gamma)$ is just the inverse image of $\textrm{C}(E,n)$ 
under the surjection $p^n:\tilde E^n\to E^n$.

Let us fix a uniformization $\tilde E\simeq E_\tau$, where $\tau\in\mathfrak H$: 
$E_\tau=\Lambda_\tau \backslash \C$, with $\Lambda_\tau=\Z+\tau\Z$. 
Then $E\simeq E_{\tau,\Gamma}$, where $E_{\tau,\Gamma}=\Lambda_{\tau,\Gamma} \backslash \C$ 
and $\Lambda_{\tau,\Gamma}:=(1/M)\Z\times(\tau/N)\Z$. Therefore 
\[
\textrm{Conf}(E,n,\Gamma)\simeq \Lambda_\tau^n \backslash (\C^n-{\rm Diag}_{\tau,n,\Gamma})\,,
\]
where 
\[
{\rm Diag}_{\tau,n,\Gamma}
:=\{(z_1,\dots,z_n)\in\C^n|z_{ij}:=z_i-z_j\in\Lambda_{\tau,\Gamma}\textrm{ for some }i\neq j\}\,.
\]
We now define a principal ${\rm exp}(\hat{\t}_{1,n}^\Gamma)$-bundle \gls{Ptng} over 
$\textrm{Conf}(E,n,\Gamma)$ as the quotient 
\[
\Lambda_\tau^n \backslash \big((\C^n-{\rm Diag}_{\tau,n,\Gamma})\times {\rm exp}(\hat{\t}_{1,n}^\Gamma)\big)\,, 
\]
where the action is determined by the following non-abelian $1$-cocycle: 
\[
\big(\zz,(a+b\tau)_i\big)\longmapsto e^{-2\pi{\rm i}bx_i}\,.
\]
\begin{remark}[Notation]
Whenever there is an element $g$ in a group $G$, and $1\leq i\leq n$, 
we write $g_i$ for the element of $G^n$ given by $g$ on the $i$-th 
component and the unit on the others. 
\end{remark}
In other words, it is the restriction on $\textrm{Conf}(E,n,\Gamma)$ of the bundle over 
$\Lambda_\tau^n \backslash \C^n$ for which a section on $U\subset\Lambda_\tau^n \backslash \C^n$ is a regular 
map $f:\pi^{-1}(U)\to{\rm exp}(\hat{\t}_{1,n}^\Gamma)$ such that
\begin{itemize}
\item $f(\mathbf{z}+\delta_i)=f(\mathbf{z})$, 
\item $f(\mathbf{z}+\tau\delta_i)=e^{-2\pi{\rm i}x_i}f(\mathbf{z})$.
\end{itemize}
Here $\pi:\C^n\to\Lambda_\tau^n \backslash \C^n$ is the canonical projection and $\delta_i$ is the 
$i$th vector of the canonical basis of $\C^n$. 

Since the $e^{-2\pi{\rm i}\bar x_i}$'s in $\mathrm{exp}(\hat{\bar{\t}}_{1,n}^\Gamma)$ pairwise commute and their product is 
$1$, then the image of $\mathcal{P}_{\tau,n,\Gamma}$ under the natural morphism 
${\rm exp}(\hat{\t}_{1,n}^\Gamma)\to{\rm exp}(\hat{\bar{\t}}_{1,n}^\Gamma)$ 
is the pull-back of a principal ${\rm exp}(\hat{\bar{\t}}_{1,n}^\Gamma)$-bundle $\bar{\mathcal{P}}_{\tau,n,\Gamma}$ over 
$\textrm{C}(E,n,\Gamma)$. 


\subsection{Variations}\label{sec6.2var}

The first variation we are interested in concerns \textit{unordered configuration spaces}. 
The symmetric group $\mathfrak{S}_n$ acts on the left freely by automorphisms of $\textrm{Conf}(E,n,\Gamma)$ by 
$$
\sigma*(z_1,\dots,z_n):=(z_{\sigma^{-1}(1)},\dots,z_{\sigma^{-1}(n)})\,.
$$
This descends to a free action of $\mathfrak{S}_n$ on $\textrm{C}(E,n,\Gamma)$. 
We then defined the unordered twisted configuration spaces 
\[\textrm{Conf}(E,[n],\Gamma):=\mathfrak{S}_n \backslash\textrm{Conf}(E,n,\Gamma) 
\text{ and } \textrm{C}(E,[n],\Gamma):=\mathfrak{S}_n \backslash\textrm{C}(E,n,\Gamma)\,.
\]
The symmetric group $\mathfrak{S}_n$ also obviously acts on the Lie algebra $\t_{1,n}^\Gamma$. 
One can then define, keeping the notation of the previous paragraph, a principal 
$\exp(\hat{\t}_{1,n}^\Gamma)\rtimes \mathfrak{S}_n$-bundle $\mathcal{P}_{\tau,[n],\Gamma}$ over 
$\textrm{Conf}(E,[n],\Gamma)$: it is the restriction on $\textrm{Conf}(E,[n],\Gamma)$ of the bundle 
over $(\Lambda_\tau^n\rtimes \mathfrak{S}_n) \backslash \C^n$ for which a section 
on $U\subset\Lambda_\tau^n \backslash \C^n\rtimes \mathfrak{S}_n$ is a regular map 
$f:\pi^{-1}(U)\to\exp(\hat{\t}_{1,n}^\Gamma)\rtimes \mathfrak{S}_n$ such that 
\begin{itemize}
\item $f(\zz+\delta_i)=f(\zz)$, 
\item $f(\zz+\tau\delta_i)=e^{-2\pi\i x_i}f(\zz)$,
\item $f(\sigma*\zz)=\sigma f(\zz)$.
\end{itemize}
In more compact form: 
\[
\text{\gls{Ptnng}}=
(\Lambda_\tau^n\rtimes \mathfrak{S}_n) \backslash \big((\C^n-{\rm Diag}_{\tau,n,\Gamma})\times {\rm exp}(\hat{\t}_{1,n}^\Gamma)\rtimes\mathfrak{S}_n\big)\,.
\]
\begin{remark}
As before, $\mathcal{P}_{\tau,[n],\Gamma}$ descends to a principal 
$\exp(\hat{\bar{\t}}_{1,n}^\Gamma)\rtimes \mathfrak{S}_n$-bundle 
$\bar{\mathcal{P}}_{\tau,[n],\Gamma}$ over the reduced unordered twisted 
configuration space $\textrm{C}(E,[n],\Gamma)$. 
\end{remark}

\medskip

The second variation concerns ordinary configuration spaces of the base $E=E_{\tau,\Gamma}$ 
of the covering map $E_{\tau}\to E_{\tau,\Gamma}$. 

Recall from \S\ref{sec:deft1n} that the group $\Gamma^n$ acts on $\hat{\t}_{1,n}^\Gamma$. 
Hence one has a principal $\exp(\hat{\t}_{1,n}^\Gamma)\rtimes \Gamma^n$-bundle 
\[
\text{\gls{Ptn}}:=
\Lambda_{\tau,\Gamma}^n \backslash \big((\C^n-{\rm Diag}_{\tau,n,\Gamma})\times {\rm exp}(\hat{\t}_{1,n}^\Gamma)\rtimes\Gamma^n\big)
\]
over $\textrm{Conf}(E,n)\simeq\Lambda_{\tau,\Gamma}^n \backslash (\C^n-{\rm Diag}_{\tau,n,\Gamma})$, 
where the action is determined by the non-abelian cocycle 
\[
\big(\zz,(\frac{u}{M}+\frac{v}{N}\tau)_i\big)\longmapsto 
e^{-\frac{2\pi{\rm i}v}{N}x_i}(\bar{u},\bar{v})_i\,.
\]
\begin{remark}
The map sending $\frac{u}{M}+\frac{v}{N}\tau$ to $(\bar{u},\bar{v})$ exhibits an isomorphism 
$\Lambda_{\tau,\Gamma}/\Lambda_{\tau}\simeq \Gamma$, that we will use on several occasions. 
Using this, if $\tilde{\alpha}=a+b\tau\in\Lambda_{\tau,\Gamma}$ is a lift of $\alpha\in\Gamma$, 
then the non-abelian cocycle is 
\[
(\zz,\alpha_i)\longmapsto e^{-2\pi{\rm i}bx_i}\alpha_i\,.
\]
\end{remark}
\begin{remark}
In a similar way as before, the above bundle obviously descends to a principal 
$\exp(\hat{\bar{\t}}_{1,n}^\Gamma)\rtimes(\Gamma^n/\Gamma)$-bundle 
$\bar{\mathcal{P}}_{(\tau,\Gamma),n}$ over the reduced ordinary configuration 
space $\textrm{C}(E,n)$. 
\end{remark}
In concrete terms, a section over $U\subset\Lambda_{\tau,\Gamma} \backslash \C^n$ of $\mathcal{P}_{(\tau,\Gamma),n}$ 
is a regular map $f:\pi^{-1}(U)\to\exp(\hat{\t}_{1,n}^\Gamma)\rtimes \Gamma^n$ such that 
\begin{itemize}
	\item $f(\zz+\delta_i/M)=(\bar{1},\bar{0})_if(\zz)$, 
	\item $f(\zz+\tau\delta_i/N)=(\bar{0},\bar{1})_ie^{-\frac{2\pi\i}{N}x_i}f(\zz)$.
\end{itemize}

\begin{remark}
We leave to the reader the task of combining the two variations. 
\end{remark}


\subsection{Flat connections on $\mathcal{P}_{\tau,n,\Gamma}$ and its variants}\label{sec:flatconn}

A flat connection $\nabla_{\tau,n,\Gamma}$ on $\mathcal{P}_{\tau,n,\Gamma}$ is the same as 
an equivariant flat connection on the trivial ${\rm exp}(\hat{\t}_{1,n}^\Gamma)$-bundle 
over $\C^n-{\rm Diag_{\tau,n,\Gamma}}$, i.e., a connection of the form 
\[
\nabla_{\tau,n,\Gamma}:=d-\sum_{i=1}^nK_i(\mathbf{z}|\tau)dz_i\,,
\]
where $K_i(-|\tau):\C^n\to\hat{\t}_{1,n}^\Gamma$ are meromorphic with only poles at 
${\rm Diag}_{\tau,n,\Gamma}$, and such that for any $i,j$: 
\begin{itemize}
	\item[(a)] $K_i(\mathbf{z}+\delta_j|\tau)=K_i(\mathbf{z}|\tau)$, 
	\item[(b)] $K_i(\mathbf{z}+\tau\delta_j|\tau)
							=e^{-2\pi{\rm i}\ad(x_j)}K_i(\mathbf{z}|\tau)$, 
	\item[(c)] $[\partial_i-K_i(\mathbf{z}|\tau),\partial_j-K_j(\mathbf{z}|\tau)]=0$. 
\end{itemize}
Moreover, the image of $\nabla_{\tau,n,\Gamma}$ under $\hat{\t}_{1,n}^\Gamma\to\hat{\bar{\t}}_{1,n}^\Gamma$ 
is the pull-back of a (necessarily flat) connection $\bar\nabla_{\tau,n,\Gamma}$ on $\bar{\mathcal{P}}_{\tau,n,\Gamma}$ 
if and only if: 
\begin{itemize}
	\item[(d)] $\bar K_i(\mathbf{z}|\tau)=\bar K_i(\mathbf{z}+u\sum_i\delta_i|\tau)$ 
						 for any $u\in\C$ and $\sum_i\bar K_i(\mathbf{z}|\tau)=0$. 
\end{itemize}
Similarly, the image of $\nabla_{\tau,n,\Gamma}$ under $\hat{\t}_{1,n}^\Gamma\to\hat{\t}_{1,n}^\Gamma\rtimes\Gamma^n$
is the pull-back of a (necessarily flat) connection $\nabla_{(\tau,\Gamma),n}$ on $\mathcal{P}_{(\tau,\Gamma),n}$ 
if and only if: 
\begin{itemize}
	\item[(e)] $K_i(\mathbf{z}+\delta_i/M|\tau)
							=(\bar1,\bar0)_j\cdot K_i(\mathbf{z}|\tau)$, 
	\item[(f)] $K_i(\mathbf{z}+\tau\delta_i/N|\tau)
							=(\bar0,\bar1)_j\cdot e^{\frac{-2\pi{\rm i}}{N}\ad(x_j)}K_i(\mathbf{z}|\tau)$, 
\end{itemize}
\begin{remark}
Observe that (e) implies (a), and that (f) implies (b). 
\end{remark}
Finally, the image  of $\nabla_{\tau,n,\Gamma}$ under $\hat{\t}_{1,n}^\Gamma\to\hat{\t}_{1,n}^\Gamma\rtimes\mathfrak{S}_n$ 
is the pull-back of a (necessarily flat) connection $\nabla_{\tau,[n],\Gamma}$ on $\bar{\mathcal{P}}_{\tau,[n],\Gamma}$ 
if and only if: 
\begin{itemize}
	\item[(g)] $K_i((ij)*\zz)=(ij)\cdot K_i(\zz)$.
\end{itemize}


\subsection{Constructing the connection}

We now construct a connection satisfying properties (d) to (g). 
Let us take the same conventions for theta functions as in \cite{CEE}.
This is the unique holomorphic function $\C\times\HH \to \C$, $(z,\tau)\mapsto \theta(z|\tau)$, such that 
\begin{itemize}
	\item $\{z | \theta(z|\tau) = 0\} = \Lambda_{\tau}$,
	\item $\theta(z+1|\tau)
= -\theta(z|\tau) = \theta(-z|\tau)$
	\item $\theta(z+\tau|\tau) = - e^{-\pi\i\tau} e^{-2\pi\i z}\theta(z|\tau)$
	\item $\partial_{z} \theta(0|\tau)=1$.
\end{itemize}
In particular, $\theta(z|\tau+1) = \theta(z|\tau)$, while $\theta(-z/\tau|-1/\tau)
= - (1/\tau) e^{(\pi\i/\tau) z^2} \theta(z|\tau)$. 
If $\eta(\tau) = q^{1/24}\prod_{n\geq 1} (1-q^n)$ where $q = e^{2\pi\i\tau}$, 
and if we set $\vartheta(z|\tau) := \eta(\tau)^3 \theta(z|\tau)$, then 
$\partial_\tau\vartheta = (1/4\pi\i) \partial_z^2\vartheta$.

Observe that for any $\tilde\alpha=(a_0,a)\in\Lambda_{\tau,\Gamma}$ lifting $\alpha\in\Gamma$, the term 
$e^{-2\pi\i ax}(\theta(z-\tilde\alpha+ x))/\left(\theta(z-\tilde\alpha)\theta(x)\right)$ only depends on 
the class $\alpha=(\bar a_0,\bar a)\in\Gamma$ of $\tilde\alpha$ mod $\Lambda_\tau$.  
Then we set 
$$
k_\alpha(x,z|\tau):=e^{-2\pi\i ax}{{\theta(z-\tilde\alpha+x|\tau)}
\over{\theta(z-\tilde\alpha|\tau)\theta(x|\tau)}}-{1\over x}
=e^{-2\pi\i ax} k(x,z-\tilde\alpha|\tau)+{{e^{-2\pi\i ax}-1}\over x}\,,
$$
where $k(x,z|\tau):={{\theta(x+z)}\over{\theta(x)\theta(z)}}-\frac{1}{x}$ (as in \cite{CEE}), and  
$$
K_{ij}(z|\tau):=\sum_{\alpha\in\Gamma}k_\alpha(\ad x_i,z|\tau)(t^{\alpha}_{ij})\,,
\quad K_i(\mathbf{z}|\tau):=-y_i+\sum_{j:j\neq i}K_{ij}(z_{ij}|\tau)\,.
$$
In the rest of the section we fix $\tau\in\mathfrak H$ and drop it from the notation. Recall from 
\cite{CEE} that $k(x,z\pm1)=k(x,z)$ and $$k(x,z\pm\tau)=e^{\mp2\pi\i x}k(x,z) + \frac{e^{\mp2\pi\i x}-1}{x}.$$ 
We then define the universal ellipsitomic KZB connection on $\mathcal{P}_{\tau,n,\Gamma}$ by
$$
\nabla_{\tau,n,\Gamma}^{\on{KZB}}:=d-\sum_{i=1}^n K_i(\mathbf{z}|\tau)dz_i\,.
$$

\begin{proposition}\label{prop:equivariance1}
The $K_{ij}(z)$'s have the following equivariance properties: 
\begin{align}
K_{ij}(z+{1\over M})=&(\bar1,\bar0)_i\cdot (K_{ij}(z)), \\
K_{ij}(z+{\tau\over N})=&(\bar0,\bar{-1})_i \cdot e^{-{{{2\pi\i}}\over N}\ad x_j}\cdot (K_{ij}(z)) 
									+ (\bar0,\bar{-1})_i\cdot ( \sum_{\alpha\in\Gamma}{{e^{-{{2\pi\i}\over N} \ad x_i}-1}\over{\ad x_i }}(t^{\alpha}_{ij})). 
\end{align}
\end{proposition}
\begin{proof}
Let us choose representatives $0\leq u \leq M-1$ and $0\leq v\leq N-1$ so that $\tilde\alpha={u\over M}+\tau{v\over N} $. 
The first equation comes from a straightforward verification. Let us show the second relation. On the one hand, 
\begin{eqnarray*}
  K_{i j} \left( z + \frac{\tau}{N} \right) & = & \sum_{\alpha
  \in \Gamma} k_{\alpha} \left( \tmop{ad} x_i ,z + \frac{\tau}{N}
  \right) ( t_{i j}^{\alpha} ) \nonumber\\
  & = & \left( \sum_{\alpha \in \Gamma} e^{-{{2\pi\i}\over N}   v \tmop{ad} (
  x_{i} )} k \left( \tmop{ad} x_i ,z + \frac{\tau}{N} -
  \tilde{\alpha} \right) + \frac{e^{-{{2\pi\i}\over N}  v \tmop{ad} x_i}
  -1}{\tmop{ad} x_i} \right) ( t_{i j}^{\alpha} ) \nonumber\\
  & = & \left( \sum_{\alpha \in \Gamma} e^{-{{2\pi\i}\over N} ( v-1 )
  \tmop{ad} x_i} k ( \tmop{ad} x_i ,z - \tilde{\alpha} ) +
  \frac{e^{-{{2\pi\i}\over N}( v-1 ) \tmop{ad} x_i} -1}{\tmop{ad} (
  x_{i} )} \right) ( t_{i j}^{\alpha - ( \overline{ 0 } , \bar{1}   )} )
  \nonumber\\
  & = & ( \overline{ 0 } , \overline{-1}   )_i\cdot \left( \sum_{\alpha \in
  \Gamma} e^{-{{2\pi\i}\over N}  ( v-1 ) \tmop{ad} x_i} k ( \tmop{ad} (
  x_{i} ) ,z - \tilde{\alpha} ) + \frac{e^{-{{2\pi\i}\over N}  ( v-1 )
  \tmop{ad} x_i} -1}{\tmop{ad} x_i} \right) ( t_{i j}^{\alpha} ).
\end{eqnarray*}
On the other hand,
\begin{eqnarray*}
  e^{-{{2\pi\i}\over N}\tmop{ad} x_j} K_{i j} ( z) & = &
  e^{-{{2\pi\i}\over N} \tmop{ad} x_j} \left( \sum_{\alpha \in \Gamma}
  k_{\alpha} ( \tmop{ad} x_i ,z) \right) ( t_{i j}^{\alpha} )
  \nonumber\\
  & = & e^{{{2\pi\i}\over N} \tmop{ad} x_i} \left( \sum_{\alpha \in
  \Gamma} e^{-{{2\pi\i}\over N}   v \tmop{ad} x_i} k ( \tmop{ad} x_i
  ,z- \tilde{\alpha} ) + \frac{e^{-{{2\pi\i}\over N}  v \tmop{ad}  x_{i}
  } -1}{\tmop{ad} x_i} \right) ( t_{i j}^{\alpha} ) \nonumber\\
  & = & \left( \sum_{\alpha \in \Gamma} e^{-{{2\pi\i}\over N}  ( v-1 )
  \tmop{ad} x_i} k ( \tmop{ad} x_i ,z - \tilde{\alpha} ) +
  \frac{e^{-{{2\pi\i}\over N}  ( v-1 )\tmop{ad} x_i} -e^{{{2\pi\i}\over N}
  \tmop{ad} x_i}}{\tmop{ad} x_i} \right) ( t_{i j}^{\alpha} ), 
\end{eqnarray*}
so
\begin{eqnarray*}
  \sum_{\alpha \in \Gamma} e^{-{{2\pi\i}\over N}( v-1 ) \tmop{ad} x_i} k
  ( \tmop{ad} x_i ,z - \tilde{\alpha} )( t_{i j}^{\alpha} ) & = & e^{-{{2\pi\i}\over N}
  \tmop{ad} x_j} K_{i j} ( z)\\ & &   - \sum_{\alpha \in \Gamma}
  \frac{e^{-{{2\pi\i}\over N}( v-1 ) \tmop{ad} x_i} -e^{{{2\pi\i}\over N}
  \tmop{ad} x_i}}{\tmop{ad} x_i} ( t_{i j}^{\alpha} ).
\end{eqnarray*}
By putting these two equations together we finally get
\begin{eqnarray*}
  K_{i j}\left( z + \frac{\tau}{N} \right) & = & (\overline{ 0 },\overline{-1})_i\cdot 
	e^{-{{2\pi\i}\over N} \tmop{ad} x_j}
  K_{i j} ( z )  \\ & & + \sum_{\alpha \in \Gamma} \frac{-
  e^{-{{2\pi\i}\over N}( v-1 ) \tmop{ad} x_i} +e^{{{2\pi\i}\over N}
  \tmop{ad} x_i} +e^{-{{2\pi\i}\over N}( v-1 ) \tmop{ad} x_i}
  -1}{\tmop{ad} x_i} ( t_{i j}^{\alpha} )\\
  & = & ( \overline{ 0 } , \overline{-1})_i\cdot e^{-{{2\pi\i}\over N}
  \tmop{ad} x_j} K_{i j}( z )   \noplus + (\overline{ 0 } , \overline{-1})_i\cdot \left( \sum_{\alpha \in \Gamma}
  \frac{e^{{{2\pi\i}\over N}\tmop{ad} x_i} -1}{\tmop{ad} x_i} (
  t_{i j}^{\alpha} ) \right).
\end{eqnarray*}
\end{proof}
Now recall that $\frac{e^{{{2\pi\i}\over N}\tmop{ad} x_i}
-1}{\tmop{ad} x_i} = \frac{1 -e^{-{{2\pi\i}\over N}\tmop{ad}  x_{j}
}}{\tmop{ad} x_j}$ and $\frac{1 -e^{-{{2\pi\i}\over N} \tmop{ad} 
x_{j} }}{\tmop{ad} x_j} ( t_{i j} ) = \left( 1 -e^{-{{2\pi\i}\over N}
\tmop{ad} x_j} \right) ( y_{i} )$. We thus have 
\begin{eqnarray*} 
  K_{i} \left( \zz + \frac{\tau}{N} \delta_{j} \right) & =
  & -y_{i} + \underset{j' \neq i,j}{\sum} K_{i j'} ( z_{i j'} ) +K_{i
  j}\left( z_{i j} + \frac{\tau}{N} \right)
\end{eqnarray*}
and therefore we get the announced relation
\begin{eqnarray*}
  K_{i} \left( \zz + \frac{\tau}{N} \delta_{j} \right) & =
  & ( \overline{ 0 } , \bar{1} )_{j}\cdot e^{-{{2\pi\i}\over N}
  \tmop{ad} x_j} K_{i} ( \tmmathbf{z} ).
\end{eqnarray*}

Consequently the $K_i(\mathbf{z})$'s satisfy conditions (e) and (f) above (and thus also (a) and (b)). 

Moreover, the $K_i(\mathbf{z})$'s also satisfy conditions (d). 
Indeed, the first part of (d) is immediate and $k_\alpha(x,z)+k_{-\alpha}(-x,-z)=0$, 
therefore $K_{ij}(z)+K_{ji}(-z)=0$, and thus $\sum_iK_i(\mathbf{z})=-\sum_iy_i$.

Finally, from their very definition, the $K_i(\mathbf{z})$'s also satisfy condition (g). 

In the next paragraph we show that the flatness condition (c) is satisfied. 


\subsection{Flatness of the connection}

\begin{proposition}\label{prop:flatness1}
$[\partial_i-K_i(\zz),\partial_j-K_j(\zz)]=0$, i.e., condition (c) is satisfied. 
\end{proposition}
\begin{proof}
First we have 
$$\partial_i(K_j(\zz))-\partial_j(K_i(\zz))=\partial_i K_{ji}(z_{ji})-\partial_j K_{ij}(z_{ij})
=\partial_i(K_{ij}(z_{ij})+K_{ji}(z_{ji}))=0$$ since $K_{ij}(z)+K_{ji}(-z)=0$. 
Therefore we have to prove that $[K_i(\zz),K_j(\zz)]=0$. 
As in \cite{CEE} it follows from the universal classical dynamical Yang-Baxter equation: 
\begin{equation}\label{eq:3} \tag{CDYBE}
-[y_i,K_{jk}]+[K_{ji},K_{ki}]+c.p.(i,j,k)=0\,,
\end{equation}
which we now prove (here $K_{ij}:=K_{ij}(z_{ij})$). For any $f(x)\in\C[[x]]$,  
$$
[y_k,f(\ad x_i)(t^{\alpha}_{ij})] = \sum_{\beta\in\Gamma}
{{f(\ad x_i) - f(-\ad x_j)}\over{\ad x_i + \ad x_j}}
[-t^{\beta}_{ki},t^{\alpha}_{ij}], 
$$
$$
[y_i,f(\ad x_j)(t^{\alpha}_{jk})] = \sum_{\beta\in\Gamma}
{{f(\ad x_j) - f(\ad x_i + \ad x_j)}\over{-\ad x_i}}
[-t^{\beta}_{ij},t^{\alpha}_{jk}], 
$$
$$
[y_j,f(\ad x_k)(t^{\alpha}_{ki})] = \sum_{\beta\in\Gamma}
{{f(-\ad x_i - \ad x_j) - f(-\ad x_i)}\over{-\ad x_j}}
[-t^{\beta}_{jk},t^{\alpha}_{ki}].  
$$
It follows that the l.h.s. of (\ref{eq:3}) is now 
\begin{eqnarray*}
&& \sum_{\alpha,\beta\in\Gamma}
\big(k_\alpha(-\ad x_j,z_{ij})k_\beta(-\ad x_k,z_{ik})
- k_\alpha(\ad x_i,z_{ij})k_{\beta-\alpha}(-\ad x_k,z_{jk})
\\ && 
+ k_\beta(\ad x_i,z_{ik})k_{\beta-\alpha}(\ad x_j,z_{jk})
+ {{k_{\beta-\alpha}(\ad x_j,z_{jk})-k_{\beta-\alpha}(\ad x_i+\ad x_j,z_{jk})}\over{\ad x_i}}
\\ && 
+ {{k_\beta(\ad x_i,z_{ik})-k_\beta(\ad x_i+\ad x_j,z_{ik})}\over{\ad x_j}}
- {{k_\alpha(\ad x_i,z_{ij})-k_\alpha(-\ad x_j,z_{ij})}\over{\ad x_i + \ad x_j}}\big)
[t^{\alpha}_{ij},t^{\beta}_{ik}]\,,
\end{eqnarray*}
and thus (\ref{eq:3}) follows from the identity
\begin{eqnarray*}
&&
k_\alpha(-v,z)k_\beta(u+v,z')
- k_\alpha(u,z)k_{\beta-\alpha}(u+v,z'-z)
+ k_\beta(u,z') k_{\beta-\alpha}(v,z'-z)
\\ && 
+ {{k_{\beta-\alpha}(v,z'-z)
- k_{\beta-\alpha}(u+v,z'-z)}\over{u}}
+ {{k_\beta(u,z') - k_\beta(u+v,z')}
\over{v}} \nonumber \\ && 
- {{k_\alpha(u,z) - k_\alpha(-v,z)}
\over{u+v}} = 0\,. 
\end{eqnarray*} 
This last identity can be written as 
\begin{eqnarray}
\left( k_\alpha(-v,z) - {1\over v}\right) 
\left( k_\beta(u+v,z') + {1\over{u+v}}\right) 
- \left( k_\alpha(u,z) + {1\over u}\right) 
\left( k_{\beta-\alpha}(u+v,z'-z) + {1\over{u+v}} \right) && \nonumber \\
+\left( k_\beta(u,z')+{1\over u}\right)\left(k_{\beta-\alpha}(v,z'-z)+{1\over v}\right)=0\,, && \label{twtheta:id}
\end{eqnarray}
which (taking into account that 
$k_\alpha(x,z)+(1/x)=e^{-2\pi\i ax}\left( k(x,z-\tilde\alpha)+(1/x)\right)$) 
is a consequence of equation (3) of \cite{CEE}. 
\end{proof}
We have therefore proved: 
\begin{theorem}
$\nabla_{\tau,n,\Gamma}$ is a flat connection on $\mathcal{P}_{\tau,n,\Gamma}$, and its image under 
$\hat{\t}_{1,n}^\Gamma\to\hat{\bar\t}_{1,n}^\Gamma$ is the pull-back of a flat connection 
$\bar\nabla_{\tau,n,\Gamma}$ on $\bar{\mathcal{P}}_{\tau,n,\Gamma}$. 
\hfill \qed
\end{theorem}

\section{Lie algebras of derivations and associated groups}
\label{Lie algebras of derivations and associated groups}

\subsection{The Lie algebras $\tilde\d_0^\Gamma$ and $\tilde\d^\Gamma$}

Let $\f_\Gamma$ be the free Lie algebra with generators $x$, $t^{\alpha}$ ($\alpha\in\Gamma$). 
Let $p,q>0$. We define $\tilde\d_0^{p,q}$ to be the subspace of $\f_\Gamma\oplus(\f_\Gamma)^{\oplus|\Gamma|}$ 
consisting of elements $$(D,C), \text{ where } C=(C_\alpha)_{\alpha\in\gamma},$$ such that 
$\deg_x(D)+\deg_t(D)=\deg_x(C_\alpha)+\deg_t(C_\alpha)=p$ and $\deg_t(D)-1=\deg_t(C_\alpha)=q$ for every 
$\alpha\in\Gamma$, and that satisfy the following set of linear equations: 
\begin{itemize}
\item[(i)] $C_\alpha(x,t^{\beta})=C_{-\alpha}(-x,t^{-\beta})$ in $\f_\Gamma$, 
\item[(ii)] $[x,D(x,t^{\beta})]+\sum_\alpha[t^{\alpha},C_\alpha(x,t^{\beta})]=0$ in $\f_\Gamma$, 
\item[(iii)] $[D(x_1,t^{\beta}_{13}),y_2]+c.p.(1,2,3)=0$ in $\t_{1,3}^\Gamma$, 
\item[(iv)] $[D(x_1,t^{\beta}_{12})+D(x_1,t^{\beta}_{13})
-[C_\alpha(x_2,t^{\beta}_{23}),y_1],t^{\alpha}_{23}]=0$ in $\t_{1,3}^\Gamma$, 
\item[(v)] $[C_\alpha(x_1,t^{\gamma}_{12}),t^{\alpha+\beta}_{13}+t^{\beta}_{23}]
+[t^{\alpha+\beta}_{13},C_{\alpha+\beta}(x_1,t^{\gamma}_{13})]
+[t^{\beta}_{23},C_\beta(x_2,t^{\gamma}_{23})]$ 
commutes with $t^{\alpha}_{12}$ in $\t_{1,3}^\Gamma$. 
\end{itemize}
Remark that (i) and (ii) imply another relation 
\begin{itemize}
\item[(vi)] $D(x,t^{\beta})=-D(-x,t^{-\beta})$\,,
\end{itemize}
which is very useful for computations. Then $\tilde\d^\Gamma_0:=\oplus_{p,q}(\tilde\d^\Gamma_0)^{p,q}$. 

We then define a Lie bracket $\<,\>$ on $\f_\Gamma\oplus(\f_\Gamma)^{\oplus|\Gamma|}$ as follows: 
$$
\<(D,C),(D',C')\>:=(\delta_C(D')-\delta_{C'}(D),[C,C']+\delta_C(C')-\delta_{C'}(C))\,,
$$
where $\delta_C\in\Der(\f_\Gamma)$ is the derivation 
\begin{itemize}
\item $x\mapsto 0$, $t^\alpha\mapsto[t^\alpha,C_\alpha]$,
\item $\delta_C$ acts on $(\f_\Gamma)^{\oplus|\Gamma|}$ componentwise on a direct sum : $\delta_C(C')_\alpha=\delta_C(C'_\alpha)$,
\item the bracket is understood componentwise as well: $[C,C']_\alpha=[C_\alpha,C'_\alpha]$.
\end{itemize}

We let the reader check that $\tilde\d^\Gamma_0$ is stable under $\<,\>$, and becomes a bigraded Lie 
algebra\footnote{The proof is straightforward but quite long. We do not give it since we do use 
another simpler Lie algebra below. }. 

We now define \gls{dtg} as the quotient of the free product $\tilde\d^\Gamma_0*\sl_2$ by the relations 
$[\tilde e,(D,C)]=0$, $[\tilde h,(D,C)]=(p-q)(D,C)$, and $(\on{ad}^p\tilde f)(D,C)=0$ if $(D,C)\in\tilde\d^\Gamma_0$ 
is homogeneous of bidegree $(p,q)$. Here 
$$\tilde e=\begin{pmatrix}0&1\\0&0\end{pmatrix}, 
\tilde h=\begin{pmatrix}1&0\\0&-1\end{pmatrix} \text{ and } 
\tilde f=\begin{pmatrix}0&0\\1&0\end{pmatrix}
$$ 
form the standard basis of $\sl_2$. If we respectively give degree $(1,-1)$, $(0,0)$ and $(-1,1)$ to 
$\tilde e$, $\tilde h$ and $\tilde f$ then $\tilde\d^\Gamma$ becomes $\Z^2$-graded. 

We then define $\tilde\d^\Gamma_+:=\ker(\tilde\d^\Gamma\to\sl_2)$, which is $(\Z_{>0})^2$-graded. 
One observes that it is positively graded and finite dimensional in each degree. Thus, it is a direct sum of finite dimensional $\sl_2$-modules. 

\subsection{The Lie algebras $\d^\Gamma_0$ and $\d^\Gamma$}

We write $\d^\Gamma_0$ for the free bigraded Lie algebra generated by $\delta_{s,\gamma}$'s ($s\geq0$, 
$\gamma\in\Gamma$) in degree $(s+1,s)$ with relations $$\delta_{s,\gamma}=(-1)^s\delta_{s,-\gamma},$$
for all $s\geq0$ and $\gamma\in\Gamma$.

We then define \gls{dg} as the quotient of the free product $\d_0^\Gamma*\sl_2$ by 
the relations\break $[\tilde e,\delta_{s,\gamma}]=0$, $[\tilde h,\delta_{s,\gamma}]=s\delta_{s,\gamma}$ 
and $\on{ad}^{s+1}(\tilde f)(\delta_{s,\gamma})=0$; and $\d^\Gamma_+$ as the kernel of 
$\d^\Gamma\to\sl_2$. 
As above, $\d^\Gamma=\d^\Gamma_+\rtimes\sl_2$, and $\d^\Gamma_+$ is positively graded 
(actually $(\Z_{>0})^2$-graded). \\

We now give examples of elements in $\tilde\d^\Gamma_0$ that are of some use below. For any 
$s\in\NN$ and $\gamma\in\Gamma$, we set 
$$
D_{s,\gamma}:=\sum_{p+q=s-1}\sum_{\beta\in\Gamma}[(\ad x)^p t^{\beta-\gamma},(-\ad x)^q t^{\beta}]
$$
and 
$$(C_{s,\gamma})_\alpha:=(\ad x)^s t^{\alpha-\gamma}+(-\ad x)^s t^{\alpha+\gamma}.$$
Observe that $(D_{s,\gamma},C_{s,\gamma})=(-1)^s(D_{s,-\gamma},C_{s,-\gamma})$. 

The following result tells us that $\delta_{s,\gamma}\mapsto(D_{s,\gamma},C_{s,\gamma})$ defines a bigraded Lie 
algebra morphism $\d_0^\Gamma\to\tilde\d_0^\Gamma$, that obviously extends to $\d^\Gamma\to\tilde\d^\Gamma$. 
\begin{proposition}
$(D_{s,\gamma},C_{s,\gamma})\in(\tilde\d^\Gamma_0)^{s+1,1}$. 
\end{proposition}
\begin{proof}
First observe that relations (i) and (vi) are obviously satisfied. 

To prove (ii) it suffices to notice that in the free Lie algebra with three generators $x,t_1,t_2$,
$$
[t_1,(\on{ad}x)^st_2]+[t_2,(-\on{ad}x)^st_1]=\sum_{p+q=s-1}[x,[(-\on{ad}x)^qt_1,(\on{ad}x)^pt_2]]\,.
$$

Let us prove (iii). In $\t_{1,n}^\Gamma$ we compute for $\#\{i,j,k\}=3$, 
$$
[y_k,(\on{ad}x_i)^pt^{\alpha}_{ij}]
=-\sum_{k+l=p-1}\sum_\beta(\on{ad}x_i)^k[t^{\beta}_{ik},(\on{ad}x_i)^lt^{\alpha}_{ij}]
$$
$$
=\sum_{k+l=p-1}\sum_\beta(\on{ad}x_i)^k(-\on{ad}x_j)^l[t^{\beta}_{ik},t^{\alpha-\beta}_{kj}]
=\sum_{k+l=p-1}\sum_\beta[(\on{ad}x_i)^kt^{\beta}_{ik},(-\on{ad}x_j)^lt^{\alpha-\beta}_{kj}]\,.
$$
Therefore, in $\t_{1,3}^\Gamma$,  
$$
[y_1,D(x_2,t^{\beta}_{23})]=
\sum_{k+l+m=s-2}\sum_{\alpha,\beta}
[[(\on{ad}x_2)^kt^{\beta}_{21},(-\on{ad}x_3)^lt^{\alpha-\beta-\gamma}_{13}]
,(-\on{ad}x_2)^mt^{\alpha}_{23}]
$$
$$
+\sum_{k+l+m=s-2}\sum_{\alpha,\beta}(-1)^{l+m+1}
[(\on{ad}x_2)^kt^{\alpha-\gamma}_{23}
,[(\on{ad}x_2)^lt^{\beta}_{21},(-\on{ad}x_3)^mt^{\alpha-\beta}_{13}]]\,.
$$
Then $[y_1,D(x_2,t^{\beta}_{23})]+c.p.(1,2,3)=0$ follows from the Jacobi identity. 

Let us prove (iv). On the one hand, 
$$
[D(x_1,t^{\beta}_{12})+D(x_1,t^{\beta}_{13}),t^{\alpha}_{23}]=
$$
$$
=\sum_{p+q=s-1}\sum_{\beta\in\Gamma}[[(\ad x_1)^p t^{\beta-\gamma}_{12},(-\ad x_1)^q t^{\beta}_{12}]
+[(\ad x_1)^p t^{\beta-\gamma}_{13},(-\ad x_1)^q t^{\beta}_{13}],t^{\alpha}_{23}]
$$
$$
=-\sum_{p+q=s-1}\sum_{\beta\in\Gamma}\big(
[(\ad x_1)^p[t^{\alpha+\beta-\gamma}_{13},t^{\alpha}_{23}],(-\ad x_1)^q t^{\beta}_{12}]
+[(\ad x_1)^pt^{\beta-\gamma}_{12},(-\ad x_1)^q[t^{\alpha+\beta}_{13},t^{\alpha}_{23}]]
$$
$$
+[(\ad x_1)^p[t^{\beta-\gamma}_{12},t^{\alpha}_{23}],(-\ad x_1)^q t^{\alpha+\beta}_{13}]
+[(\ad x_1)^pt^{\alpha+\beta-\gamma}_{13},(-\ad x_1)^q[t^{\beta}_{12},t^{\alpha}_{23}]]\big)
$$
$$
=[t^{\alpha}_{23},\sum_{p+q=s-1}\sum_{\beta\in\Gamma}
(\ad x_1)^p[t^{\alpha+\beta-\gamma}_{13},(-\ad x_1)^q t^{\beta}_{12}]
+(\ad x_1)^p[t^{\beta}_{12},(-\ad x_1)^q t^{\alpha+\beta+\gamma}_{13}]]
$$
$$
=[t^{\alpha}_{23},\sum_{p+q=s-1}\sum_{\beta\in\Gamma}
(\ad x_2)^p(-\ad x_3)^q[t^{\alpha+\beta-\gamma}_{13}+(-1)^st^{\alpha+\beta+\gamma}_{13},t^{\beta}_{12}]]\,.
$$
On the other hand, 
$$
[C_\alpha(x_2,t^{\beta}_{23}),y_1]
=[(\ad x_2)^s t^{\alpha-\gamma}_{23}+(-\ad x_2)^s t^{\alpha+\gamma}_{23},y_1]
$$
$$
=-\sum_{p+q=s-1}\sum_{\beta\in\Gamma}(\on{ad}x_2)^p(-\on{ad}x_3)^q
[t^{\beta}_{12},t^{\alpha+\beta-\gamma}_{31}+(-1)^st^{\alpha+\beta+\gamma}_{31}]\,.
$$
Therefore (iv) is satisfied. 

Let us prove (v). We compute  
$$
[C_\alpha(x_1,t^{\gamma}_{12}),t^{\alpha+\beta}_{13}+t^{\beta}_{23}]
=[(\on{ad}x_1)^st^{\alpha-\gamma}_{12}+(-\on{ad}x_1)^st^{\alpha+\gamma}_{12}
,t^{\alpha+\beta}_{13}+t^{\beta}_{23}]
$$
$$
=(\on{ad}x_2)^s[t^{\alpha+\gamma}_{12}+(-1)^st^{\alpha-\gamma}_{12},t^{\alpha+\beta}_{13}]
+(\on{ad}x_1)^s[t^{\alpha-\gamma}_{12}+(-1)^st^{\alpha+\gamma}_{12},t^{\beta}_{23}]
$$
$$
=(\on{ad}x_2)^s[t^{\alpha+\beta}_{13},t^{\beta-\gamma}_{23}+(-1)^st^{\beta+\gamma}_{23}]
+(\on{ad}x_1)^s[t^{\beta}_{23},t^{\alpha+\beta-\gamma}_{13}+(-1)^st^{\alpha+\beta+\gamma}_{13}]\,.
$$
Therefore, by defining $A=t^{\beta-\gamma}_{23}+(-1)^st^{\beta+\gamma}_{23}$ and 
$B=t^{\alpha+\beta-\gamma}_{13}+(-1)^st^{\alpha+\beta+\gamma}_{13}$ 
we obtain 
$$
[t^{\alpha}_{12},[C_\alpha(x_1,t^{\gamma}_{12}),t^{\alpha+\beta}_{13}+t^{\beta}_{23}]]
=[t^{\alpha}_{12},[t^{\alpha+\beta}_{13},(\on{ad}x_2)^sA]+[t^{\beta}_{23},(\on{ad}x_1)^sB]]
$$
$$
=[[t^{\alpha}_{12},t^{\alpha+\beta}_{13}],(-\on{ad}x_3)^sA]
+[t^{\alpha+\beta}_{13},(-\on{ad}x_3)^s[t^{\alpha}_{12},A]]
$$
$$+[[t^{\alpha}_{12},t^{\beta}_{23}],(-\on{ad}x_3)^sB]
+[t^{\beta}_{23},(-\on{ad}x_3)^s[t^{\alpha}_{12},B]]
$$
$$=[[t^{\beta}_{23},t^{\alpha}_{12}],(-\on{ad}x_3)^sA]
+[t^{\alpha+\beta}_{13},(-\on{ad}x_3)^s[B,t^{\alpha}_{12}]]
$$
$$
+[[t^{\alpha+\beta}_{13},t^{\alpha}_{12}],(-\on{ad}x_3)^sB]
+[t^{\beta}_{23},(-\on{ad}x_3)^s[A,t^{\alpha}_{12}]]
$$
$$=[[t^{\beta}_{23},(\on{ad}x_2)^sA]
+[t^{\alpha+\beta}_{13},(\on{ad}x_1)^sB],t^{\alpha}_{12}]\,.
$$
This finishes the proof.
\end{proof}
\begin{remark}
We do not know if $\d_0^\Gamma\to\tilde\d^\Gamma_0$ is injective or not. 
\end{remark}

\subsection{Derivations of $\t_{1,n}^\Gamma$ and $\bar\t_{1,n}^\Gamma$}

\begin{lemma}
There is a bigraded Lie algebra morphism $\tilde\d^\Gamma_0\to\Der(\t_{1,n}^{\Gamma})$, taking 
$(D,C)\in\tilde\d^\Gamma_0$ to the derivation $\xi_{(D,C)}:$
\begin{align*}
x_i&\longmapsto0, \\
y_i&\longmapsto\sum_{j:j\neq i}D(x_i,t^{\beta}_{ij}), \\
t^{\alpha}_{ij}&\longmapsto[t^{\alpha}_{ij},C_\alpha(x_i,t^{\beta}_{ij})].
\end{align*}

This induces a bigraded Lie algebra morphism $\tilde\d^\Gamma_0\to\Der(\bar\t_{1,n}^\Gamma)$. 
\end{lemma}
\begin{proof}
We have to prove that defining relations of $\t_{1,n}^{\Gamma}$ are preserved by 
$\xi:=\xi_{(D,C)}$. First observe that relations 
$[x_i,x_j]=[x_i+x_j,t^{\alpha}_{ij}]=[x_i,t^{\alpha}_{jk}]=[t^{\alpha}_{ij},t^{\alpha}_{kl}]=0$ 
are obviously preserved. 
Then conditions (i) and (ii) respectively imply that $t^{\alpha}_{ij}=t^{-\alpha}_{ji}$ and 
$[x_i,y_j]=\sum_\alpha t^{\alpha}_{ij}$ are preserved. 
Condition (vi) implies that $[x_i,y_j]=[x_j,y_i]$ is preserved, and (vi) together with (iii) imply that 
$[y_i,y_j]=0$ is preserved. 
Therefore it follows from the centrality of $\sum_ix_i$ and $\xi(\sum_ix_i)=0$ that 
$$\xi([x_i,y_i])=\xi(-\sum_{j:j\neq i}[x_j,y_i])=\xi(\sum_{j;j\neq i}\sum_\alpha t^{\alpha}_{ij}).$$
Condition (iv) ensures that $[y_i,t^{\alpha}_{jk}]=0$ is preserved, and together with (vi) it 
implies that $[y_i+y_j,t^{\alpha}_{ij}]=0$ is preserved. 
Finally condition (v) implies that the twisted infinitesimal braid relations are preserved, and 
the first part of the statement follows. 

For the second part of the statement it remains to prove that the centrality of $\sum_iy_i$ is preserved. 
This follows directly from the identity $\xi(\sum_iy_i)=0$ that we now prove. Relation (vi) implies that for any 
$i\neq j$ one has $D(x_i,t_{ij}^\beta)=-D(-x_i,t_{ij}^{-\beta})=-D(x_j,t_{ji}^\beta)$ in $\t_{1,n}^\Gamma$ 
(the last equality happens since $\deg_t(D)=\deg_t(C_\alpha)+1>0$), and hence 
$$
\xi(\sum_iy_i)=\sum_{i\neq j}D(x_i,t_{ij}^\beta)=\sum_{i<j}D(x_i,t_{ij}^\beta)-\sum_{j<i}D(x_j,t_{ji}^\beta)=0\,.
$$
We are done (the compatibility with bracket and grading are easy to check).

The last part of the statementis a consequence of the fact that $\xi(\sum_iy_i)=\xi(\sum_ix_i)=0$, that we have already proved.
\end{proof}
We now prove that this morphism extends to a Lie algebra morphism $\tilde\d^\Gamma\to\Der(\t_{1,n}^\Gamma)$: 
\begin{proposition}\label{deltagamma} 
There is a bigraded Lie algebra morphism $\tilde\d^\Gamma\to\Der(\t_{1,n}^\Gamma)$ taking $(D,C)\in\tilde\d^\Gamma_0$ to 
$\xi_{(D,C)}$ and $g=\begin{pmatrix}a&b\\c&d\end{pmatrix}\in\sl_2$ to the derivation 
$$\xi_g:t_{ij}^\alpha\mapsto0, 
\begin{pmatrix}x_i&y_i\end{pmatrix}\mapsto\begin{pmatrix}x_i&y_i\end{pmatrix}\begin{pmatrix}a&b\\c&d\end{pmatrix}.$$
This induces a bigraded Lie algebra morphism $\tilde\d^\Gamma\to\Der(\bar\t_{1,n}^\Gamma)$. 
\end{proposition}
\noindent In what follows we write $\ddd:=\tilde h$, $\XXX:=\tilde e$ and 
$\Delta_0:=\tilde f$ and $\tilde\ddd:=\xi_{\tilde h}$, $\tilde\XXX:=\xi_{\tilde e}$ and 
$\tilde\Delta_0:=\xi_{\tilde f}$. 
\begin{proof}
It is obvious that for any $g,g'\in\sl_2$, $\xi_g$ defines a derivation of the same degree 
of $\t_{1,n}^\Gamma$, and that $\xi_{[g,g']}=[\xi_g,\xi_{g'}]$. Hence there is a bigraded Lie algebra 
morphism $\sl_2*\tilde\d^\Gamma_0\to\Der(\t_{1,n}^\Gamma)$. Let us prove that it factors through the quotient 
$\tilde\d^\Gamma$. 

It is relatively clear that $[\tilde\XXX,\xi_{(D,C)}]=0$ and $[\tilde\ddd,\xi_{(D,C)}]=(p-q)(D,C)$ if 
$(D,C)\in(\tilde\d^\Gamma_0)^{p,q}$. Thus it remains to prove that $(\on{ad}\tilde\Delta_0)^p(\xi_{(D,C)})=0$ if 
$(D,C)\in(\tilde\d^\Gamma_0)^{p,q}$. We do this now. 
Let us write $\xi:=\xi_{(D,C)}$ and $A:=(\on{ad}\tilde\Delta_0)^p(\xi)$. Then after an easy computation one 
obtains on generators: 
\begin{align*}
A(x_i)=&-p\tilde\Delta_0^{p-1}\xi(y_i)=-p\tilde\Delta_0^{p-1}(\sum_{j:j\neq i}D(x_i,t_{ij}^\beta)), \\
A(y_i)=&\tilde\Delta_0^p\xi(y_i)=\tilde\Delta_0^p(\sum_{j:j\neq i}D(x_i,t_{ij}^\beta)), \\
A(t_{ij}^\alpha)=&\tilde\Delta_0^p\xi(t_{ij}^\alpha)=\tilde\Delta_0^p([t_{ij}^\alpha,C_\alpha(x_i,t_{ij}^\beta)]). 
\end{align*}
Finally remark that there is an increasing filtration on $\t_{1,n}^\Gamma$ defined by $\deg(x_i)=1$ and 
$\deg(t_{ij}^\alpha)=\deg(y_i)=0$. $\Delta_0$ decreases the degree by $1$ and vanishes on degree zero elements. 
The result then follows from the fact that $\deg_x(C_\alpha)=p-q<p$ and $\deg_x(D)=p-q-1<p-1$. 
\end{proof}
Now composing with $\d_0^\Gamma\to\tilde\d_0^\Gamma$ (resp.~$\d^\Gamma\to\tilde\d^\Gamma$) one obtains a Lie algebra 
morphism $\d_0^\Gamma\to\Der(\t_{1,n}^\Gamma)$ (resp.~$\d^\Gamma\to\Der(\t_{1,n}^\Gamma)$). We write 
$\xi_{s,\gamma}:=\xi_{(D_{s,\gamma},C_{s,\gamma})}$ for the image of $\delta_{s,\gamma}$. 
We then have $\t_{1,n}^\Gamma\rtimes\d^\Gamma=(\t_{1,n}^\Gamma\rtimes\d_+^\Gamma)\rtimes\sl_2$, with 
$\t_{1,n}^\Gamma\rtimes\d_+^\Gamma$ positively graded (since both $\t_{1,n}^\Gamma$ and $\d_+^\Gamma$ are 
$(\Z_{\geq0})^2$-graded) and a sum of finite dimensional $\sl_2$-modules. Therefore we can construct the 
semi-direct product group 
\begin{equation}\label{Gngamma}
\text{\gls{Gng}}:=\exp(\t_{1,n}^\Gamma\rtimes\d^\Gamma_+)^{\wedge} \rtimes{\rm SL}_2(\C),
\end{equation}
where $\exp(\t_{1,n}^\Gamma\rtimes\d^\Gamma_+)^{\wedge}$ is the exponential group associated to the degree completion 
of $\t_{1,n}^\Gamma\rtimes\d^\Gamma_+$. 

Similarly, we define $\text{\gls{bGng}}:=\exp(\bar\t_{1,n}^\Gamma\rtimes\d^\Gamma_+)^{\wedge} \rtimes{\rm SL}_2(\C)$. \\

Notice that one can also define semi-direct product groups 
$\tilde{\mathbf{G}}_n^{\Gamma}:=\exp(\t_{1,n}^\Gamma\rtimes\tilde\d^\Gamma_+)^{\wedge} \rtimes{\rm SL}_2(\C)$ and 
$\tilde{\bar{\mathbf{G}}}_n^\Gamma := \exp(\bar\t_{1,n}^\Gamma\rtimes\tilde\d^\Gamma_+)^{\wedge} \rtimes{\rm SL}_2(\C)$. 
We therefore have the following commutative diagram: 
\begin{equation}\xymatrix{
\mathbf{G}_n^\Gamma \ar[d]\ar[r] & \tilde{\mathbf{G}}_n^\Gamma \ar[d] \\
\bar{\mathbf{G}}_n^\Gamma \ar[r] & \tilde{\bar{\mathbf{G}}}_n^\Gamma.
}\end{equation}

\begin{lemma}
The kernel of $\tilde\d^\Gamma_0\to\Der(\t_{1,n}^\Gamma)$ ($n\geq2$) is the space of elements $(0,C)$ for which 
$C_\alpha$ is proportional to $t^\alpha$, and $\ker(\d_0^\Gamma\to\Der(\t_{1,n}^\Gamma))=\C\delta_{0,0}$. 
\end{lemma}
\begin{proof}
Let us first prove it for $n=2$. Recall that 
$\bar \t_{1,2}^\Gamma=\t_{1,2}^\Gamma/(x_1+x_2,y_1+y_2)$, so it is the Lie algebra generated by $x$ 
(the class of $x_1$), $y$ (the class of $y_1$) and $t^\alpha$'s (classes of $t_{12}^\alpha$'s) with 
the relation $[x,y]=\sum_{\alpha\in\Gamma}t^\alpha$. Then the derivation $\xi_{(D,C)}$ associated to 
$(D,C)\in\tilde\d^\Gamma_0$ is given by $$x\mapsto0, y\mapsto D(x,t^\beta), 
t^\alpha\mapsto[t^\alpha,C_\alpha(x,t^\beta)].$$ This derivation vanishes if and only if $D=0$ and 
$C_\alpha$ is proportional to $t^\alpha$. 
Finally, the result for $n\geq2$ follows from the fact that 
$$\xi^{(2)}_{(D,C)}=(u\mapsto u^{1,2,\emptyset,\dots,\emptyset})\circ\xi_{(D,C)}^{(n)}\circ(u\mapsto u^{1,\dots,n}),$$ 
where $\xi_{(D,C)}^{(n)}$ denotes the derivation of $\t_{1,n}^\Gamma$ associated to $(D,C)$. 
\end{proof}

\subsection{Comparison morphisms}

Let $\rho:\Gamma_1\hookrightarrow\Gamma_2$ an injective morphism of abelian groups. There is a comparison morphism 
$\tilde\d_0^{\Gamma_1}\to\tilde\d_0^{\Gamma_2}$, $(D,C)\mapsto(D^\rho,C^\rho)$ defined by 
$$
D^\rho:=D\left(x,\sum_{\gamma\in{\rm coker}(\rho)}t^{\rho(\beta)+\gamma}\right)\,,\,
(C^\rho)_\alpha:=C_\alpha\left(x,\sum_{\gamma\in{\rm coker}(\rho)}t^{\rho(\beta)+\gamma}\right)\,, 
$$
that depends on the choice of a section $\on{coker}(\rho)\to\Gamma_2$. 
It extends to $\tilde\d^{\Gamma_1}\to\tilde\d^{\Gamma_2}$ by sending the generators of $\mathfrak{sl}_2$ to themselves. 
These comparison morphisms are compatible with the morphisms $\tilde\d^{\Gamma_i}\to\Der(\t_{1,n}^{\Gamma_i})$, for $i=1,2$. 
Namely, there is a commutative diagram 
$$\xymatrix{
\tilde\d^{\Gamma_1}\ltimes\t_{1,n}^{\Gamma_1} \ar[d]\ar[r] & \t_{1,n}^{\Gamma_1} \ar[d] \\
\tilde\d^{\Gamma_2}\ltimes\t_{1,n}^{\Gamma_2} \ar[r] & \t_{1,n}^{\Gamma_2}
}$$
where the morphism $\t_{1,n}^{\Gamma_1}\to \t_{1,n}^{\Gamma_2}$ is the one defined by 
$$
x_i\mapsto x_i\,,\quad y_i\mapsto y_i\,,\quad t_{ij}^\beta\mapsto \sum_{\gamma\in{\rm coker}(\rho)}t_{ij}^{\rho(\beta)+\gamma}\,.
$$
This induces comparison morphisms for the corresponding groups, that fit into a commutative diagram 
\begin{equation}
\xymatrix{
\tilde{\mathbf{G}}_n^{\Gamma_1} \ar[d]\ar[r] & \tilde{\mathbf{G}}_n^{\Gamma_2} \ar[d] \\
\tilde{\bar{\mathbf{G}}}_n^{\Gamma_1} \ar[r] & \tilde{\bar{\mathbf{G}}}_n^{\Gamma_2}.
}
\end{equation}

In particular we obtain a canonical natural inclusion $\mathbf{G}_n^0\to\mathbf{G}_n^\Gamma$ (which descends 
to an inclusion $\bar{\mathbf{G}}_n^0\to\bar{\mathbf{G}}_n^\Gamma$), given by the inclusion $0\hookrightarrow\Gamma$.

\section{Bundles with flat connections on moduli spaces}
\label{Bundles with flat connections on moduli spaces}


\subsection{On some subgroups of $\on{SL}_2(\Z)$ and moduli spaces}

Recall that $M,N\geq 1$ are integers, and that $\Gamma:=\Z/M\Z\times\Z/N\Z$. 
We consider the following (finite index) subgroup of $\on{SL}_2(\Z)$:
\[
\text{\gls{SL2g}}:=
\left\{\begin{pmatrix}a&b\\c&d\end{pmatrix}\in\on{SL}_2(\Z)\,\big|\,
a \equiv 1~\mathrm{mod}~M,d \equiv 1~\mathrm{mod}~N, b\equiv0~\mathrm{mod}~N~\mathrm{and}~c\equiv0~\mathrm{mod}~M\right\}\,. 
\]
We define 
\[
Y(\Gamma):=\on{SL}_2^\Gamma(\Z) \backslash\mathfrak H\,,
\]
that has the structure of a complex orbifold, which we call the \textit{moduli of $\Gamma$-structured elliptic curves}. 

\begin{definition}
A $\Gamma$-structure on an elliptic curve $E$ is an injective group morphism $\phi:\Gamma\hookrightarrow E$. 
An equivalence between two $\Gamma$-structured elliptic curves, $(E,\phi)$ and $(E',\phi')$, is an equivalence 
$\psi:E\tilde\to E'$ such that $\psi\circ\phi=\phi'$
\end{definition}

On the one hand, every $\Gamma$-structured elliptic curve is equivalent to one that is in the following standard form: 
\begin{itemize}
\item the elliptic curve is $E=E_\tau$, with $\tau\in\mathfrak H$;  
\item the injective group morphism $\phi=\phi_\tau$ sends $(\bar{a},\bar{b})$ to the class of 
$\frac{a}{M}+\frac{b\tau}{N}\in\mathbb{C}$.  
\end{itemize}
On the other hand, an equivalence $E_{\tau_1}\tilde\to E_{\tau_2}$, determined by an element $g\in \on{SL}_2(\Z)$, intertwines 
the standard $\Gamma$-structures if and only if $g$ belongs to the congruence subgroup $\on{SL}_2^\Gamma(\Z)$. 

\begin{remark}
The biggest congruence subgroup on which the connection we will construct in this 
section is well defined and flat is the subgroup $\on{\tilde{SL}}^\Gamma_2(\Z)$ of 
$\on{SL}_2(\Z)$ consisting of matrices $\begin{pmatrix}a&b\\c&d\end{pmatrix}\in\on{SL}_2(\Z)$ 
such that $Mb\equiv0$ mod $N$ and $Nc\equiv0$ mod $M$. Nevertheless, in order to retrieve 
the twisted elliptic KZB connection defined at the level of configuration spaces, it suffices 
to consider the usual congruence subgroup $\on{SL}^\Gamma_2(\Z)\subset \on{\tilde{SL}}^\Gamma_2(\Z)$. 
\end{remark}

Recall the following standard group actions: 
\begin{itemize}
\item The group $\on{SL}_2(\Z)$ acts on the left on $\mathbb{C}^n\times\mathfrak{H}$: 
\[
\begin{pmatrix}a&b\\c&d\end{pmatrix}*(\zz|\tau):=
\left(\frac{\zz}{c\tau+d}\big|\frac{a\tau+b}{c\tau+d}\right)\,.
\]
This obviously descends to a left action of $\on{SL}_2(\Z)$ 
on $(\mathbb{C}^n\times\mathfrak{H})/\mathbb{C}$, where 
$\mathbb{C}$ acts diagonally on $\mathbb{C}^n$: 
$u\cdot(\zz|\tau):=(\zz+u\sum_i\delta_i|\tau)$. 
\item The group $(\Z^n)^2$ acts on the left on $\mathbb{C}^n\times\mathfrak{H}$:
\[
(\m,\n)*(\zz|\tau):=(\zz+\m+\tau\n|\tau)\,.
\]
It obvioulsy descends to a left action of $(\Z^n)^2/\Z^2$ 
on $\mathbb{C}^n\times\mathfrak{H}/\mathbb{C}$, where $\Z^2$ is 
the diagonal subgroup in $(\Z^n)^2=(\Z^2)^n$. 
\item Finally, there is a right action of $\on{SL}_2(\Z)$ on $(m,n)\in\mathbb{Z}^2$ by automorphisms: 
\[
\begin{pmatrix}a&b\\c&d\end{pmatrix} : \begin{pmatrix}n&m\end{pmatrix} \to \begin{pmatrix}n&m\end{pmatrix}\begin{pmatrix}a&b\\c&d\end{pmatrix}.
\]
We can thus form the semi-direct products $(\Z^n)^2\rtimes\on{SL}_2(\Z)$
and $((\Z^n)^2/\Z^2)\rtimes\on{SL}_2(\Z)$.
\end{itemize}

A few observations are then in order: 
\begin{itemize}
\item The above actions are compatible in the sense that there is a left action of 
$(\Z^n)^2\rtimes\on{SL}_2(\Z)$ on $\mathbb{C}^n\times\mathfrak{H}$, which 
descends to an action of $\big((\Z^n)^2/\Z^2\big)\rtimes\on{SL}_2(\Z)$ on 
$(\mathbb{C}^n\times\mathfrak{H})/\mathbb{C}$, where $\Z^2$ is embedded in $(\Z^n)^2$ via the diagonal map. 
One can think of translation by $\C$ as a left or right action as it commutes with the $((\Z^n)^2 \rtimes \on{SL}_2(\Z))$-action.
\item The action of $(\Z^n)^2$ preserves the subset 
\[
\mathrm{Diag}_{n,\Gamma}:=
\{(\zz|\tau)\in\C^n\times{\mathfrak H}|\zz\in\mathrm{Diag}_{\tau,n,\Gamma}\}\,.
\]
\item The action of the subgroup $\on{SL}^\Gamma_2(\Z)\subset \on{SL}_2(\Z)$ 
also preserves $\mathrm{Diag}_{n,\Gamma}$. 
\end{itemize}

We are thus ready to define several variants of $Y(\Gamma)$ ``with marked points'':  
\begin{itemize}
\item We define the quotient 
\[
\text{\gls{bM1ng}}:=(\Z^n)^2\rtimes\on{SL}^\Gamma_2(\Z) \backslash \big((\C^n\times{\mathfrak H})
-{\rm Diag}_{n,\Gamma}\big)/\mathbb{C}
\]
and call it the {\it moduli space of $\Gamma$-structured elliptic curves with $n$ ordered marked points}. 
\item It has a non-reduced variant 
\[
\mathbf{p}:\text{\gls{M1ng}}:=
(\Z^n)^2\rtimes\on{SL}^\Gamma_2(\Z) \backslash \big((\C^n\times{\mathfrak H})
-{\rm Diag}_{n,\Gamma}\big)
\twoheadrightarrow\bar\cM_{1,n}^\Gamma\,.
\]
\item One can also define the {\it moduli space of $\Gamma$-structured elliptic curves with $n$ unordered marked points} 
\[
\text{\gls{bM1nng}}:=\mathfrak{S}_n \backslash\bar\cM_{1,n}^\Gamma
\]
and its non-reduced variant 
\[
\text{\gls{M1nng}}:=\mathfrak{S}_n \backslash\cM_{1,n}^\Gamma\,.
\]
\end{itemize}
\begin{remark}
We have $\bar\cM_{1,1}^\Gamma=\bar\cM_{1,[1]}^\Gamma=Y(\Gamma)$, and 
$\cM_{1,1}^\Gamma=\cM_{1,[1]}^\Gamma$ is the universal curve over it. 
The fiber of $\cM_{1,n}^\Gamma\to Y(\Gamma)$ (resp.~$\bar\cM_{1,n}^\Gamma\to Y(\Gamma)$) 
at (the class of) $\tau$ is precisely the twisted (resp.~reduced twisted) configuration space 
$\textrm{Conf}(E_{\tau,\Gamma},n,\Gamma)$ (resp.~$\textrm{C}(E_{\tau,\Gamma},n,\Gamma)$). 
Moreover, the map
\[
h:\bar\cM_{1,2}^\Gamma \longrightarrow \bar\cM_{1,1}^\Gamma
\]
factors through (and is open in) $\cM_{1,1}^\Gamma$. 
We can interpret $\bar\cM_{1,2}^\Gamma $ as the $\Gamma$-punctured 
universal curve over $Y(\Gamma)$.
\end{remark}


\subsection{Principal bundles over $\cM_{1,n}^\Gamma$ and $\bar\cM_{1,n}^\Gamma$}

In this paragraph, $\mathbf{G}^\Gamma_n$ is defined as in \eqref{Gngamma} and we define a principal 
$\mathbf{G}^\Gamma_n$-bundle \gls{Png} over $\cM_{1,n}^\Gamma$ whose 
image under the natural morphism $\mathbf{G}^\Gamma_n\to\bar{\mathbf{G}}^\Gamma_n$ is the pull-back of a principal 
$\bar{\mathbf{G}}^\Gamma_n$-bundle \gls{bPng} over $\bar\cM_{1,n}^\Gamma$. Let us fix the notation first: 
for $u\in\C^\times$ and $v,w_i\in\C$ ($i=1,\dots,n$), 
$$
u^{\ddd}:=\begin{pmatrix}u&0\\0&u^{-1}\end{pmatrix}, 
e^{v\XXX}:=\begin{pmatrix}1&v\\0&1\end{pmatrix}.
$$
Since $[\XXX,x_i]=0$ then it makes sense to define 
$e^{v\XXX+\sum_iw_ix_i}:=e^{v\XXX}e^{\sum_iw_ix_i}$. 
In particular, we have $\Ad(u^\ddd)(x_i)=ux_i$ and $\Ad(u^\ddd)(y_i)=y_i/u$ ($\forall i$), 
$\Ad(u^\ddd)(\XXX)=u^2\XXX$ and $\Ad(u^\ddd)(\Delta_0)=\Delta_0/u^2$. Let $\pi:\C^n\times\mathfrak H\to\cM_{1,n}$ be the canonical 
projection.
\begin{proposition}\label{prop:bundle}
There exists a unique principal $\mathbf{G}^\Gamma_n$-bundle $\mP_{n,\Gamma}$ over $\cM_{1,n}^\Gamma$ for which a section 
on $U\subset\cM_{1,n}^\Gamma$ is a function $f:\pi^{-1}(U)\to\mathbf{G}^\Gamma_n$ such that 
\begin{align*}
f(\mathbf{z}+\delta_i|\tau)&=f(\mathbf{z}|\tau), \\
f(\mathbf{z}+\tau\delta_i|\tau)&=e^{-2\pi{\rm i}x_i }f(\mathbf{z}|\tau), \\
f(\zz,\tau+1)&=f(\zz|\tau),\\
f({\zz\over\tau}|-{1\over\tau})&=\tau^{\ddd}e^{{2\pi\i\over\tau}(\XXX+\sum_iz_ix_i)}f(\zz|\tau). 
\end{align*}

Moreover, the image of $\mP_{n,\Gamma}$ under $\mathbf{G}^\Gamma_n\to\bar{\mathbf{G}}^\Gamma_n$ is the pull-back of a unique 
principal $\bar{\mathbf{G}}^\Gamma_n$-bundle $\bar\mP_{n,\Gamma}$ over $\bar\cM_{1,n}^\Gamma$ for which a section on 
$U\subset\bar\cM_{1,n}^\Gamma$ is a function ${f:(\mathbf{p}\circ\pi)^{-1}(U)\to\bar\cM_{1,n}^\Gamma}$ 
satisfying the above conditions (with $x_i$'s replaced by $\bar x_i$'s) and such that 
$f(\zz+v\sum_i\delta_i|\tau)=f(\zz|\tau)$ for any $v\in\C$. 
\end{proposition}
\begin{proof}
First recall that for $\Gamma=0$ this is precisely \cite[Proposition 3.4]{CEE}. 
Then observe that there is an obvious map $\iota:\cM_{1,n}^\Gamma\to\cM_{1,n}^0$. 
Therefore we define $\mP_{n,\Gamma}$ (resp.~$\bar\mP_{n,\Gamma})$ to be the image under 
the natural inclusion $\mathbf{G}_n^0\to\mathbf{G}_n^\Gamma$ (resp.~$\bar{\mathbf{G}}_n^0\to\bar{\mathbf{G}}_n^\Gamma$) 
of $\iota^*\mP_{n,0}$ (resp.~$\iota^*\bar\mP_{n,0}$). 

We thus proved existence. Unicity is obvious. 
\end{proof}
In other words, there exists a unique non-abelian 1-cocycle $(c_g)_{g\in(\Z^{n})^{2}\rtimes\on{SL}_{2}(\Z)}$ on 
$\C^n\times\mathfrak H$ with values in $\mathbf{G}^\Gamma_{n}$ such that $c_{(\delta_i,0)}=1$, 
$c_{(0,\delta_i)}=e^{-2\pi\i x_i}$, $c_S=1$ and 
$$
c_T(\zz|\tau)=\tau^\ddd e^{(2\pi\i/\tau)(\XXX+\sum_jz_jx_j)}
=e^{2\pi\i(\tau\XXX+\sum_jz_jx_j)}\tau^\ddd\,,
$$
where $S=\begin{pmatrix}1&1\\0&1\end{pmatrix}$ and $T=\begin{pmatrix}0&-1\\1&0\end{pmatrix}$ are 
the generators of $\on{SL}_2(\Z)$. 
Here {\em cocycle} means (as in \cite{CEE}) that $c_g$'s are holomorphic functions 
$\C^n\times\mathfrak H\to\mathbf{G}_n^\Gamma$ satisfying the cocycle condition 
$c_{gg'}(\zz|\tau)=c_g(g'*(\zz,\tau))c_{g'}(\zz|\tau)$. 
\begin{remark}
Notice that we do have a $(\Z^{n})^{2}\rtimes\on{SL}_{2}(\Z)$-cocycle (since our bundle is 
define as the pull-back of a bundle on $\cM^0_{1,1}$) but the cocycle defining $\mP_{n,\Gamma}$ 
is its restriction to $(\Z^{n})^{2}\rtimes\on{SL}^\Gamma_{2}(\Z)$. 
\end{remark}


\subsection{Connections on $\mP_{n,\Gamma}$ and $\bar\mP_{n,\Gamma}$}

A connection on $\mP_{n,\Gamma}$ is the same as an equivariant connection on the trivial 
$\mathbf{G}_n^\Gamma$-bundle over $\C^n\times\mathfrak H-{\rm Diag}_{n,\Gamma}$. Namely, it is 
of the form $\nabla_{n,\Gamma}:=d-\eta(\zz|\tau)$, where $\eta$ is a $\t_{1,n}^\Gamma\rtimes\d^\Gamma$-valued 
meromorphic one-form on $\C^\n\times\mathfrak H$ with only poles on ${\rm Diag}_{n,\Gamma}$, 
and the equivariance condition reads: for any $g\in(\Z^n)^2\rtimes{\rm SL}^\Gamma_2(\Z)$, 
\begin{equation}\label{eq:equiv}
g^*\eta=(dc_g(\zz|\tau))c_g(\zz|\tau)^{-1}+\Ad(c_g(\zz|\tau))(\eta(\zz|\tau))\,.
\end{equation}

We now construct such a connection. For any $\gamma\in\Gamma$, we define
$g_\gamma(x,z|\tau):=\partial_xk_\gamma(x,z|\tau)$ and
$$
\tilde\varphi_\gamma(x|\tau)=
\sum_{s\geq 0}A_{s,\gamma}(\tau)x^s:=g_{-\gamma}(x,0|\tau)\,. 
$$
Then we set 
$$
\Delta(\zz|\tau):=-\frac1{2\pi\i}\left(\Delta_0+\frac12\sum_{s\geq 0,\gamma\in\Gamma}A_{s,\gamma}(\tau)\delta_{s,\gamma}
-\sum_{i<j}g_{ij}(z_{ij}|\tau)\right)\,, 
$$ 
where $g_{ij}(z|\tau):=\sum_{\alpha\in\Gamma}g_\alpha(\ad x_i,z|\tau)(t^{\alpha}_{ij})$. 
And finally, with $K_i(\zz|\tau)$'s as in \S\ref{sec:flatconn}, we define 
$$\eta(\zz|\tau):=\Delta(\zz|\tau)d\tau+\sum_iK_i(\zz|\tau)dz_i.$$ 
\begin{remark}
One can see that $\tilde{\varphi}_\0(x)=(\theta'/\theta)'(x)+1/x^2$ and that for any $\gamma\in\Gamma-\{0\}$ 
$$
\tilde{\varphi}_\gamma(x)=\partial_x\left(e^{2\pi\i cx}{{\theta(\tilde\gamma+x)}\over{\theta(\tilde\gamma)\theta(x)}}
-{1\over x}\right)\,,
$$
where $\tilde\gamma=(c_0,c)\in\Lambda_{\tau,\Gamma}-\Lambda_\tau$ is any lift of $\gamma$. 
\end{remark}
\begin{proposition}\label{prop:equiv}
The equivariance identity \eqref{eq:equiv} is satisfied for any 
$g\in(\Z^n)^2\rtimes{\rm SL}_2(\Z)$.
\end{proposition}
Before proving this statement, let us notice that the $\on{SL}_2(\Z)$-equivariance is stronger than what we need 
(the $\on{SL}_2^\Gamma(\Z)$-equivariance), but easier to prove. The action of ${\rm SL}_2(\Z)$ moves the poles 
while $\on{SL}_2^\Gamma(\Z)$ fixes them. In both cases, it makes sense to prove this proposition for meromorphic forms on $\C^n\times \h$.
\begin{proof}
For $g=(\delta_j,0)$, the identity translates into $K_i(\zz+\delta_j|\tau)=K_i(\zz|\tau)$ ($i=1,\dots,n$) 
and $\Delta(\zz+\delta_j|\tau)=\Delta(\zz|\tau)$, which are immediate. 

For $g=(0,\delta_j)$, the identity translates into $K_i(\zz+\tau\delta_j|\tau)=e^{-2\pi\i\ad(x_j)}K_i(\zz|\tau)$ 
($\forall i$) and 
\begin{equation}\label{equivD}
\Delta(\zz+\tau\delta_j|\tau)+K_j(\zz+\tau\delta_j|\tau)=e^{-2\pi\i\ad(x_j)}\Delta(\zz|\tau).
\end{equation}
The first equality is proved in \S\ref{sec:flatconn}, and we prove the second one now. 
First remember that for any $\tau\in\mathfrak H$, $z\in\C-(\frac1M\Z+\frac\tau N\Z))$ and $\alpha\in\Gamma$, we have 
the following identity in $\C[[x]]$: 
\begin{equation}\label{eq:gktheta}
e^{-2\pi\i x}(g_\alpha(x,z)-1/x^2)+1/x^2-2\pi\i(k_\alpha(x,z+\tau)+1/x)=g_\alpha(x,z+\tau)\,.
\end{equation}
Then, we can compute $2\pi\i\left(K_j(\zz+\tau\delta_j|\tau)-e^{-2\pi\i\ad(x_j)}\Delta(\zz|\tau)\right)$: 
it is equal to 
$$
2\pi\i\left(\sum_{k:k\neq j}k_\alpha(\ad x_j,z_{jk}+\tau)-y_j\right)+\Delta_0
+\frac{1-e^{-2\pi\i\ad x_j}}{\ad x_j}(y_j)
+\frac12\sum_{\substack{s\geq 0,\\\gamma\in\Gamma}}A_{s,\gamma}\delta_{s,\gamma}
-e^{-2\pi\i\ad x_j}\sum_{k<l}g_{kl}(z_{kl})\,,
$$
and, therefore, using 
$$
\frac{1-e^{-2\pi\i\ad x_j}}{\ad x_j}(y_j)-2\pi\i y_j
=\left(\frac{e^{-2\pi\i\ad x_j}-1}{(\ad x_j)^2}+\frac{2\pi\i}{\ad x_j}\right)
\left(\sum_{\alpha\in\Gamma}\sum_{k:k\neq j}t_{jk}^\alpha\right),
$$
together with \eqref{eq:gktheta}, we obtain 
$$
\Delta_0+\frac12\sum_{s\geq 0,\gamma\in\Gamma}A_{s,\gamma}\delta_{s,\gamma}
-\sum_{\substack{k<l\\ k,l\neq j}}g_{kl}(z_{kl})
-\sum_{\substack{k:k\neq j\\\alpha\in\Gamma}}g_\alpha(\ad x_j,z_{jk}+\tau)(t^\alpha_{jk})\,,
$$
which is precisely equal to $-2\pi\i\Delta(\zz+\tau\delta_j)$. 

For $g=S$, the identity translates into $K_i(\zz|\tau+1)=K_i(\zz)$ ($\forall i$) and 
$\Delta(\zz|\tau+1)=\Delta(\zz)$. Both equalities obviously follow from $\theta(z|\tau+1)=\theta(z|\tau)$. 

For $g=T$, the identity translates into 
\begin{equation}\label{eq:g=T-K}
\frac1\tau K_i(\frac\zz\tau|-\frac1\tau)=\Ad\left(c_T(\zz|\tau)\right)(K_i(\zz|\tau))+2\pi\i x_i,
\end{equation}
for all $i\in\{1,\dots,n\}$ and 
\begin{equation}\label{eq:g=T-Delta}
\frac1{\tau^2}\left(\Delta(\frac\zz\tau|-\frac1\tau)-\sum_iz_iK_i(\frac\zz\tau|-\frac1\tau)\right)
=\Ad\left(c_T(\zz|\tau)\right)(\Delta(\zz|\tau))+\frac\ddd\tau-2\pi\i\XXX\,.
\end{equation}
Let us check \eqref{eq:g=T-K} first. 
$\Ad(e^{2\pi\i(\sum_jz_jx_j+\tau\XXX)}\tau^\ddd)(-y_i)+2\pi\i x_i$ equals 
$$
-\Ad(e^{2\pi\i\sum_jz_jx_j})(y_i/\tau)
=-\frac{y_i}{\tau}-\frac{e^{2\pi\i\ad(\sum_jz_jx_j)}-1}{\ad(\sum_jz_jx_j)}([\sum_jz_jx_j,\frac{y_i}{\tau}])
$$
$$
=-\frac{y_i}{\tau}-\frac{e^{2\pi\i\sum_jz_j\ad x_j}-1}{\sum_jz_j\ad x_j}
(\sum_{\substack{j:j\neq i\\\alpha\in\Gamma}}\frac{z_{ji}}{\tau}t_{ij}^\alpha)
=-\frac{y_i}{\tau}-\sum_{j:j\neq i}\frac{e^{2\pi\i z_{ij}\ad x_i}}{z_{ij}\ad x_i}
(\sum_{\alpha\in\Gamma}\frac{z_{ji}}{\tau}t_{ij}^\alpha)\,.
$$
Therefore  
\begin{equation}\label{eq:y_i/tau}
-\frac{y_i}{\tau}=\Ad(c_T(\zz|\tau))(-y_i)+2\pi\i x_i
-\sum_{j:j\neq i}\frac{e^{2\pi\i z_{ij}\ad x_i }}{\ad x_i}(\sum_{\alpha\in\Gamma}\frac{t_{ij}^\alpha}{\tau})\,.
\end{equation}
Now, since 
\[ \theta \left( - \frac{z}{\tau}  \middle| - \frac{1}{\tau} \right) = -
   \frac{1}{\tau} e^{\frac{\pi \i}{\tau} z^2} \theta (z | \tau), \]
we obtain
\begin{eqnarray*}
  k_{\alpha} \left( x, \frac{z}{\tau}  \middle| - \frac{1}{\tau} \right) & = &
  e^{- 2 \pi {\i}ax}  \frac{\theta \left( \dfrac{z}{\tau} - \left( a_0 -
  \dfrac{a}{\tau} \right) + x \middle| \tau \right)}{\theta \left(
  \dfrac{z}{\tau} - \left( a_0 - \dfrac{a}{\tau} \right)  \middle| \tau
  \right) \theta (x| \tau)} - \frac{1}{x}\\
  & = & - \tau e^{2 \pi {\i}z x - 2 \pi {\i}a_0 \tau x} \frac{\theta (\tau x
  + z + a - \tau a_0  | \tau)}{\theta (z + a - \tau a_0  | \tau) \theta (\tau
  x| \tau)} - \frac{1}{x}\\
  & = & \tau e^{2 \pi {\i}z x} k_{T \alpha} (\tau x, z| \tau) +
  \dfrac{e^{2 \pi \i z x} - 1}{\tau x},
\end{eqnarray*}
where $T (\bar a_0,\bar a)=(-\bar a,\bar a_0)$.
Now substituting $(x,z)=(\ad x_j,z_j)$ in
\begin{equation}\label{eq:modulark}
\frac1\tau(k_\alpha(x,\frac z\tau|-\frac1\tau)=e^{2\pi\i zx}k_{T\alpha}(\tau x,z|\tau)+\frac{e^{2\pi\i zx}-1}{\tau x}\,, 
\end{equation}
then applying to $t_{ij}^\alpha$, summing over $j\neq i$ and $\alpha\in\Gamma$ and adding up \eqref{eq:y_i/tau}, we 
obtain \eqref{eq:g=T-K} by using that $$e^{2\pi\i z_{ij}\ad x_i}k_\alpha(\tau\ad x_i,z_{ij}|\tau)(t_{ij}^\alpha)
=\Ad(e^{2\pi\i(\tau\XXX+\sum_jz_jx_j)}\tau^\ddd)(k_\alpha(\ad x_i,z_{ij}|\tau)(t_{ij}^\alpha)).$$ \\
We now check \eqref{eq:g=T-Delta}. Differentiating \eqref{eq:modulark} w.r.t.~$x$ and dividing by $\tau$, we get 
$$
\frac1{\tau^2}g_\alpha(x,\frac z\tau|-\frac1\tau)=e^{2\pi\i zx}g_{T \alpha}(\tau x,z|\tau)
+\frac{2\pi\i z}{\tau^2}k_\alpha(x,\frac z\tau|-\frac1\tau)+\frac{1+2\pi\i zx-e^{2\pi\i zx}}{\tau^2x^2}\,.
$$
Now substituting $(x,z)=(\ad x_i,z_{ij})$, applying to $t_{ij}^\alpha$, and summing over $\alpha\in\Gamma$ we obtain 
\begin{eqnarray*}
\frac1{\tau^2}g_{ij}(\frac\zz\tau|-\frac1\tau)
& = & \Ad(c_T(\zz|\tau))\left(g_{ij}(\zz|\tau)\right)+\frac{2\pi\i z_{ij}}{\tau^2}K_{ij}(\frac{z_{ij}}{\tau}|-\frac1\tau) \\
&& +\left(\frac{1+2\pi\i z_{ij}\ad x_i-e^{2\pi\i z_{ij}\ad x_i}}{\tau^2(\ad x_i)^2}\right)
(\sum_{\alpha\in\Gamma} t_{ij}^\alpha)\,.
\end{eqnarray*}
Then taking the sum over $i<j$ one gets 
\begin{equation}\label{eq:gcgB}
\frac1{\tau^2}\sum_{i<j}g_{ij}(\frac\zz\tau|-\frac1\tau)=\Ad(c_T(\zz|\tau))\left(\sum_{i<j}g_{ij}(\zz|\tau)\right)
+\frac{2\pi\i}{\tau^2}\sum_iz_iK_i(\frac{\zz}{\tau}|-\frac1\tau)+B(\zz)\,,
\end{equation}
where 
$$B(\zz):=\sum_i\frac{2\pi\i z_iy_i}{\tau^2}
+\sum_{i<j}\left(\frac{1+2\pi\i z_{ij}\ad x_i-e^{2\pi\i z_{ij}\ad x_i}}{\tau^2(\ad x_i)^2}\right)
(\sum_\alpha t_{ij}^\alpha).$$
\begin{lemma}
$\Ad\left(c_T(\zz|\tau)\right)(\Delta_0)
=\frac{\Delta_0}{\tau^2}+\frac{2\pi\i\ddd}{\tau}-(2\pi\i)^2(\frac1{\tau}\sum_iz_ix_i+\XXX)+B(\zz)$. 
\end{lemma}
\begin{proof}[Proof of the lemma]
We first compute 
\begin{eqnarray*}
\Ad\left(c_T(\zz|\tau)\right)(\Delta_0)
& = & \Ad(e^{2\pi\i(\tau\XXX+\sum_iz_ix_i)})(\frac{\Delta_0}{\tau^2})
=\Ad(e^{2\pi\i\sum_iz_ix_i})(\frac{\Delta_0}{\tau^2}+\frac{2\pi\i\ddd}{\tau}-(2\pi\i)^2\XXX) \\
& = & \Ad(e^{2\pi\i\sum_iz_ix_i})(\frac{\Delta_0}{\tau^2})+\frac{2\pi\i\ddd}{\tau}
-(2\pi\i)^2(\frac1{\tau}\sum_iz_ix_i+\XXX)\,.
\end{eqnarray*}
It remains to show that $\Ad(e^{2\pi\i\sum_iz_ix_i})(\frac{\Delta_0}{\tau^2})=\frac{\Delta_0}{\tau^2}+B(\zz)$. 
The proof of this fact goes along the same lines of computation as in \cite[pp.16-17]{CEE}. 
\end{proof}
\noindent Using the above lemma and equation \eqref{eq:gcgB}, one sees that equation \eqref{eq:g=T-Delta} follows from 
$$
\Ad(c_T(\zz|\tau)(\sum_{s,\gamma}A_{s,\gamma}(\tau)\delta_{s,\gamma})
=\sum_{s,\gamma}A_{s,\gamma}(-\frac1\tau)\delta_{s,\gamma}\,.
$$
This last equality is proved using $[x_i,\delta_{s,\gamma}]=0=[\XXX,\delta_{s,\gamma}]$, 
$[\ddd,\delta_{s,\gamma}]=s\delta_{s,\gamma}$, and, since $$\tilde{\varphi}_\gamma(x|-\frac1\tau)=\tau^2\tilde{\varphi}_{T\gamma}(\tau x|\tau),$$ we get 
$A_{s,\gamma}(-\frac1\tau)=\tau^{s+2}A_{s,T \gamma}(\tau)$. 
\end{proof}
We therefore have: 
\begin{theorem}
$\nabla_{n,\Gamma}$ defines a connection on $\mP_{n,\Gamma}$. Moreover, its image under 
$\mathbf{G}_n^\Gamma\to\bar{\mathbf{G}}_n^\Gamma$ is the pull-back of a connection $\bar\nabla_{n,\Gamma}$ on 
$\bar\mP_{n,\Gamma}$. 
\end{theorem}
\begin{proof}
The first part follows from Proposition \ref{prop:equiv} above. For the second part, we need to prove the three following identities: 
\begin{itemize}
\item $\sum_i\bar K_i(\zz|\tau)=0$; 
\item $\bar K_i(\zz+u\sum_j\delta_j|\tau)=\bar K_i(\zz|\tau)$, for all $i$; 
\item $\bar\Delta(\zz+u\sum_j\delta_j|\tau)=\bar\Delta(\zz|\tau)$. 
\end{itemize}
The first two equalities have already been proven, and the last one is obvious. 
\end{proof}


\subsection{Flatness}

In this paragraph we prove the flatness of $\nabla_{n,\Gamma}$ (and thus of $\bar\nabla_{n,\Gamma}$). 
\begin{proposition}\label{prop:flatness2}
For any $i\in\{1,\dots,n\}$, $[\partial_\tau-\Delta(\zz|\tau),\partial_i-K_i(\zz|\tau)]=0$. 
\end{proposition}
\noindent In what follows, we often drop $\tau$ from the notation when it does not lead to any confusion. 
\begin{proof}
Let us first prove that $\partial_\tau K_i(\zz)=\partial_i\Delta(\zz)$. 
This follows from the identity $\partial_zg_\alpha(x,z)=2\pi\i\partial_\tau k_\alpha(x,z)$, 
which is proved as follows (here $\tilde\alpha=(a_0,a)$ is any lift of $\alpha$): 
\begin{eqnarray*}
& \partial_zg_\alpha(x,z)&=\partial_z\partial_x k_\alpha(x,z)=\partial_z\partial_x
\left(e^{-2\pi\i ax}k(x,z-\tilde\alpha)+{{e^{-2\pi\i ax}-1}\over x}\right) \\
&& =e^{-2\pi\i ax}\partial_z\partial_x k(x,z-\tilde\alpha)-2\pi\i ae^{-2\pi\i ax}\partial_z k(x,z-\tilde\alpha) \\
&& =2\pi\i e^{-2\pi\i ax}\partial_\tau k(x,z-\tilde\alpha)-2\pi\i ae^{-2\pi\i ax} \partial_z k(x,z-\tilde\alpha) \\
&& =2\pi\i\partial_\tau\left(e^{-2\pi\i ax}k(x,z-\tilde\alpha)\right)=2\pi\i\partial_\tau k_\alpha(x,z). 
\end{eqnarray*}

It remains to prove that $[\Delta(\zz),K_i(\zz)] = 0$. 

Let us first prove it in the case $n=2$. Namely, we will prove that 
\begin{equation}\label{eq-tautwisted}
[\Delta_0 + \frac12\sum_{s\geq0,\gamma\in\Gamma}A_{s,\gamma}\delta_{s,\gamma}
-\sum_{\alpha\in\Gamma}g_\alpha(\on{ad}x_1,z)(t^{\alpha}_{12})\,,
\,y_2+\sum_{\beta\in\Gamma}k_\beta(\on{ad}x_1,z)(t^{\beta}_{12})]=0.
\end{equation}
One the one hand, 
\begin{eqnarray*}
&& [\Delta_0+\frac12\sum_{s\geq0,\gamma\in\Gamma}A_{s,\gamma}\delta_{s,\gamma}
-\sum_{\alpha\in\Gamma}g_\alpha(\on{ad}x_1,z)(t^{\alpha}_{12})\,,\,y_2] \\
&& =[y_1,\sum_{\alpha\in\Gamma}g_\alpha(\on{ad}x_1,z)(t^{\alpha}_{12})]
-\frac12\sum_{\alpha,\gamma\in\Gamma}\sum_{p,q}a_{p,q}^\gamma
[\on{ad}^px_1(t^{\alpha-\gamma}_{12}),\on{ad}^qx_1(t^{\alpha}_{12})]\,,
\end{eqnarray*}
where 
$$
\frac{\tilde{\varphi}_\gamma(u)-\tilde{\varphi}_{-\gamma}(v)}{u+v}=\sum_{p,q}a_{p,q}^\gamma u^pv^q\,.
$$ 
On the other hand, 
\begin{align*}
[\Delta_0,\sum_\beta k_\beta(\on{ad}x_1,z)(t^{\beta}_{12})]
=& [y_1,\sum_\beta g_\beta(\on{ad}x_1,z)(t^{\beta}_{12})]\\
& +\sum_{p,q}\sum_{\alpha,\beta\in\Gamma}b^{\alpha,\beta}_{p,q}(z)
[\on{ad}^px_1(t^{\alpha}_{12}),\on{ad}^qx_1(t^{\beta}_{12})]\,,
\end{align*}

where the series $\sum_{p,q}b^{\alpha,\beta}_{p,q}(z)u^pv^q$ is given by 
$$
\frac12\left(\frac1{v^2}\left(k_\beta(u+v,z)-k_\beta(u,z)-v\partial_uk_\beta(u,z)\right)
-\frac1{u^2}\left(k_\alpha(u+v,z)-k_\alpha(v,z)-u\partial_vk_\alpha(v,z)\right)\right)\,.
$$
Therefore the l.h.s.~of (\ref{eq-tautwisted}) equals 
$$
\frac12\left(
\sum_{p,q}\sum_{\alpha,\beta\in\Gamma}c^{\alpha,\beta}_{p,q}(z)
[\on{ad}^px_1(t^{\alpha}_{12}),\on{ad}^qx_1(t^{\beta}_{12})]\right)\,,
$$
where $\sum_{p,q}c_{p,q}^{\alpha,\beta}u^pv^q(z)$ is given by 
\begin{eqnarray*}
\frac1{v^2}\left(k_\beta(u+v,z)-k_\beta(u,z)-vg_\beta(u,z)\right)
-\frac1{u^2}\left(k_\alpha(u+v,z)-k_\alpha(v,z)-ug_\alpha(v,z)\right) && \\
+\frac{\tilde{\varphi}_{\beta-\alpha}(u)-\tilde{\varphi}_{\alpha-\beta}(v)}{u+v}
+k_\alpha(u+v,z)\tilde{\varphi}_{\alpha-\beta}(v)-k_\beta(u+v,z)\tilde{\varphi}_{\beta-\alpha}(u) && \\
+k_\beta(u,z)g_\alpha(v,z)-g_\beta(u,z)k_\alpha(v,z)\,,&&
\end{eqnarray*}
which can be rewritten as 
\begin{eqnarray}
\left(g_{\beta-\alpha}(u,z-z')-\frac1{u^2}\right)\left(k_\alpha(u+v,z')+\frac1{u+v}\right)
-\left(g_{\alpha-\beta}(v,z'-z)-\frac1{v^2}\right)\left(k_\beta(u+v,z)+\frac1{u+v}\right) && \nonumber \\
+\left(g_\alpha(v,z')-\frac1{v^2}\right)\left(k_\beta(u,z)+\frac1u\right)
-\left(g_\beta(u,z)-\frac1{u^2}\right)\left(k_\alpha(v,z')+\frac1v\right), \label{eq-utile} && 
\end{eqnarray}
with $z=z'$. 
Thus, to end the proof of equation (\ref{eq-tautwisted}), the following lemma is sufficient: 
\begin{lemma}\label{lem-utile}
Expression (\ref{eq-utile}) equals zero. 
\end{lemma}
\begin{proof}[Proof of the lemma.]
The case $\alpha=\beta=0$ follows from an explicit computation. Then we choose lifts $\tilde\alpha=(a_0,a)$ 
and $\tilde\beta=(b_0,b)$ of $\alpha$ and $\beta$, respectively. One has 
\begin{eqnarray*}
&& k_\alpha(x,z)+1/x=e^{-2{\rm i}\pi ax}\left(k(x,z-\tilde\alpha)+1/x\right) \quad \textrm{and} \\
&& g_\alpha(x,z)-1/x^2=e^{-2{\rm i}\pi ax}\left(g(x,z-\tilde\alpha)-1/x^2\right)
-2{\rm i}\pi b\left(k_\alpha(x,z)+1/x\right)\,.
\end{eqnarray*}
Therefore (\ref{eq-utile}) equals 
\begin{eqnarray*}
&& -2{\rm i}\pi (a-b)\big(\left(k_\alpha(v,z')+\frac1{v}\right)\left(k_\beta(u,z)+\frac1u\right) 
+\left(k_{\beta-\alpha}(u,z-z')+\frac1{u}\right)\left(k_\alpha(u+v,z')+\frac1{u+v}\right) \\
&& +\left(k_{\alpha-\beta}(v,z'-z)-\frac1{v}\right)\left(k_\beta(u+v,z)+\frac1{u+v}\right)\big)\,,
\end{eqnarray*}
which vanishes because of (\ref{twtheta:id}). 
\end{proof}

Let us now assume that $n> 2$. 

Let $\t_{n,+}^\Gamma\subset\t_{1,n}^\Gamma$ be the subalgebra 
generated by $x_i,t^{\alpha}_{jk}$ ($i,j,k = 1,\ldots,n$, $j\neq k$, $\alpha\in\Gamma$). 

There are functions $E_{ij}(\zz)$ with values in $\t_{n,+}^\Gamma$ defined by 
$E_{ij}(\zz)= [\Delta_0,k_{ij}] -[y_i,g_{ij}]$, which 
decomposes as $e_{ij}(\zz) + \sum_{k\neq i,j} e_{ijk}(\zz)$, where $e_{ij}(\zz)$ takes its values in 
$$
\on{Span}_{p,q,\alpha,\beta}[(\ad x_i)^p(t^{\alpha}_{ij}),(\ad x_j)^q(t^{\beta}_{ij})]
$$ 
and $e_{ijk}(\zz)$ takes its values in 
$\on{Span}_{\alpha,\beta}\C[\ad x_i,\ad x_j][t^{\alpha}_{ij},t^{\beta}_{jk}]$. 
Explicitely, 
$$
e_{ij}(\zz)= \sum_{\alpha,\beta}\sum_{p,q}b_{p,q}^{\alpha,\beta}(z_{ij})
[\ad ^px_i(t^{\alpha}_{ij}),\ad ^qx_i(t^{\beta}_{ij})]\,,
$$
where $b_{p,q}^{\alpha,\beta}(z)$ is as before, and 
$$
e_{ijk}(\zz) = \sum_{\alpha,\beta}\left(
{{k_\alpha(\ad x_i,z_{ij})-k_\alpha(-\ad x_j,z_{ij})}\over{(\ad x_i + \ad x_j)^2}}
- {{g_\alpha(-\ad x_j,z_{ij})}\over{\ad x_i + \ad x_j}}\right)
[t^{\alpha}_{ij},t^{\beta}_{ik}]. 
$$

We then define $Y_{ijk}(\zz)=[y_i,g_{jk}]$, which takes its values in 
$\on{Span}_{\alpha,\beta}\C[\ad x_i,\ad x_j][t^{\alpha}_{ij},t^{\beta}_{jk}]$. 
Explicitly, 
$$
Y_{ijk}(\zz) = -\sum_{\alpha,\beta}{{g_\beta(\ad x_j,z_{jk}) - g_{-\beta}(\ad x_k,-z_{jk})}
\over{\ad x_j + \ad x_k}}[t^{\alpha}_{ij},t^{\beta}_{jk}]
$$
(remember that $g_\alpha(u,z)=g_{-\alpha}(-u,-z)$). We obtain  
\begin{eqnarray}
[\Delta(\zz),K_1(\zz)] & = & \sum_{i>1}\left([\Delta_0,k_{1i}]-[y_1,g_{1i}]
+[\frac12\sum_\alpha\delta_{\tilde{\varphi}_\alpha},k_{1i}]-[g_{1i},k_{1i}]\right)
-[\frac12\sum_\alpha\delta_{\tilde{\varphi}_\alpha},y_1] \nonumber \\
& & -\sum_{1<i<j}\left([g_{1i},k_{1j}]+[g_{1j},k_{1i}]+[g_{ij},k_{1i}+k_{1j}]\right) \nonumber \\
& = & \sum_{i>1}\left(e_{12}+[\frac12\sum_\alpha\delta_{\tilde{\varphi}_\alpha},k_{12}]
-[g_{12},k_{12}]-[\frac12\sum_\alpha\delta_{\tilde{\varphi}_\alpha},y_1]\right)_{1i} \label{n=2} \\
& & +\sum_{1<i<j}\left(e_{1ij}+e_{1ji}-Y_{1ij}-[g_{ij},k_{1i}+k_{1j}]-[g_{1i},k_{1j}]-[g_{1j},k_{1i}]\right), 
\nonumber
\end{eqnarray} 
where $\{-\}_{1i}$ is the natural morphism $\t_{1,2}^\Gamma\to\t_{1,n}^\Gamma$, $u_1\mapsto u_1$, 
$u_2\mapsto u_i$ ($u=x,y$), $t^{\alpha}_{12}\mapsto t^{\alpha}_{1i}$. 
It is easy to see that the line (\ref{n=2}) equals $\sum_{i>1}\left([\Delta(z_{1i}),K_1(z_{1i})]\right)_{1i}$ 
which is zero as we have seen before (case $n=2$). 

Therefore $[\Delta(\zz),K_1(\zz)]$ equals 
\begin{eqnarray*}
 && \sum_{1<i<j}\sum_{\alpha,\beta}
\big({{k_\alpha(\ad x_1,z_{1i})-k_\alpha(-\ad x_i,z_{1i})-g_\alpha(-\ad x_i,z_{1i})
(\ad x_1 + \ad x_i)}\over{(\ad x_1+\ad x_i)^2}}[t^{\alpha}_{1i},t^{\beta}_{1j}] \\
&& -{{k_\beta(\ad x_1,z_{1j})-k_\beta(-\ad x_j,z_{1j})-g_\beta(-\ad x_j,z_{1j})
(\ad x_1 + \ad x_j)}\over{(\ad x_1 + \ad x_j)^2}}[t^{\alpha}_{1i},t^{\beta}_{1j}] \\
&& -{{g_{\beta-\alpha}(\ad x_i,z_{ij})-g_{\alpha-\beta}(\ad x_j,-z_{ij})}
\over{\ad x_i + \ad x_j}}[t^{\alpha}_{1i},t^{\beta}_{1j}] \\ 
&& -\left(k_\alpha(\ad x_1,z_{1i})g_{\beta-\alpha}(-\ad x_j,z_{ij})
-k_\beta(\ad x_1,z_{1j})g_{\beta-\alpha}(\ad x_i,z_{ij})\right)[t^{\alpha}_{1i},t^{\beta}_{1j}] \\
&& -\left(k_\beta(-\ad x_j,z_{1j})g_\alpha(-\ad x_i,z_{1i})
-k_\alpha(-\ad x_i,z_{1i})g_\beta(-\ad x_j,z_{1j})\right)[t^{\alpha}_{1i},t^{\beta}_{1j}]\big)\,, 
\end{eqnarray*}
which is zero because of Lemma \ref{lem-utile}. 
\end{proof} 
We have therefore proved (Proposition \ref{prop:flatness1} and Proposition \ref{prop:flatness2} above): 
\begin{theorem}\label{thm:nabla}
The connection $\nabla_{n,\Gamma}$ is flat, and thus so is $\bar\nabla_{n,\Gamma}$. 
\hfill \qed
\end{theorem}

Let us now show how the universal KZB connexion over moduli spaces coincides with the one defined over configuration spaces.

\begin{remark}
The connection $\nabla_{n,\Gamma}$ defined above is an extension to 
the twisted moduli space $\cM_{1,n}^\Gamma$ of the connection 
$\nabla_{n,\tau,\Gamma}$ defined over the twisted configuration space 
$\on{Conf}(E_{\tau,\Gamma},n,\Gamma)$ from Subsection \ref{sec:flatconn}.

Indeed, the pull-back of the principal 
${\mathbf{G}}_n^\Gamma$-bundle with flat connection 
$(\mathcal{P}_{n,\Gamma},\nabla_{n,\Gamma})$ along 
the inclusion 
\[
\on{Conf}(E_{\tau,\Gamma},n,\Gamma)\hookrightarrow\cM_{1,n}^\Gamma
\]
of the fiber at (the class of) $\tau$ in $Y(\Gamma)$ admits 
a reduction of structure group to 
\[
\exp(\t_{1,n}^\Gamma)\subset{\mathbf{G}}_n^\Gamma\,,
\]
and one easily sees from our explicit formul\ae that it coincides with 
$(P_{\tau,n,\Gamma},\nabla_{\tau,n,\Gamma})$ constructed in Subsection \ref{sec:flatconn}. 




Similarly, the connection $\bar\nabla_{n,\Gamma}$ is an extension to 
the twisted moduli space $\bar\cM_{1,n}^\Gamma$ of the connection 
$\bar\nabla_{n,\tau,\Gamma}$ defined over the reduced twisted configuration 
space $\on{C}(E_{\tau,\Gamma},n,\Gamma)$. 
\end{remark}


\subsection{Variations}

Let us first consider the unordered variants 
\[
{\mathcal M}^\Gamma_{1,[n]}=\mathfrak{S}_n \backslash {\mathcal M}^\Gamma_{1,n}
\quad\mathrm{and}\quad
\bar{\mathcal M}^\Gamma_{1,[n]}=\mathfrak{S}_n \backslash \bar{\mathcal M}^\Gamma_{1,n}\,,
\]
where, as before, the action of $\mathfrak{S}_n$ is again by permutation on $\mathbb{C}^n$. 
\begin{proposition}
1.~There exists a unique principal ${\mathbf G}^\Gamma_{n}\rtimes \mathfrak{S}_{n}$-bundle 
\gls{Pnng} over ${\mathcal M}^\Gamma_{1,[n]}$, such that a section over 
$U\subset {\mathcal M}^\Gamma_{1,[n]}$ is a function 
\[
f:\tilde\pi^{-1}(U) \to {\mathbf G}^\Gamma_{n}\rtimes \mathfrak{S}_{n}
\]
satisfying the conditions of Proposition \ref{prop:bundle} as well as 
$f(\sigma\zz|\tau) = \sigma^{-1} f(\zz|\tau)$ for $\sigma\in \mathfrak{S}_n$ (here 
$\tilde\pi : (\C^{n}\times \HH) - \on{Diag}_{n,\Gamma} \to {\mathcal M}^\Gamma_{1,[n]}$ 
is the canonical projection). \\
\indent 2.~There exists a unique flat connection $\nabla_{[n],\Gamma}$ on ${\mathcal P}_{[n],\Gamma}$, 
whose pull-back to $(\C^{n}\times \HH) - \on{Diag}_{n,\Gamma}$ is the connection 
\[
\on d - \Delta(\zz|\tau) \on{d}\tau - \sum_{i}K_{i}(\zz|\tau)\on{d}z_{i}
\]
on the trivial ${\mathbf G}_{n}^\Gamma\rtimes \mathfrak{S}_n$-bundle. \\
\indent 3.~The image of $({\mathcal P}_{[n],\Gamma},\nabla_{[n],\Gamma})$ 
under ${\mathbf G}^\Gamma_{n}\rtimes \mathfrak{S}_{n}\to \bar{\mathbf G}^\Gamma_{n}\rtimes \mathfrak{S}_{n}$ 
is the pull-back of a flat principal $\bar{\mathbf G}^\Gamma_{n}\rtimes \mathfrak{S}_{n}$-bundle 
$(\bar{\mathcal P}_{[n],\Gamma},\bar\nabla_{[n],\Gamma})$ on $\bar{\mathcal M}^\Gamma_{1,[n]}$. 
\end{proposition}
\begin{proof}
For the proof of the first point, one easily checks that 
$\sigma c_{\tilde g}(\zz|\tau) \sigma^{-1} = c_{\sigma \tilde g\sigma^{-1}}(\sigma^{-1}\zz)$, 
where $\tilde g\in (\Z^{n})^{2}\rtimes \on{SL}^\Gamma_{2}(\Z)$, $\sigma\in \mathfrak{S}_n$. 
It follows that there is a unique cocycle 
$c_{(\tilde g,\sigma)} : \C^{n}\times\HH\to \bar{\mathbf G}^\Gamma_{n}\rtimes \mathfrak{S}_n$ 
such that $c_{(\tilde g,1)} = c_{\tilde g}$ and 
$c_{(1,\sigma)}(\zz|\tau)=\sigma$. 

For the proof of the second point, taking into account Theorem \ref{thm:nabla},
one only has to show that this connection is $\mathfrak{S}_n$-equivariant. 
We have already mentioned that $\sum_{i}\bar K_{i}(\zz|\tau)\on{d}z_{i}$ is equivariant, 
and $\bar\Delta(\zz|\tau)$ is also checked to be so. 

The third point is obvious. 
\end{proof}

\medskip

For every (class of) $\tau$ in $Y(\Gamma)$, one has an action of $\Gamma^n$ on the fiber 
$\on{Conf}(E_{\tau,\Gamma},n,\Gamma)$ at $\tau$ of $\cM_{1,n}^\Gamma\twoheadrightarrow Y(\Gamma)$, 
resp.~an action of $\Gamma^n/\Gamma$ on the fiber $\on{C}(E_{\tau,\Gamma},n,\Gamma)$ at $\tau$ of 
$\bar\cM_{1,n}^\Gamma\twoheadrightarrow Y(\Gamma)$. 
Recall that 
\[
\Gamma^n \backslash \on{Conf}(E_{\tau,\Gamma},n,\Gamma)=\on{Conf}(E_{\tau,\Gamma},n)
\quad\mathrm{and}\quad
(\Gamma^n/\Gamma) \backslash  \on{C}(E_{\tau,\Gamma},n,\Gamma)=\on{C}(E_{\tau,\Gamma},n)\,.
\]
This action depends holomorphically of $\tau$, so that there is an action of $\Gamma^n$ on $\cM_{1,n}^\Gamma$, 
resp.~an action of $\Gamma^n/\Gamma$ on $\bar\cM_{1,n}^\Gamma$. 
\begin{proposition}\label{equiv:m}
1.~There exists a unique principal ${\mathbf G}^\Gamma_{n}\rtimes \Gamma^n$-bundle \gls{Pn}
over $\Gamma^n \backslash{\mathcal M}^\Gamma_{1,n}$, such that a section over 
$U\subset \Gamma^n \backslash {\mathcal M}^\Gamma_{1,n}$ is a function 
\[
f:\tilde\pi^{-1}(U) \to {\mathbf G}^\Gamma_{n}\rtimes \Gamma^n
\]
satisfying the following conditions: 
\begin{align*}
f(\mathbf{z}+\frac{\delta_i}{M}|\tau)&=(\bar 1,\bar 0)_if(\mathbf{z}|\tau), \\
f(\mathbf{z}+\tau\frac{\delta_i}{N}|\tau)&=e^{-{{2\pi\i}\over N}x_i }(\bar 0,\bar 1)_if(\mathbf{z}|\tau), \\
f(\zz,\tau+1)&=f(\zz|\tau),\\
f({\zz\over\tau}|-{1\over\tau})&=\tau^{\ddd}e^{{2\pi\i\over\tau}(\XXX+\sum_iz_ix_i)}f(\zz|\tau). 
\end{align*}
Here, $\tilde\pi : (\C^{n}\times \HH) - \on{Diag}_{n,\Gamma} \to \Gamma^n \backslash {\mathcal M}^\Gamma_{1,n}$ 
is the canonical projection. \\
\indent 2.~There exists a unique flat connection on this bundle whose pull-back to 
$(\C^{n}\times \HH) - \on{Diag}_{n,\Gamma}$ is the connection 
\[
\on d - \Delta(\zz|\tau) \on{d}\tau - \sum_{i}K_{i}(\zz|\tau)\on{d}z_{i}
\]
on the trivial ${\mathbf G}_{n}^\Gamma\rtimes \Gamma^n$-bundle. \\
\indent 3.~The image of the above flat bundle under 
${\mathbf G}^\Gamma_{n}\rtimes \Gamma^n\to \bar{\mathbf G}^\Gamma_{n}\rtimes(\Gamma^n/\Gamma)$ 
is the pull-back of a flat principal $\bar{\mathbf G}^\Gamma_{n}\rtimes(\Gamma^n/\Gamma)$-bundle 
on $(\Gamma^n/\Gamma) \backslash \bar{\mathcal M}^\Gamma_{1,n}$. 
\end{proposition}
\begin{proof}
The first assertion is left to the reader. Assertion 3 is evident. Let us prove assertion 2.
By Proposition \ref{prop:equivariance1}, we know that the $K_i$ satisfy 
\begin{itemize}
	\item[(e)] $K_i(\mathbf{z}+\frac{\delta_j}{M}|\tau)
							=(\bar1,\bar0)_j\cdot K_i(\mathbf{z}|\tau)$, 
	\item[(f)] $K_i(\mathbf{z}+\frac{\tau\delta_j}{N}|\tau)
							=(\bar0,\bar1)_j\cdot e^{-{{2\pi\i}\over N}\ad(x_j)}K_i(\mathbf{z}|\tau)$.
\end{itemize}
The fact that $\Delta(\zz+{\delta_j\over M}|\tau)=(\bar{1},\bar{0})_j\cdot\Delta(\zz|\tau)$ is immediate. Thus, it remains to 
show that $\Delta(\zz+{\tau\delta_j\over N}|\tau)=e^{-{{2\pi\i}\over N}\ad(x_j)}(\bar{0},\bar{1})_j\cdot (\Delta(\zz|\tau)-K_j(\zz|\tau))$
which is proved in Lemma \ref{Equivariancetau} below.
\end{proof}

\begin{lemma}\label{Equivariancetau}
We have
\begin{equation}\label{equivD2}
\Delta(\zz+{\tau\delta_j\over N}|\tau)=e^{-{{2\pi\i}\over N}\ad(x_j)}(\bar{0},\bar{1})_j\cdot (\Delta(\zz|\tau)-K_j(\zz|\tau)).
\end{equation}
\end{lemma}

\begin{proof}
On the one hand, 
$$
-2\pi\i\Delta(\zz+{\tau\delta_j \over N})=\Delta_0+\frac12\sum_{s\geq 0,\gamma\in\Gamma}A_{s,\gamma}\delta_{s,\gamma}
-\sum_{\substack{k<l\\ k,l\neq j}}g_{kl}(z_{kl})
-\sum_{\substack{k:k\neq j\\\alpha\in\Gamma}}g_\alpha(\ad x_j,z_{jk}+{\tau \over N})(t^\alpha_{jk}).
$$
On the other hand, as
\begin{eqnarray*}
e^{-{{2\pi\i}\over N}\ad(x_j)}(\Delta_0) & = & (1-(1-e^{-{{2\pi\i}\over N}\ad(x_j)})(\Delta_0)=
(\Delta_0)+\frac{1-e^{-{{2\pi\i}\over N}ad x_j}}{\ad x_j}(y_j) \\ & = & 
\frac{e^{-{{2\pi\i}\over N}\ad x_j}-1}{(\ad x_j)^2}\left(\sum_{\alpha\in\Gamma}\sum_{k:k\neq j}t_{jk}^\alpha\right)
\end{eqnarray*}

and the $\delta_{s,\gamma}$ commute with the $x_j$, we compute 
$$
2\pi\i\left(K_j(\zz+{\tau \over N}\delta_j|\tau)-e^{-{{2\pi\i}\over N}\ad(x_j)  }(\bar{0},\bar{1})_j\cdot\Delta(\zz|\tau)\right)
$$
$$
=2\pi\i\left((\bar{0},\overline{-1})_j\cdot K_j(\zz+{\tau \over N}\delta_j|\tau)-e^{-{{2\pi\i}\over N}\ad(x_j)  }\Delta(\zz|\tau)\right)
$$
\begin{eqnarray*}
& = & 2\pi\i(\bar{0},\overline{-1})_j\cdot \left(\sum_{k:k\neq j}k_\alpha(\ad x_j,z_{jk}+{\tau \over N})-y_j\right)+\Delta_0
+\frac{1-e^{-{{2\pi\i}\over N}\ad x_j  }}{\ad x_j}(y_j) \\
&  & +\frac12\sum_{\substack{s\geq 0,\\\gamma\in\Gamma}}A_{s,\gamma}\delta_{s,\gamma}
-e^{-{{2\pi\i}\over N}\ad x_j   }\sum_{k<l}g_{kl}(z_{kl}).
\end{eqnarray*}
Next, by combining
$$
K_{ij}(z+{\tau\over N})=(\bar0,\bar{-1})_i \cdot e^{-{{2\pi\i}\over N}\ad x_j}\cdot (K_{ij}(z)) 
									+ (\bar0,\bar{-1})_i\cdot ( \sum_{\alpha\in\Gamma}{{e^{-{{2\pi\i}\over N}\ad x_i}-1}\over{\ad x_i }}(t^{\alpha}_{ij}))\,, 
$$
with equation
$$
g_\alpha(x,z)-1/x^2=e^{-2{\rm i}\pi ax}\left(g(x,z-\tilde\alpha)-1/x^2\right)
-2{\rm i}\pi a_0\left(k_\alpha(x,z)+1/x\right)\,,
$$
we can follow the same lines as in the proof of relation \eqref{equivD} to obtain the wanted equation.
\end{proof}

We also leave to the reader the task of combining several variants. 

\section{Realizations}
\label{Realizations}

\subsection{Realizations of $\bar\t_{1,n}^\Gamma$ and $\bar\t_{n,+}^\Gamma$}\label{sec:real1}

Let $\g$ be a Lie algebra and $t_\g\in S^2(\g)^\g$ be nongenerate. Assume that there is a group 
morphism $\Gamma\to\on{Aut}(\g,t_\g)$ and set $\l:=\g^\Gamma$ and 
$\u:=\oplus_{\chi\in\widehat\Gamma-\{0\}}\g_\chi$, 
where $\g_\chi$ is the eigenspace of $\g$ corresponding to the character $\chi:\Gamma \longrightarrow \C^{\star}$. 
Then $\g=\l\oplus \u$ with $[\l,\u]\subset\u$, and $\t=\t_\l+\t_\u$ with $\t_\l\in S^2(\l)^\l$ and 
$\t_\u\in S^2(\u)^\l$. 
We denote by $(a,b)\mapsto\<a,b\>$ the invariant pairing on $\l$ corresponding to $t_\l$ and write 
$t_\l=\sum_\nu e_\nu\otimes e_\nu$. 

Let $\Diff(\l^*)$ be the algebra of algebraic differential operators on $\l^*$. It has generators $\x_l$, 
$\partial_l$ ($l\in\l$) and relations $\x_{tl+l'}=t\x_l+\x_{l'}$, $\partial_{tl+l'}=t\partial_l+\partial_{l'}$, 
$[\x_l,\x_{l'}]=0=[\partial_l,\partial_{l'}]$ and $[\partial_l,\x_{l'}]=\<l,l'\>$. 
Moreover, one has a Lie algebra morphism $\l\to\Diff(\l^*);l\mapsto X_l:=\sum_\nu\x_{[l,e_\nu]}\partial_{e_\nu}$. 
We denote by $\l^{\rm diag}$ the image of the induced morphism 
$$
\l\ni l\mapsto Y_l:=X_l\otimes 1+1 \otimes \sum_{i=1}^nl^{(i)}\in\Diff(\l^*)\otimes U(\g)^{\otimes n}\,,
$$
and define \gls{Hgl} as the Hecke algebra of $A_n:=\Diff(\l^*)\otimes U(\g)^{\otimes n}$ with respect to $\l^{\on{diag}}$. 
Namely, $H_n(\g,\l^*):=(A_n)^\l/(A_n\l^{\rm diag})^\l$. It acts in an obvious way on 
$(\cO_{\l^*}\otimes(\otimes_{i=1}^n V_i))^\l$ if $(V_i)_{1\leq i\leq n}$ is a collection of $\g$-modules. 

Let us set $\x_\nu:=\x_{e_\nu}$ and $\partial_\nu:=\partial_{e_\nu}$, and write $\alpha^{(i)}\cdot$ for the action of 
$\alpha\in\Gamma$ on the $i$-th component in $U(\g)^{\otimes n}$. 
\begin{proposition}\label{prop:real1}
There is a unique Lie algebra morphism $\rho_\g:\bar\t_{1,n}^\Gamma\to H_{n}(\g,\l^*)$ defined by
\begin{align*}
\bar x_i&\longmapsto M\sum_\nu\x_\nu\otimes e_\nu^{(i)},  \\
\bar y_i&\longmapsto -N\sum_\nu\partial_\nu\otimes e_\nu^{(i)}, \\
\bar t_{ij}^\alpha&\longmapsto 1\otimes(\alpha^{(1)}\cdot t_\g)^{(ij)}.
\end{align*}
\end{proposition}
\begin{proof}
Let us use the presentation of $\bar\t_{1,n}^\Gamma$ coming from Lemma \ref{lem:pres1}. 
The only non trivial check is that the relation $\sum_j \bar x_j = 0$ is preserved. 
We have
\begin{eqnarray*}
  \rho_{\mathfrak{g}} \left( \overset{n}{\sum_{i = 1}} x_i \right) & = &M
  \sum_{\nu} \mathrm{x}_{\nu} \otimes \overset{n}{\sum_{i = 1}} e_{\nu}^{(i)}
  =M \sum_{\nu} \left( \mathrm{x}_{\nu} \otimes 1 \right) \left( 1 \otimes
  \overset{n}{\sum_{i = 1}} e_{\nu}^{(i)} \right) \\ & \equiv & M \sum_{\nu} \left(
  \mathrm{x}_{\nu} \otimes 1 \right) (Y_{\nu} - X_{\nu} \otimes 1)\\
  & \equiv & M - \sum_{\nu} \mathrm{x}_{\nu} X_{\nu} \otimes 1 = M
  \underset{\nu_1, \nu_2}{\sum} \mathrm{x}_{e_{\nu_1}} \mathrm{x}_{[e_{\nu_1},
  e_{\nu_2}]} \partial_{\nu_2} \otimes 1 = 0
\end{eqnarray*}
as $\mathrm{x}_{e_{\nu_1}}$ commutes with $\mathrm{x}_{[e_{\nu_1},e_{\nu_2}]}$ and $t_{\mathfrak{l}}$ is invariant. 
Here the sign $\equiv$ means that both terms define the same equivalence class in  $H_{n}(\g,\l)$. 

The proof that $\sum_j \bar y_j = 0$ is preserved is a consequence of the fact that 
$\rho_{\mathfrak{g}}\left(\sum_j \bar y_j\right)=0$, which was proven in \cite[Proposition 6.1]{CEE}. 
\end{proof}
Let $\bar\t_{n,+}^\Gamma\subset\bar\t_{1,n}^\Gamma$ be the Lie subalgebra generated by 
$\bar x_i$'s and $\bar t_{jk}^\alpha$'s. 
Then the restriction of $\rho_\g$ to $\bar\t_{n,+}^\Gamma$ lifts to a Lie algebra morphism 
$\bar\t_{n,+}^\Gamma\to(\cO_{\l^*}\otimes U(\g)^{\otimes n})^\l$. 
Moreover, $(\cO_{\l^*}\otimes U(\g)^{\otimes n})^\l$ is a subalgebra of $H_{n}(\g,\l^*)$ 
that is a Lie ideal for the commutator, and one has a commutative diagram 
$$
\xymatrix{
\bar\t_{1,n}^\Gamma\times\bar\t_{n,+}^\Gamma\ar[d]\ar[r]^{(u,v)\mapsto[u,v]} & \bar\t_{n,+}^\Gamma \ar[d] \\
H_{n}(\g,\l^*)\times(\cO_{\l^*}\otimes U(\g)^{\otimes n})^\l \ar[r] & (\cO_{\l^*}\otimes U(\g)^{\otimes n})^\l\,. 
}
$$

\subsection{Realizations of $\bar\t_{1,n}^\Gamma\rtimes\d^\Gamma$}

Let us write $t_\g=\sum_ua_u\otimes a_u$. 

\begin{proposition}
The Lie algebra morphism $\rho_\g$ of Proposition \ref{prop:real1} extends to a Lie 
algebra morphism $\bar\t_{1,n}^\Gamma\rtimes\d^\Gamma\to H_{n}(\g,\l^*)$ defined by 
\begin{align*}
\ddd&\longmapsto-\frac12(\sum_\nu\x_\nu\partial_\nu+\partial_\nu\x_\nu)\otimes1, \\
\XXX&\longmapsto\frac12(\sum_\nu\x_\nu^2)\otimes1, \\
\Delta_0&\longmapsto-\frac12(\sum_\nu\partial_\nu^2)\otimes1,\\
\xi_{s,\gamma}&\longmapsto \frac{1}{|\Gamma|} \sum_{\nu_1,\cdots,\nu_s,u}\x_{\nu_1}\cdots\x_{\nu_s}\otimes
\sum_{i=1}^n\left(\ad(e_{\nu_1})\cdots\ad(e_{\nu_s})(a_u)\odot(\gamma\cdot a_u)\right)^{(i)}\,.
\end{align*}
\end{proposition}
Here $\odot$ denotes the symmetric product: $A\odot B:=AB+BA$. 
\begin{proof}
Since $t_\g$ is invariant under the commuting actions of $\Gamma$ and $\l$ then the relation 
$\xi_{s,\gamma}=(-1)^s\xi_{s,-\gamma}$ is also preserved.  
This invariance argument also implies that 
$[\rho_\g(\xi_{s,\gamma}),\rho_\g(\bar x_i)]$ equals 
$$
\frac{1}{|\Gamma|} \sum_{\nu_1,\cdots,\nu_s,\nu,u}\x_{\nu_1}\cdots\x_{\nu_s}\x_\nu\otimes
\sum_{t=1}^s\left(\ad(e_{\nu_1})\cdots\ad([e_\nu,e_{\nu_t}])\cdots\ad(e_{\nu_s})(a_u)\odot(\gamma\cdot a_u)\right)^{(i)}\,,
$$
which is zero since the first and second factors are respectively symmetric and antisymmetric in $(\nu,\nu_t)$. 
Let us now prove that the relation $[\xi_{s,\gamma},\bar t_{ij}^\alpha]
=[\bar t_{ij}^\alpha,(\ad \bar x_i)^s(\bar t_{ij}^{\alpha-\gamma})+(\ad \bar x_j)^s(\bar t_{ij}^{\alpha+\gamma})]$ 
is preserved. It is sufficient to do it for $n=2$: 
$$
\rho_\g(\xi_{s,\gamma}+(\ad x_1)^s(t_{12}^{\alpha-\gamma})+(\ad x_2)^s(t_{12}^{\alpha+\gamma}))
=\sum_{\nu_1,\cdots,\nu_s}\x_{\nu_1}\cdots\x_{\nu_s}\otimes(\alpha^{(1)}\cdot\Delta(B_{\nu_1,\cdots,\nu_s}))\,,
$$
where $\Delta$ is the standard coproduct of $U\g$ and 
$B_{\nu_1,\cdots,\nu_s}:=\sum_u{\ad(e_{\nu_1})\cdots\ad(e_{\nu_s})(a_u)\odot(\gamma\cdot a_u)}$; therefore 
$\rho_\g(\xi_{s,\gamma}+(\ad x_1)^s(t_{12}^{\alpha-\gamma})+(\ad x_2)^s(t_{12}^{\alpha+\gamma}))$ 
commutes with $\rho_\g(t_{12}^{\alpha})$. 
Hence it remains to prove that the relation $[\xi_{s,\gamma},\frac{y_i}{N}]=
\sum_{j:j\neq i}D_{s,\gamma}(\frac{x_i}{M},\frac{t_{ij}^\beta}{|\Gamma|})$  is preserved. For this 
we compute $[\rho_\g(\xi_{s,\gamma}),\rho_\g(\frac{y_i}{N})]$: it equals 
\begin{align*}
& \frac{1}{|\Gamma|} \sum_{\substack{\nu_1,\cdots,\nu_s \\ \nu,u}}\big(\sum_{j=1}^n[\partial_\nu,\x_{\nu_1}\cdots\x_{\nu_s}]\otimes
e_\nu^{(i)}\big(\ad(e_{\nu_1})\cdots\ad(e_{\nu_s})(a_u)\odot(\gamma\cdot a_u)\big)^{(j)} \cr
&+\x_{\nu_1}\cdots\x_{\nu_s}\partial_\nu\otimes
[e_\nu,\ad(e_{\nu_1})\cdots\ad(e_{\nu_s})(a_u)\odot(\gamma\cdot a_u)]^{(i)}\big) \cr
&= \frac{1}{|\Gamma|} \sum_{l=1}^s\sum_{\nu_1,\dots,\nu_s,\nu}\x_{\nu_1}\cdots\check{\x}_{\nu_l}\cdots\x_{\nu_s}\otimes
\sum_{j=1}^n\left(e_\nu^{(i)}\big(\ad(e_{\nu_1})\cdots\ad(e_{\nu_s})(a_u)\odot(\gamma\cdot a_u)\big)^{(j)}
-(i\leftrightarrow j)\right)\,.
\end{align*}
The term corresponding to $j=i$ is the linear map $S^{s-1}(\l)\to U(\g)^{\otimes n}$ such that for $x\in\l$
$$
x^{s-1}\longmapsto \frac{1}{|\Gamma|}\sum_{\substack{p+q=s-1\\\nu,u}}
[e_\nu,\ad(x)^p\ad(e_\nu)\ad(x)^q(a_u)\odot(\gamma\cdot a_u)]^{(i)}\,.
$$
Using $\l$-invariance of $\sum_ua_u\odot(\gamma\cdot a_u)$ one obtains that this last expression equals 
$$
= \frac{1}{|\Gamma|}\sum_{\substack{p+q+r=s-1\\\nu,u}}
\big(\ad(x)^p\ad([e_\nu,x])\ad(x)^q\ad(e_\nu)(\ad x)^r(a_u)\odot(\gamma\cdot a_u)
$$
$$
+\ad(x)^p\ad(e_\nu)\ad(x)^q\ad([e_\nu,x])\ad(x)^r(a_u)\odot(\gamma\cdot a_u)\big)^{(i)}\,,
$$
which is zero from the $\l$-invariance of $t_\l=\sum_\nu e_\nu\otimes e_\nu$. 
The term corresponding to $j\neq i$ is the linear map $S^{s-1}(\l)\to U(\g)^{\otimes n}$ such that for $x\in\l$ 
$$
x^{s-1}\longmapsto \frac{1}{|\Gamma|} \sum_{\substack{p+q=s-1\\\nu,u}}
\left(\ad(x)^p\ad(e_\nu)\ad(x)^q(a_u)\odot(\gamma\cdot a_u)\right)^{(j)}e_\nu^{(i)}-(i\leftrightarrow j)
$$
$$
= \frac{1}{|\Gamma|} \sum_{\substack{p+q=s-1\\\nu,u}}
\left(\ad(x)^p([e_\nu,a_u])\odot(-\ad(x))^q(\gamma\cdot a_u)\right)^{(j)}e_\nu^{(i)}-(i\leftrightarrow j)
$$
$$
= \frac{1}{|\Gamma|} \sum_{\substack{p+q=s-1\\\nu,u}}(-1)^q 
\left(\ad(x)^p([e_\nu,a_u])\odot(\ad(x))^q(\gamma\cdot a_u)\right)^{(j)}e_\nu^{(i)}-(i\leftrightarrow j)
$$
$$
= \frac{1}{|\Gamma|} \sum_{\substack{p+q=s-1\\\nu,u}}(-1)^q 
\left(\ad(x)^p([e_\nu,a_u])\odot(\ad(x))^q(\gamma\cdot a_u)\right)^{(j)}e_\nu^{(i)}-(i\leftrightarrow j)
$$
$$
= \frac{1}{|\Gamma|^2}   \sum_{\beta \in \Gamma} 
\sum_{\substack{p+q=s-1\\v,u}}(-1)^q \left(\ad(x)^p([a_v,a_u])\odot(\ad(x))^q(\gamma\cdot a_u)\right)^{(j)}(\beta \cdot a_v)^{(i)}-(i\leftrightarrow j)
$$
$$
= \frac{1}{|\Gamma|^2} \sum_{\beta \in \Gamma} 
\sum_{p+q=s-1}(-1)^q\sum_{\nu,u} \left( \ad(x)^p(a_v) \odot \ad(x)^q  (\gamma\cdot a_u)\right)^{(i)}(\beta \cdot [a_u,a_v])^{(j)}-(i\leftrightarrow j)
$$
$$
= \frac{1}{|\Gamma|^2} \sum_{\beta \in \Gamma} 
\sum_{p+q=s-1}(-1)^q\sum_{\nu,u} \left( \ad(x)^p(\beta \cdot a_v) \odot \ad(x)^q  ((\beta + \gamma)\cdot a_u)\right)^{(i)}[a_u,a_v]^{(j)}-(i\leftrightarrow j)
$$
$$
= \frac{1}{|\Gamma|^2} \sum_{\beta \in \Gamma} 
\sum_{p+q=s-1}(-1)^q\sum_{\nu,u} \left( \ad(x)^p((\beta- \gamma) \cdot a_v) \odot \ad(x)^q  ((\beta )\cdot a_u)\right)^{(i)}[a_u,a_v]^{(j)}-(i\leftrightarrow j)
$$
which coincides with the image of 
$$D_{s,\gamma}\left(\frac{x_i}{M},\frac{t_{i j}^{\beta}}{|\Gamma|}\right)=
\sum_{p+q=s-1}\sum_{\beta\in\Gamma}\left[\left(\ad \frac{x_i}{M}\right)^p 
\left(\frac{t_{i j}^{\beta}}{|\Gamma|}\right),\left(-\ad \frac{x_i}{M}\right)^q \left(\frac{t_{i j}^{\beta}}{|\Gamma|}\right)\right]$$ 
under $\rho_\g$. In conclusion we get the relation 
$$\rho_\g\left(\left[\xi_{s,\gamma},\frac{y_i}{N}\right]\right)=\left[\rho_\g(\xi_{s,\gamma}),\rho_\g\left(\frac{y_i}{N}\right)\right].$$ 
A direct computation shows that the commutation relations of $[\XXX,\xi_{s,\gamma}]=0$, $[\ddd, \xi_{s,\gamma}]=s\xi_{s,\gamma}$ 
and $\on{ad}^{s+1}(\Delta_0)(\xi_{s,\gamma})=0$ are
preserved, which finishes the proof.
\end{proof}

\subsection{Reductions}

Assume that $\l$ is finite dimensional and there is a reductive decomposition $\l=\h\oplus\m$, 
i.e.~$\h\subset\l$ is a subalgebra and $\m\subset\l$ is a vector subspace such that 
$[\h,\m]\subset\m$. We also assume that $t_\l=t_\h+t_\m$ with $t_\h=\sum_{\bar\nu}e_{\bar\nu}\otimes e_{\bar\nu}\in S^2(\h)^\h$ 
and $t_\m\in S^2(\m)^\h$, and that for a generic $h\in\h$, $\ad(h)_{|\m}\in{\rm End}(\m)$ is invertible. This last condition means that 
$$P(\lambda):={\rm det}(\ad(\lambda^\vee))_{|\m})\in S^{\rm dim(\m)}(\h)$$ is nonzero, where 
$\lambda^\vee:=(\lambda\otimes{\rm id})(t_\h)$ for any $\lambda\in\h^*$.
%

We now define \gls{Hgh}. As in the previous paragraph, $\Diff(\h^*)$ has generators $\bar\x_h$, 
$\bar\partial_h$ ($h\in\h$) and relations 
\begin{align*}
\bar\x_{th+h'}=t\bar\x_h+\bar\x_{h'}, \\
\bar\partial_{th+h'}=t\bar\partial_h+\bar\partial_{h'}, \\
[\bar\x_h,\bar\x_{h'}]=0=[\bar\partial_h,\bar\partial_{h'}],\\
[\bar\partial_h,\bar\x_{h'}]=\<h,h'\>,
\end{align*}
and $\Diff(\h^*_{reg})=\Diff(\h^*)[\frac1P]$ with 
$[\bar\partial_l,\frac1P]=-\frac{[\bar\partial_l,P]}{P^2}$. One has a Lie algebra morphism 
$$\h\to\Diff(\h^*);h\longmapsto\bar X_h:=\sum_{\bar\nu}\x_{[h,e_{\bar\nu}]}\partial_{e_{\bar\nu}}.$$ 
We denote by $\h^{\rm diag}$ the image of the map 
$$\h\ni h\longmapsto\bar Y_h:=\bar X_h+\sum_{i=1}^nl^{(i)}\in\Diff(\h^*_{reg})\otimes U(\g)^{\otimes n}=:B_n,$$
and define $H_{n}(\g,\h^*_{reg})$ as the Hecke algebra of 
$B_n$ with respect to $\h^{\on{diag}}$: $$H_n(\g,\h^*_{reg}):=(B_n)^\h/(B_n\h^{\rm diag})^\h.$$ 
It acts in an obvious way on $(\cO_{\h^*_{reg}}\otimes(\otimes_{i=1}^n V_i))^\h$ if 
$(V_i)_{1\leq i\leq n}$ is a collection of $\g$-modules. 
Finally, let us set, for $\lambda \in \h^*$,
$$
r(\lambda) := (\on{id}\otimes (\on{ad}\lambda^{\vee})_{|\m}^{-1})(t_{\m}). 
$$
Then, following \cite{EE}, $r : \h^*_{\on{reg}} \to \wedge^2(\m)$ is an $\h$-equivariant map
satisfying the classical dynamical Yang-Baxter equation (CDYBE)
$$
\sum_{\bar\nu} e_{\bar\nu}^{(1)}\partial_{\bar\nu} r^{(23)}
+[r^{(12)},r^{(13)}]+c.p.(1,2,3)=0\,,
$$
and we write $r = \sum_\delta a_\delta \otimes b_\delta \otimes \ell_\delta
\in (\m^{\otimes 2} \otimes 
S(\h)[1/P])^\h$.
\begin{proposition}
There is a unique Lie algebra morphism $\rho_{\g,\h}:\bar\t_{1,n}^\Gamma\to H_n(\g,\h^*_{reg})$ given by 
\begin{align*}
\bar x_{i}&\longmapsto  M \sum_{\bar\nu}\bar{\on{x}}_{\bar\nu} \otimes 
h_{\bar\nu}^{(i)},  \\
\bar y_{i}&\longmapsto -N \sum_{\bar\nu} \bar\partial_{\bar\nu} 
\otimes h_{\bar\nu}^{(i)} + \sum_{j} \sum_{\delta} \ell_\delta \otimes 
a_\delta^{(i)}b_\delta^{(j)}, \\
\bar t_{ij}^\alpha&\longmapsto 1\otimes(\alpha^{(1)}\cdot t_\g)^{(ij)}.
\end{align*}
\end{proposition}

\begin{proof}
First of all, the images of the above elements are all $\h$-invariant. 
As in \cite{CEE}, we will imply summation over repeated indices, and adopt the following conventions: 
$\bar\partial_{e_{\bar\nu}}=\bar\partial_{\bar\nu}$, 
$\bar{\on{x}}_{e_{\bar\nu}}=\bar{\on{x}}_{\bar\nu}$, 
and $1\otimes-$'s and $-\otimes 1$'s may be dropped from the notation. 

In particular, $\rho_{\g,\h}(\bar x_i) = h_{\bar\nu}^{(i)}\bar{\on{x}}_{\bar\nu}$, 
$\rho_{\g,\h}(\bar y_{i}) = - h_{\nu}^{(i)}\bar\partial_{\nu} + \sum_{j=1}^n r(\lambda)^{(ij)}$ 
(here, for $x\otimes y\in \g^{\otimes 2}$, $(x\otimes y)^{(ii)}:= x^{(i)} y^{(i)}$).

We will use the same presentation of $\bar\t_{1,n}^\Gamma$ as in Lemma \ref{lem:pres1}. 
The relations $[\bar x_{i},\bar x_{j}] = 0$ and $\bar t_{ij}^\alpha=\bar t_{ji}^{-\alpha}$
are obviously preserved.

Let us check that $[\bar x_{i},\bar y_{j}]=\sum \bar t_{ij}^\alpha$ is preserved. 
Indeed, for $i\neq j$, 
\begin{eqnarray*}
\frac{1}{MN}[\rho_{\g,\h}(\bar x_{i}),\rho_{\g,\h}(\bar y_{j})] 
& = & -\sum_{\bar\nu_1,\bar\nu_2}[\bar{\on{x}}_{\bar\nu_1},\partial_{\bar\nu_2}]h_{\bar\nu_1}^{(i)}h_{\bar\nu_2}^{(j)} 
+ \sum{\bar\nu,\delta,k}\bar{\on{x}}_{\bar\nu}[h_{\bar\nu}^{(i)},\ell_{\delta} \otimes a_{\delta}^{(j)}b_{\delta}^{(k)}]\\
& = & t_{\h}^{(ij)} + t_{\m}^{(ij)} = t_{\l}^{(ij)} = \frac{1}{MN} \sum_{\alpha\in\Gamma} \alpha^{(i)} \cdot t_{\g}^{(ij)}
\end{eqnarray*}
by the same argument as in Proposition \ref{prop:real1}.

Let us check that $\sum_{i}\bar x_{i} = \sum_{i}\bar y_{i}=0$ are preserved. 
We have $\sum_{i}\rho_{\g,\h}(\bar x_{i}) = 0$
and $\sum_{i}\rho_{\g,\h}(\bar y_{i}) = \sum_{\bar\nu,i}h_{\bar\nu}^{(i)}\partial_{\bar\nu}$
(by the antisymmetry of $r$), which equals zero as in Proposition \ref{prop:real1}.

The fact that the relation $[\bar y_{i},\bar y_{j}]=0$ is satisfied for $i\neq j$ is a consequence of the dynamical 
Yang-Baxter equation (this follows from the exact same argument as in the proof of \cite[Proposition 63]{CEE}).

Next, $[\bar x_{i},\bar t_{jk}^\alpha]=0$ is preserved ($i,j,k$ distinct). Indeed, 
$$
[\rho_{\g,\h}(\bar x_{i}),\rho_{\g,\h}(\bar t_{jk}^\alpha)] = \sum_{\bar\nu}\bar{\on{x}}_{\bar\nu}[h_{\bar\nu}^{(i)},\alpha^{(i)}\cdot t_\g^{(jk)}]=0\,.
$$

Finally $[\bar y_{i},\bar t_{jk}^\alpha] = 0$ is preserved ($i,j,k$ distinct): 
\begin{align*}
[\rho_{\g,\h}(\bar y_{i}),\rho_{\g,\h}(\bar t_{jk}^\alpha)] = &
[-\sum_{\bar\nu}h_{\bar\nu}^{(i)}\bar\partial_{\bar\nu}+\sum_{l}r^{(il)},\alpha^{(j)}\cdot t_\g^{(jk)})] \\
= & [r(\lambda)^{(ij)} +r(\lambda)^{(ik)},\alpha^{(j)}\cdot t_\g^{(jk)})]=0\,,
\end{align*}
where the last equality follows the the $\g$-invariance of $t_{\g}$. 
\end{proof}

\begin{remark}
We expect that there exists a Lie algebra morphism 
\[
{\rm red}_{\l,\h}:H_n(\g,\l^*)\to H_n(\g,\h^*_{reg})
\]
such that the following diagram commutes
$$
\xymatrix{
\t_{1,n}^\Gamma\ar[r]^{\rho_{\g}}\ar[dr]_{\rho_{\g,\h}} & H_n(\g,\l^*) \ar[d]^{{\rm red}_{\l,\h}} \\
 & H_n(\g,\h^*_{reg})
}
$$
\end{remark}

\subsection{Elliptic dynamical $r$-matrix systems as realizations of the universal $\Gamma$-KZB system on twisted configuration spaces}

Let $K(z)$ be a meromorphic function on $\C$ with values in the subalgebra 
$\hat{\t}_{2,+}^\Gamma\subset\hat{\t}_{1,2}^\Gamma$ generated 
by $x_1$, $x_2$, $t_{12}^\alpha$ ($\alpha\in\Gamma$), such that 
$K(-z)=-K(z)^{2,1}$ and satisfying the universal CDYBE with a spectral parameter 
$$
-[y_1,K(z_{23})^{2,3}]+[K(z_{12})^{1,2},K(z_{13})^{1,3}]+c.p.(1,2,3)=0\,.
$$
On the one hand, it follows from \S\ref{sec:real1} that the image 
$r(\x,z):=\rho_{\mathfrak g}(K(z))$ of $K(z)$ under $\rho_{\mathfrak g}:\hat{\t}_{2,+}^\Gamma\to(\hat\cO_{\l^*}\otimes\g^{\otimes2})^\l$ 
is a dynamical $r$-matrix\footnote{Remember that $\cO_{\l^*}:=S(\l)$ and $\hat\cO_{\l^*}:=\hat S(\l)$. } with spectral parameter, 
i.e.~a solution of the CDYBE with a spectral parameter for the pair $(\l,\g)$
$$
\sum_\nu e_\nu^{(1)}\partial_\nu r(\x,z_{23})^{(23)}
+[r(\x,z_{12})^{(12)},r(\x,z_{13})^{(13)}]+c.p.(1,2,3)=0\,,
$$
which satisfies $r(\x,-z)=-r(\x,z)^{(21)}$. 
On the other hand, the image of $K(z)$ under $\rho_{\g,\h}:\hat{\t}_{2,+}^\Gamma\to(\hat\cO_{\h_{reg}^*}\otimes\g^{\otimes2})^\h$
 is precisely equal to the restriction $\rho_{\g}(K(z))|_{\h^*} \in (\hat\cO_{\h_{reg}^*}\otimes\g^{\otimes2})^\h$ of $\rho_{\g}(K(z))$ to $\h^*$.
Then applying \cite[Proposition 0.1]{EE}, we conclude that 
$$
\tilde r (\bar\x,z):=\rho_{\g,\h}(K(z))+r(\lambda)
$$
is a solution of the CDYBE with spectral parameter for $(\h,\g)$: 
$$
\sum_{\bar\nu} e_{\bar\nu}^{(1)}\partial_{\bar\nu} \tilde r(\bar\x,z_{23})^{(23)}
+[\tilde r(\bar\x,z_{12})^{(12)},\tilde r(\bar\x,z_{13})^{(13)}]+c.p.(1,2,3)=0\,.
$$

Then for any $n$-tuple $\underline{V}=(V_1,\dots,V_n)$ of $\g$-modules one has a flat connection 
$\nabla_{\tau,n,\Gamma}^{(\underline{V})}$ on the trivial vector bundle over $\C^n-{\rm Diag}_{\tau,n\Gamma}$ 
with fiber $(\cO_{\h^*_{reg}}\otimes(\otimes_iV_i))^\h$, defined by the following compatible system of 
first order differential equations: 
\begin{equation}\label{eq-dynrmat-sys}
\partial_{z_i}F(\bar\x,\zz)=\sum_{\bar\nu}e_{\bar\nu}^{(i)}\cdot\bar\partial_{\bar\nu}F(\bar\x,\zz)
+\sum_{j:j\neq i}\tilde r^{(ij)}(\bar\x,z_{ij})\cdot F(\bar\x,\zz)\,.
\end{equation}
Here $\zz\mapsto F(\bar\x,\zz)$ is a function with values in $(\cO_{\h^*_{reg}}\otimes(\otimes_iV_i))^\h$. \\

\medskip

Starting from $K(z)=K_{12}(z)$ as in \S\ref{sec:flatconn}, it would be interesting to know if one can recover 
(up to gauge equivalence), using the above realization morphisms, the generalization of Felder's elliptic 
dynamical $r$-matrices \cite{FICM} constructed in \cite{ES,FP}. 

\section{Formality of subgroups of the pure braid group on the torus}
\label{Formality of subgroups of the pure braid group on the torus}


\subsection{Relative formality}\label{relfor}

Let $G$ and $S$ be two groups, with $S$ finite, and let $\varphi : G\to S$ be a surjective group morphism 
with finitely generated kernel $\on{Ker}\varphi$. We then consider the category of commuting triangles 
$$
\xymatrix{
G \ar[r]\ar[rd]^{\varphi} & G' \ar[d]^{\varphi'} \\
& S
}
$$
where $G'$ is pro-algebraic, and $\varphi'$ is surjective with $\kk$-prounipotent kernel. 
This category has an initial object, denoted $\varphi(\kk):G(\varphi,\kk)\to S$, which we call the 
\textit{relative ($\kk$-prounipotent) completion} of $G$ with respect to $\varphi$. 

\medskip

Observe that, if we regard the finite group $S$ as an affine algebraic group, then this is a particular 
case of the relative completion defined in \cite{HainMalcev}. It also  coincides with the partial completion 
defined in \cite[\S1.1]{En} (which seems to force $S$ to be finite). 

\medskip

Right exactness of relative completion (see e.g.~\cite[Proposition 2.4]{HM}), together with standard 
characterization of obstructions to left exactness, provides us with an exact sequence\footnote{This can 
also be seen as the end of the long exact sequence from \cite[Theorem 1.17]{Pri}. } 
\[
H_2(S,\kk)\longrightarrow \big(\on{Ker}\varphi\big)(\kk) \longrightarrow G(\varphi,\kk) \longrightarrow S \longrightarrow 1.
\]
Since $S$ is finite, $H_2(S,\kk)=0$, and thus we get that the kernel $\on{Ker}\big(\varphi(\kk)\big)$ of 
$\varphi(\kk)$ is the usual $\kk$-prounipotent completion $\big(\on{Ker}\varphi\big)(\kk)$ of the 
kernel of $\varphi$, which we can therefore unambiguously denote $\on{Ker}\varphi(\kk)$. 

\begin{lemma}
Every extension 
\[
1 \longrightarrow U \longrightarrow H \longrightarrow S \longrightarrow 1
\]
of a finite group by a $\kk$-prounipotent one splits. 
\end{lemma}

\begin{proof}
We consider the filtration $(F_i)_i$ given by the lower central series of $U$, 
and prove by induction that 
\[
1 \longrightarrow U/F_i \longrightarrow H/F_i \longrightarrow S \longrightarrow 1
\]
splits. \\
\underline{Initial step ($i=2$):} Recall that $F_1=U$, and that $F_1/F_2$ is abelian and finitely generated, 
so that 
\[
1 \longrightarrow U/F_2 \longrightarrow H/F_2 \longrightarrow S \longrightarrow 1
\]
splits, as every extension of a finite group by a finite dimensional representation splits 
(this is because the cohomology of a finite group with coefficients in a divisible module vanishes). \\
\underline{Induction step:} There is a (surjective) morphism of extensions
\[
\xymatrix{
1 \ar[r] & U/F_{i+1} \ar[r]  \ar[d] 
& H/F_{i+1} \ar[r] \ar[d] & S \ar[r]  \ar[d]&  1. \\
1 \ar[r] & U/F_i \ar[r]
& H/F_i \ar[r]  & S \ar[r] &  1 
}
\]
Assuming (by induction) that the bottom extension splits, we obtain that the corresponding obstruction 
class in the first non-abelian cohomology $H^1\big(S,U/F_i\big)$ is trivial. 
Hence, by exactness of 
\[
H^1\big(S,F_i/F_{i+1}\big)
\longrightarrow H^1\big(S,U/F_{i+1}\big)
\longrightarrow H^1\big(S,U/F_i\big),
\]
we get that the obstruction class for the splitting of the top extension lies in the image of 
\[
H^1\big(S,F_i/F_{i+1}\big)
\longrightarrow H^1\big(S,U/F_{i+1}\big)\,.
\]
We conclude by using the vanishing of group cohomology of a finite group in a finite dimensional representation. 
\end{proof}

The above Lemma tells us in particular that $G(\varphi,\kk)\simeq\on{Ker}(\varphi)(\kk)\rtimes S$, and justifies 
the following definition from \cite[\S1.2]{En}\footnote{In \cite{En}, Enriquez speaks about \textit{relative formality}. 
We prefer to speak about \textit{relative filtered-formality} in order to remain consistent with our conventions in the 
absolute case $S=1$ (recall that we were following the convention from \cite{SW} in the absolute case). }. 
\begin{definition}
If $S$ is finite, we say that the surjective group morphism $\varphi\,:\,G\to S$ with finitely generated kernel 
is \textit{(relatively) filtered-formal} if there exists a group isomorphism 
\[
G(\kk,\varphi) \tilde\longrightarrow \on{exp}\big(\hat{\on{gr}}\on{Lie}\on{Ker}\varphi(\kk)\big)\rtimes S
\]
over $S$. This is equivalent to having an $S$-equivariant filtered-formality isomorphism 
\[
\on{Ker}\varphi(\kk) \tilde\longrightarrow \hat{\on{gr}}\on{Lie}\on{Ker}\varphi(\kk)\,.
\] 
\end{definition}

\begin{example}
The surjective morphism $\on{B}_n \twoheadrightarrow \SG_n$, where $\on{B}_n$ is the standard 
$n$ strands braid group is filtered-formal. This morphism, or rather the exact sequence 
\[
1 \longrightarrow \on{PB}_n \longrightarrow \on{B}_n \longrightarrow \SG_n \longrightarrow 1\,,
\]
can be deduced from the covering map $\on{Conf}(\C,n)\to \on{Conf}(\C,n)/\mathfrak{S}_n$. 
Note that filtered-formality of smooth complex algebraic varieties is proven in \cite{Mor78} 
in a functorial way, which implies in particular the wanted relative filtered-formality.  
An explicit filtered-formality isomorphism was first given in \cite{Kohno3} when $\kk = \C$ 
(in terms of the monodromy of the KZ connection) and then in \cite{DrGal} for $\kk = \Q$ (using 
an associator). 
We also refer to \cite[Example 1.5]{HainMalcev} for interesting considerations about this example. 
More precisely, one has an $\mathfrak{S}_n$-equivariant isomorphism 
${\tmop{PB}}_{n}(\kk)\tilde\to \exp(\hat{\mathfrak{t}}_{n})$. 
\end{example}

\begin{example}
Let $M\in\mathbb{N}$ be a positive integer. 
From the covering map $\on{Conf}(\C^\times,n,M)\to \on{Conf}(\C^\times,n)/\mathfrak{S}_n$ one also gets an exact sequence 
\[
1 \longrightarrow \on{PB}_n^M \longrightarrow \on{B}_n^1 \longrightarrow S \longrightarrow 1\,,
\]
where $S:=(\Z/M\Z)^n\rtimes\SG_n$. 
It follows from \cite[\S1.3--1.6]{En} that the surjective morphism $\on{B}_n^1 \twoheadrightarrow S$ is filtered-formal. 
More precisely, Enriquez exhibits an $S$-equivariant isomorphism ${\tmop{PB}}_{n}^M(\kk)\tilde\to \exp(\hat{\mathfrak{t}}_{n}^M)$. 
\end{example}

\subsection{Subgroups of $\on{B}_{1,n}$}

For $\tau\in\mathfrak H$, let $U_{\tau,n,\Gamma}\subset\C^n-{\rm Diag}_{\tau,n,\Gamma}$ be the 
open subset of all $\zz=(z_1,\dots,z_n)$ of the form $z_i=a_i+\tau b_i$, where $0<a_1<\cdots<a_n<1/M$ 
and $0<b_n<\cdots<b_1<1/N$. If $\zz_0\in U_{\tau,n,\Gamma}$, then it both defines a point in the 
$\Gamma$-twisted configuration space $\textrm{Conf}(E_{\tau,\Gamma},n,\Gamma)$ and in the (non-twisted) 
unordered configuration space $\textrm{Conf}(E_{\tau,\Gamma},[n])$.

Recall that the map 
\[
\textrm{Conf}(E_{\tau,\Gamma},n,\Gamma)\twoheadrightarrow\textrm{Conf}(E_{\tau,\Gamma},[n])
\]
is a covering map with structure group $\Gamma^n\rtimes\mathfrak{S}_n$. 
Hence we get a short exact sequence 
\[
1\longrightarrow \on{PB}_{1,n}^\Gamma\longrightarrow \on{B}_{1,n} 
\overset{\varphi_n}{\longrightarrow} \Gamma^n\rtimes\mathfrak{S}_n\longrightarrow 1\,,
\]
where \gls{PB1nG}$:=\pi_1(\textrm{Conf}(E_{\tau,\Gamma},n,\Gamma),\zz_0)$ and 
\gls{B1n}$:=\pi_1\big(\textrm{Conf}(E_{\tau,\Gamma},[n]),\zz_0\big)$. 

We will also consider \gls{PB1n}$=\pi_1\big(\textrm{Conf}(E_{\tau,\Gamma},n),\zz_0\big)$, and the 
short exact sequence 
\[
1\longrightarrow \on{PB}_{1,n}^\Gamma\longrightarrow \on{PB}_{1,n} \longrightarrow \Gamma^n\longrightarrow 1
\]
associated with the $\Gamma^n$-covering map 
\[
\textrm{Conf}(E_{\tau,\Gamma},n,\Gamma)\twoheadrightarrow\textrm{Conf}(E_{\tau,\Gamma},n)\,.
\]

Our main aim in this Section is to construct a relative filtered-formality isomorphism for  
\[
\on{B}_{1,n} \twoheadrightarrow \Gamma^n\rtimes\mathfrak{S}_n\,.
\]
Moreover, we will have an explicit description of the relative completion in terms of the 
Lie algebra $\t_{1,n}^\Gamma$. 


\subsection{The monodromy morphism $\on{B}_{1,n}\to\exp(\hat{\t}_{1,n}^\Gamma)\rtimes(\Gamma^n\rtimes\mathfrak{S}_n)$}

The monodromy of the flat $\exp(\hat{\t}_{1,n}^\Gamma)\rtimes(\Gamma^n\rtimes\mathfrak{S}_n)$-bundle 
$(\mathcal{P}_{(\tau,\Gamma),[n]},\nabla_{(\tau,\Gamma),[n]})$ on $\mathrm{Conf}(E_{\tau,\Gamma},[n])$ provides 
us with a group morphism 
\[
\mu_{\zz_0,(\tau,\Gamma),[n]}\,:\,
\on{B}_{1,n}\longrightarrow\exp(\hat{\t}_{1,n}^\Gamma)\rtimes(\Gamma^n\rtimes\mathfrak{S}_n)\,.
\]

This actually fits into a morphism of short exact sequences
\[
\xymatrix{
		1 \ar[r] &
		\on{PB}_{1, n}^{\Gamma} \ar[r] \ar[d] &
		\on{B}_{1,n} \ar[r] \ar[d] &
		\Gamma^n \rtimes \mathfrak{S}_n \ar[r] \ar@{=}[d] &
		1 \\
		1 \ar[r] &
		\exp(\hat{\t}_{1, n}^{\Gamma}) \ar[r] &
		\exp(\hat{\t}_{1, n}^{\Gamma})\rtimes(\Gamma^n \rtimes \mathfrak{S}_n) \ar[r] &
		\Gamma^n \rtimes \mathfrak{S}_n \ar[r] &
		1
}\,,
\]
where the first vertical morphism is the monodromy morphism 
\[
\mu_{\zz_0,\tau,n,\Gamma}\,:\,
\on{PB}_{1,n}^\Gamma\longrightarrow\exp(\hat{\t}_{1,n}^\Gamma)
\]
of associated with the flat $\exp(\hat{\t}_{1,n}^\Gamma)$-bundle 
$(\mathcal{P}_{\tau,n,\Gamma},\nabla_{\tau,n,\Gamma})$ on $\mathrm{Conf}(E_{\tau,\Gamma},n,\Gamma)$. 

Indeed, this comes from the fact that $\nabla_{(\tau,\Gamma),[n]}$ is obtained by descent, from 
$\nabla_{\tau,n,\Gamma}$ and using its equivariance properties (see \S\ref{sec6.2var}). 
More precisely, the monodromy of $\nabla_{(\tau,\Gamma),[n]}$ along a loop $\gamma$ based at $\zz_0$ 
in $\mathrm{Conf}(E_{\tau,\Gamma},[n])$ can be computed along the following steps: 
\begin{itemize}
\item First consider the unique lift $\tilde\gamma$ of $\gamma$ departing from 
$\zz_0\in \mathrm{Conf}(E_{\tau,\Gamma},n,\Gamma)$. Note that it ends at $g\cdot\zz_0$, 
$g\in\Gamma^n\rtimes\mathfrak{S}_n$. If $g=(g_1,\ldots,g_n)\in\Gamma^n$ and $\zz_0=(z_1,\ldots,z_n)$ 
we will simply write $g\cdot\zz_0:=(z_1^{g_1},\ldots z_n^{g_n})$.
\item Then compute the holonomy of $\nabla_{\tau,n,\Gamma}$ along $\tilde\gamma$: 
this is an element in $\exp(\hat{\t}_{1,n}^\Gamma)$, as $\nabla_{\tau,n,\Gamma}$ is defined on a principal 
$\exp(\hat{\t}_{1,n}^\Gamma)$-bundle obtained as a quotient of the trivial one on $\C^n-{\rm Diag}_{\tau,n,\Gamma}$ 
(see \S\ref{sec-confspaces}), that we abusively denote $\mu_{\zz_0,\tau,n,\Gamma}(\tilde\gamma)$.
\item Finally, $\mu_{\zz_0,(\tau,\Gamma),[n]}(\gamma)=g\mu_{\zz_0,\tau,n,\Gamma}(\tilde\gamma)$. 
\end{itemize}

Having such a morphism of exact sequences guarantees that it factors through a morphism
\[
\xymatrix{
		1 \ar[r] &
		{\hat{\on{PB}}_{1,n}^\Gamma}(\C) \ar[r] \ar[d] &
		{\hat{\on{B}}_{1,n}}(\varphi_n,\C) \ar[r] \ar[d] &
		\Gamma^n \rtimes \mathfrak{S}_n \ar[r] \ar@{=}[d] &
		1 \\
		1 \ar[r] &
		\exp(\hat{\t}_{1, n}^{\Gamma}) \ar[r] &
		\exp(\hat{\t}_{1, n}^{\Gamma})\rtimes(\Gamma^n \rtimes \mathfrak{S}_n) \ar[r] &
		\Gamma^n \rtimes \mathfrak{S}_n \ar[r] &
		1
}\,,
\]
where ${\hat{\on{B}}_{1,n}}(\varphi_n,\C)$ is the relative prounipotent completion of the morphism 
$\on{B}_{1,n}\to\Gamma^n\rtimes\mathfrak{S}_n$, and ${\hat{\on{PB}}_{1,n}^\Gamma}(\C)$ is the prounipotent 
completion of $\on{PB}_{1,n}^\Gamma$. 

We will call the vertical maps the completed monodromy morphisms. 

In the remainder of this Section we will prove that these completed monodromy morphisms are isomorphisms, 
exhibiting in particular a relative filtered-formality isomorphism for $\on{B}_{1,n}\to\Gamma^n\rtimes\mathfrak{S}_n$. 
\begin{theorem}\label{formalitythm}
The completed monodromy morphism 
\[
{\hat{\on{B}}_{1,n}}(\varphi_n,\C)\longrightarrow \exp(\hat{\t}_{1, n}^{\Gamma})\rtimes(\Gamma^n \rtimes \mathfrak{S}_n)
\]
is an isomorphism. Equivalently, the completed monodromy morphism 
\[
\hat{\mu}_{\zz_0,\tau,n,\Gamma}(\C)\,:\,{\hat{\on{PB}}_{1,n}^\Gamma}(\C)\longrightarrow \exp(\hat{\t}_{1, n}^{\Gamma})
\]
is an isomorphism. 
\end{theorem}

Our aim now is to prove Theorem \ref{formalitythm}. 
For this we will prove, as usual, that the induced morphism on Malcev Lie algebras 
\[
\Lie(\mu_{\zz_0,\tau,n,\Gamma})\,:\,\pb_{1,n}^\Gamma\to \hat{\t}_{1,n}^\Gamma
\]
is an isomorphism of filtered Lie algebras. 


\subsection{A morphism $\t_{1,n}^\Gamma\to {\rm gr}(\pb_{1,n}^\Gamma)$}

Let us start with a few algebraic facts about $\on{PB}_{1,n}$ and $\on{PB}_{1,n}^\Gamma$. 
The group $\on{PB}_{1,n}$ is generated by the $X_i$'s and $Y_i$'s ($i=1,\dots,n$), 
where $X_i$ (resp.~$Y_i$) is the class of the path given by $[0,1]\ni t\mapsto\zz_0+t\delta_i/M$ 
(resp.~$[0,1]\ni t\mapsto\zz_0+t\tau\delta_i/N$). 
One sees easily that $X_i^M$ (resp.~$Y_i^N$) is the class of the path given by 
$[0,1]\ni t\mapsto\zz_0+t\delta_i$ (resp.~$[0,1]\ni t\mapsto\zz_0+t\tau\delta_i$), so that 
$X_i^M$ and $Y_i^N$ are elements of $\on{PB}_{1,n}^\Gamma$.
$$
\includegraphics[scale=1]{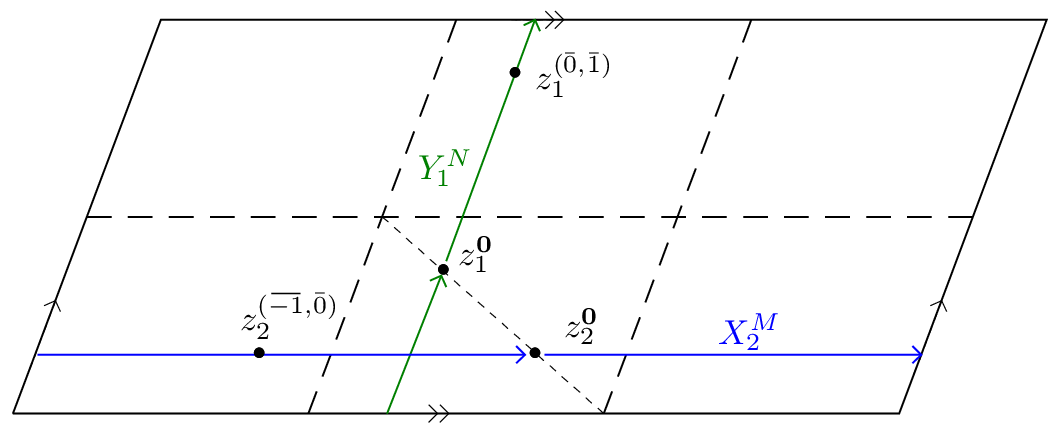}
$$
One has an obvious inclusion $\on{PB}_n\hookrightarrow \on{PB}_{1,n}^\Gamma$ coming from 
the identification of $\C$ with the fundamental domain 
\[
\{z=a+b\tau\in\C|0< a< \frac{1}{M},0< b<\frac{1}{N}\}
\]
of $E_{\tau,\Gamma}$. 

Recall that we write the composition of paths from left to right. Then one can check (by simply drawing) that the following relations are satisfied in $\on{PB}_{1,n}$: 
\begin{itemize}
\item[(T1)] $(X_i,X_j)=1=(Y_i,Y_j)$ ($i<j$), 
\item[(T2)] $(X_i,Y_j)=P_{ij}$, and is conjugated to $(X_j^{-1},Y_i^{-1})$ ($i<j$), 
\item[(T3)] $(X_1,Y_1^{-1})=P_{1n}\cdots P_{13} P_{12}$, 
\item[(T4)] $(X_i,P_{jk})=1=(Y_i,P_{jk})$ ($\forall i$, $j<k$), 
\item[(T5)] $(X_iX_j,P_{ij})=1=(Y_iY_j,P_{ij})$ ($i<j$). 
\end{itemize}
One also observes that $X_1\cdots X_n$ and $Y_1\cdots Y_n$ are central in $\on{PB}_{1,n}$. 

Now it follows from the geometric description of $\on{PB}_{1,n}^\Gamma$ that it is generated by 
$X_i^M$, $Y_i^N$ ($i=1,\dots,n$), and $P_{ij}^\alpha:=X_j^{-p}Y_j^{-q}P_{ij}Y_j^{q}X_j^{p}$ 
($i<j$, $1\leq p\leq M$, $1\leq q\leq N$ and $\alpha=(\bar p,\bar q)$). 
One can for instance represent lifts of $X_3$, $Y_3$ and $P_{12}^{(\bar{1},\bar{1})}$ in 
$\textrm{Conf}(E_{\tau,\Gamma},n,\Gamma)$ as follows

$$
\includegraphics[scale=1]{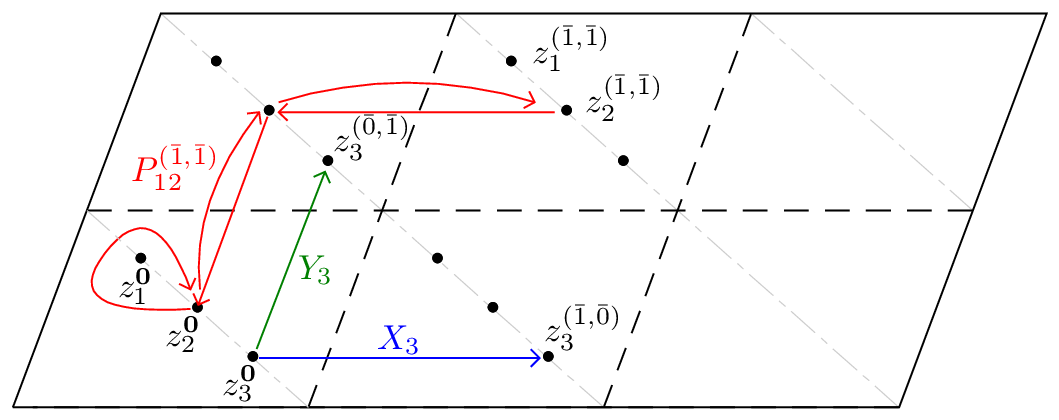}
$$

Observe that the standard descending filtration on $\hat{\t}_{1,n}^\Gamma$ 
coincides with the descending filtration coming from the grading of $\t_{1,n}^\Gamma$ defined 
in \S\ref{sec:deft1n}. 
\begin{proposition}\label{prop-gr-sigma}
There is a surjective graded Lie algebra morphism $p_n:\t_{1,n}^\Gamma\to {\rm gr}(\pb_{1,n}^\Gamma)$, 
sending 
\begin{itemize}
\item $x_i\longmapsto\sigma\big(\log(X_i^M)\big)$ for $i=1,\dots,n$,
\item $y_i\longmapsto\sigma\big(\log(Y_i^{N})\big)$ for $i=1,\dots,n$,
\item $t_{ij}^\alpha\longmapsto\sigma\big(\log(P_{ij}^\alpha)\big)$ for $i<j$,
\end{itemize}
where $\sigma$ denotes the symbol map $\pb_{1,n}^\Gamma\to{\rm gr}(\pb_{1,n}^\Gamma)$. 
\end{proposition}
\begin{proof}
It is sufficient to check that the defining relations of $\t_{1,n}^\Gamma$ are preserved by the above assignment. 
The relation $[x_i,x_j]=0=[y_i,y_j]$ is obviously preserved, thanks to (T1). 

Now using (T2) and the identity 
$$
(X^M,Y^N)=\prod_{i=0}^{M-1}X^{M-i+1}(\prod_{j=0}^{N-1}Y^j(X,Y)Y^{-j})X^{i-M-1}
$$ 
(which is true in the free group $F_2$, and thus in any group) with $X=X_i$ and $Y=Y_j$ ($i < j$), 
one obtains that $[x_j,y_i]=\sum_\alpha t_{ij}^\alpha$ is preserved. 
The same reasoning with $X=X_i$ and $Y=Y_j^{-1}$ ($i\neq j$) shows that 
$[x_i,y_j]=\sum_\alpha t_{ij}^\alpha$ is preserved as well. 

Using (T3) and the above identity with $X=X_1$ and $Y=Y_1^{-1}$, one also obtains that 
$[x_1,y_1]=-\sum_{\alpha}\sum_{j:1\neq j}t_{1j}^\alpha$ is preserved. 
Now it is obvious that the centrality of $\sum_i x_i$ and $\sum_iy_i$ is preserved, and thus it follows that 
$[x_i,y_i]=-\sum_{\alpha}\sum_{j:j\neq i}t_{ij}^\alpha$ is also preserved for any $i\in\{1,\dots,n\}$. 

For any $\alpha=(\bar p,\bar q)$ we compute 
\begin{eqnarray*}
(X_i^M,P_{jk}^\alpha) & = & X_i^MX_k^{-p}Y_k^{-q}P_{jk}Y_k^{q}X_k^{p}X_i^{-M}X_k^{-p}Y_k^{-q}P_{jk}^{-1}Y_k^{q}X_k^{p} \\
& = & X_k^{-p}(X_i^M,Y_k^{-q})Y_k^{-q}X_i^MP_{jk}X_i^{-M}Y_k^{q}(X_i^M,Y_k^{-q})^{-1}
Y_k^{-q}P_{jk}^{-1}Y_k^{q}X_k^{p} \\
& = & X_k^{-p}(X_i^M,Y_k^{-q})Y_k^{-q}P_{jk}Y_k^{q}(X_i^M,Y_k^{-q})^{-1}Y_k^{-q}P_{jk}^{-1}Y_k^{q}X_k^{p}\,.
\end{eqnarray*}
On the one hand, $\sigma\big(\log(X_i^M,P_{jk}^\alpha)\big)=[\sigma(\log(X_i^M)),\sigma(\log(P_{jk}^\alpha))]$, 
and one the other hand, the leading term of the $\log$ of the r.h.s.~lies in higher degree. 
Hence one obtains that $[x_i,t_{jk}^\alpha]=0$ is preserved. The proof that $[y_i,t_{jk}^\alpha]=0$ is preserved 
is identical, and the proof that $[x_i+x_j,t_{ij}^\alpha]=0=[y_i+y_j,t_{ij}^\alpha]$, $[t_{ij}^\alpha,t_{kl}^\beta]=0$ 
and $[t_{ij}^\alpha,t_{ik}^{\alpha+\beta}+t_{jk}^\beta]=0$ are preserved is similar. 
\end{proof}


\subsection{The filtered-formality of $\on{PB}_{1,n}^\Gamma$ (end of the proof of Theorem \ref{formalitythm})}

To prove that $\Lie(\mu_{\zz_0,\tau,n,\Gamma})$ is an isomorphism, it is sufficient to prove that 
it is an isomorphism on associated graded. According to Proposition \ref{prop-gr-sigma}, we simply 
have to prove that $\phi:={\rm gr}\Lie(\mu_{\zz_0,\tau,n,\Gamma})\circ p_n$ is an isomorphism of graded Lie algebras. 

We will actually be more specific and prove the following: 
\begin{lemma}
We have $\phi(x_{i}) =  y_{i}$, $\phi(y_{i}) = - 2\pi\i x_{i}+\tau y_i$ and 
$\phi(t_{ij}^{\alpha}) =  2\pi\i t_{ij}^{\alpha}$. 
In particular, $\phi$ is an automorphism. 
\end{lemma}
\begin{proof}
Recall (see the appendix for more details) that $\mu_{\zz_0,\tau,n,\Gamma}$ can be computed as follows. 
Let $F_{\zz_{0}}:U_{\tau}\to\exp(\hat{\t}_{1, n}^{\Gamma})$ be such that 
\[
\begin{cases}
(\partial/\partial z_{i})F_{\zz_{0}}(\zz) = K_{i}^{\Gamma}(\zz|\tau)F_{\zz_{0}}(\zz)\,, \\ 
F_{\zz_{0}}(\zz_{0})=1\,.
\end{cases}
\]
Then consider 
\[
H_{\tau,n}^{\Gamma} := 
\left\{\zz = (z_{1},...,z_{n}) | z_{i} = a_{i} + \tau b_{i}, 0<b_n<...<b_1<\frac{1}{N}\right\}
\]
and
\[
V_{\tau,n}^{\Gamma} := 
\left\{\zz = (z_{1},...,z_{n}) | z_{i} = a_{i} + \tau b_{i}, 0<a_1<...<a_n<\frac{1}{M}\right\}\,.
\]
Let $F_{\zz_{0}}^{H}$ (resp.~$F_{\zz_{0}}^{V}$) be the analytic prolongations 
of $F_{\zz_{0}}$ to $H_{\tau,n}^{\Gamma}$ (resp.~$V_{\tau,n}^{\Gamma}$). Then
\[
\mu_{\zz_0,\tau,n,\Gamma}(X_{i}^M) = F_{\zz_{0}}^{H}(\zz)F_{\zz_{0}}^{H}(\zz + \delta_{i})^{-1}
\quad 
\mathrm{and}
\quad
\mu_{\zz_0,\tau,n,\Gamma}(Y_{i}^N)e^{2\pi\i x_{i}} = 
F_{\zz_{0}}^{V}(\zz)F_{\zz_{0}}^{V}(\zz + \tau\delta_{i})^{-1}\,. 
\]
Knowing that $\on{log}F_{\zz_{0}}(\zz) = -\sum_{i} (z_{i} - z_{i}^{0})y_{i}$ + terms of degree $\geq 2$, we get 
\[
\on{log}\mu_{\zz_0,\tau,n,\Gamma}(X_{i}^M) = y_{i} \text{ + terms of degree } \geq 2
\]
and
\[
\on{log}\mu_{\zz_0,\tau,n,\Gamma}(Y_{i}^N) = - 2\pi\i x_{i} + \tau y_{i} \text{ + terms of degree } \geq 2\,.
\]
This gives us that $\phi(x_{i}) =  y_{i}$ and $\phi(y_{i}) = - 2\pi\i x_{i} + \tau y_i$. 

In order to compute $\on{log}\mu_{\zz_0,\tau,n,\Gamma}(P_{ij}^\alpha)$, which is also equal to 
$\on{log}\mu_{\zz_0,(\tau,\Gamma),n}(P_{ij}^\alpha)$, we will need to compute 
$\mu_{\zz_0,(\tau,\Gamma),n}(X_i)$, $\mu_{\zz_0,(\tau,\Gamma),n}(Y_i)$ and 
$\mu_{\zz_0,(\tau,\Gamma),n}(P_{ij})$: 

\begin{itemize}
\item As usual, and with our conventions, 
\[
\mu_{\zz_0,(\tau,\Gamma),n}(P_{ij})=\exp(2\pi\i t_{ij}^{\0}\text{ + terms of degree } \geq 3)\,, 
\]
where $\0=(\bar{0},\bar{0})$.
\item We also have 
\[
F_{\zz_{0}}^{H^{\Gamma}}(\zz)
=\mu_{\zz_0,(\tau,\Gamma),n}(X_{i})(\bar{-1},\bar{0})_iF_{\zz_{0}}^{H^{\Gamma}}(\zz + \frac{\delta_{i}}{M}) \,,
\]
which implies that
\[
\mu_{\zz_0,(\tau,\Gamma),n}(X_{i})\in\exp(\t_{1,n}^\Gamma)(\bar{1},\bar{0})_i\,.
\]
\item We finally have 
\[
F_{\zz_{0}}^{V}(\zz)
=\mu_{\zz_0,(\tau,\Gamma),n}(Y_{i})(\bar{0},\bar{-1})_ie^{{2\pi\i \over N}x_{i}}F_{\zz_{0}}^{V}(\zz + \frac{\tau\delta_{i}}{N}) \,,
\]
which implies that 
\[
\mu_{\zz_0,(\tau,\Gamma),n}(Y_{i})\in\exp(\t_{1,n}^\Gamma)(\bar{0},\bar{1})_i\,.
\]
\end{itemize}
Hence, if $\alpha = (\bar{p},\bar{q})\in\Gamma$, then 
\[
\mu_{\zz_0,(\tau,\Gamma),n}(X_i^{p}Y_i^{q})=g(\bar{p},\bar{q})_i\,,
\]
with $g\in\exp(\t_{1,n}^\Gamma)$, and 
\[
\mu_{\zz_0,(\tau,\Gamma),n}(Y_j^{-q}X_i^{-p})=(\bar{-p},\bar{-q})_ig^{-1}\,.
\]
Therefore 
\begin{eqnarray*}
\mu_{\zz_0,(\tau,\Gamma),n}(P_{ij}^\alpha) 
& = & g(\bar{p},\bar{q})_i\exp(2\pi\i t_{ij}^{\0}
\text{ + terms of degree } \geq 3)(\bar{-p},\bar{-q})_ig^{-1} \\
& = & g\exp(2\pi\i t_{ij}^\alpha\text{ + terms of degree } \geq 3)g^{-1}\,.
\end{eqnarray*}
This shows that $\on{log}\mu_{\zz_0,(\tau,\Gamma),n}(P_{ij}^\alpha)=2\pi\i t_{ij}^\alpha\text{ + terms of degree } \geq 3$, 
so that $\phi(t_{ij}^\alpha)=2\pi\i t_{ij}^\alpha$. This ends the proof of the Lemma. 
\end{proof}

 Finally, if we denote $\hat{\overline{\on{PB}}}_{1,n}^\Gamma(\C):=\hat\pi_1(\textrm{C}(E_{\tau,\Gamma},n,\Gamma),
  \bar{\zz}_0)(\mathbb{C})$, where $\bar{\zz}_0$ is the image of $\zz_0$\break by the projection 
  $\textrm{Conf}(E_{\tau,\Gamma},n) \to \textrm{C}(E_{\tau,\Gamma},n)$, then
   the isomorphism $\hat\mu_{\zz_0,\tau,n,\Gamma}(\C)$ descends to an isomorphism 
   $\bar{\hat\mu}_{\bar{\zz}_0,\tau,n,\Gamma}(\C): \hat{\overline{\on{PB}}}_{1,n}^\Gamma(\C)\to\exp(\hat{\bar{\t}}_{1,n}^\Gamma)$.\\

Now let $\overline{\on{B}}_{1, n}$ be the fundamental group $\pi_1(\tmop{C}(E_{\tau, \Gamma},[ n ]),[ \bar{\mathbf{z}}_0 ]) $. 
By considering the short exact sequence
\[ 
1 \longrightarrow \overline{\tmop{PB}}_{1, n}^{\Gamma} \longrightarrow \overline{\on{B}}_{1, n}
   \xrightarrow{\bar\varphi_n}(\Gamma^n/\Gamma) \rtimes \mathfrak{S}_n \longrightarrow 1, 
\]
we deduce that the map 
\[ 
	\hat{ \overline{\on{B}}}_{1, n} (\bar\varphi_n, \mathbb{C}) \longrightarrow \exp
   (\hat{\mathfrak{ \bar{\t}}}_{1, n}^{\Gamma}) \rtimes \big((\Gamma^n/\Gamma) \rtimes
   \mathfrak{S}_n\big) 
\]
is also relatively filtered-formal. 
In conclusion, we obtain the summarizing commutative cube 
$$ 
\xymatrix{
    \hat{\on{PB}}_{1, n}^{\Gamma}(\mathbb{C}) \ar[rr]^{\simeq} \ar@{->>}[dd] \ar@{^{(}->}[dr] 
		&	& \exp(\hat{\mathfrak{\t}}_{1, n}^{\Gamma} )  \ar@{^{(}->}[dr] \ar@{->>}[dd] |!{[dl];[dr]}\hole \\
    & \hat{\on{B}}_{1, n} (\varphi_n, \mathbb{C}) \ar[rr]^{\simeq} \ar@{->>}[dd] 
		& & \exp(\hat{\mathfrak{ \t}}_{1, n}^{\Gamma} ) \rtimes (\Gamma^n \rtimes\mathfrak{S}_n) \ar@{->>}[dd] \\
    \hat{\overline{\on{PB}}}_{1,n}^\Gamma(\C) \ar[rr]^{~~~~\simeq} |!{[ur];[dr]}\hole \ar@{^{(}->}[dr] 
		& & \exp(\hat{\bar{\t}}_{1,n}^\Gamma) \ar@{^{(}->}[rd] \\
    & \hat{ \overline{\on{B}}}_{1, n} (\bar\varphi_n, \mathbb{C}) \ar[rr]^{\simeq} 
		& & \exp(\hat{\mathfrak{ \bar{\t}}}_{1, n}^{\Gamma} ) \rtimes \big((\Gamma^{n}/\Gamma) \rtimes\mathfrak{S}_{n}\big).
  }
$$

\section{Representations of Cherednik algebras}
\label{Representations of Cherednik algebras}

\subsection{The Cherednik algebra of a wreath product}

In this paragraph $\Gamma$ is any finite group such that 
$\Gamma\subset{\rm Aut}(\C)$, $\underline{k}=(k_\alpha)_\alpha\in\C^\Gamma$ is such 
that $k_{\alpha}=k_{-\alpha}$ and $G:=\Gamma\wr \mathfrak{S}_n$. We define the Cherednik algebra 
$H_n^\Gamma(\underline{k})$ as the quotient of the algebra $\C\<x_1,\dots,x_n,y_1,\dots,y_n\>\rtimes\C[G]$ 
by the relations 
\begin{itemize}
\item $\sum_i\x_i=\sum_i\y_i=0\,$
\item $[\x_i,\x_j]=0=[\y_i,\y_j],$
\item $[\x_i,\y_j]=\frac1n-\sum_{\alpha\in\Gamma}k_\alpha s_{ij}^\alpha\,(i\neq j),$
\end{itemize}
where $s_{ij}^\alpha=(\alpha_i-\alpha_j)s_{ij}$, and $s_{ij}$ is the permutation 
of $i$ and $j$. 
\begin{center}

\end{center}
\begin{remark}
Since $\Gamma\subset{\rm Aut}(\C)$, $H_n^\Gamma(\underline{k})$ admits a geometric construction. 
Define $X:=\{\zz\in\C^n|\sum_iz_i=0\}$ and consider the following action of $G$ on it: $\mathfrak{S}_n$ 
acts in an obvious way and $$\alpha_i(\zz)=(\alpha^{(i)}-\frac1n\sum_j\alpha^{(j)})(\zz),$$ where 
$\alpha^{(k)}$ is the action of $\alpha\in\Gamma$ on the $k$-th factor of $\C^n$. 
Following \cite{Et2} one can construct a Cherednik algebra $H_{1,\underline{k},0}(X,G)$ on $X/G$. 
It can be defined as the subalgebra of $\Diff(X)\rtimes\C[G]$ generated by the function algebra 
$\mathcal O_X$, the group $G$ and the Dunkl-Opdam operators $D_i-D_j$, where 
$$
D_i=\partial_{z_i}
+\sum_{\substack{j:j\neq i\\\alpha\in\Gamma}}k_\alpha\frac{1-s_{ij}^\alpha}{(-\alpha)(z_i)-\alpha(z_j)}\,.
$$
One can then prove that there is a unique isomorphism of algebras 
$H_n^\Gamma(\underline{k})\to H_{1,\underline{k},0}(X,G)$ defined by 
\begin{align*}
\x_i\longmapsto & z_i, \\
\y_i\longmapsto & D_i-\frac1n\sum_{j}D_j, \\
G\ni g\longmapsto & g.
\end{align*}

\end{remark}

\subsection{Morphisms from $\bar\t_{1,n}^\Gamma$ to the Cherednik algebra}

\begin{proposition}
For any $a,b\in\C$ there is a morphism of Lie algebras 
$\phi_{a,b}:\bar\t_{1,n}^\Gamma\to H_n^\Gamma(\underline{k})$ defined by
\begin{align*}
\bar x_i & \longmapsto   a\x_i\, \\
\bar y_i & \longmapsto   b\y_i\,, \\
\bar t_{ij}^\alpha & \longmapsto  ab\left(\frac1n-k_\alpha s_{ij}^\alpha\right)\,.
\end{align*}
\end{proposition}
\begin{proof}
Straightforward from the alternative presentation of $\bar\t_{1,n}^\Gamma$ in Lemma \ref{lem:pres1}.
\end{proof}

Hence any representation $V$ of $H_n^\Gamma(\underline{k})$ yields a family of flat connections 
$\nabla_{a,b}^{(V)}$ over the configuration space $\textrm{C}(E,[n],\Gamma)$.

\subsection{Monodromy representations of Hecke algebras}

Let $E$ be an elliptic curve and $\tilde E\to E$ the $\Gamma$-covering as in \S\ref{sec-confspaces}. Define 
$X=\tilde E^n/\tilde E$ and $G=(\Gamma\wr \mathfrak{S}_n)/\Gamma^{\rm diag}$. Then the set $X'\subset X$ of points with trivial 
stabilizer is such that $X'/G=\textrm{C}(E,[n],\Gamma)$. 

Let us recall from \cite{Et2} the construction of the Hecke algebra $\mathcal H_n^\Gamma(q,\underline{t})$ of $X/G$. 
It is the quotient of the group algebra of the orbifold fundamental group $\bar B_{1,n}^\Gamma$ of 
$\textrm{C}(E,[n],\Gamma)$ by the additional relations $(T_\alpha-q^{-1}t_\alpha)(T_\alpha+q^{-1}t_\alpha^{-1})=0$, 
where $T_\alpha$ is an element of $\bar B_{1,n}^\Gamma$ homotopic as a free loop to a small loop around the 
divisor $Y_\alpha:=\cup_{i\neq j}\{z_i=\alpha\cdot z_j\}$ in $X/G$, in the counterclockwise 
direction.\footnote{Here the sugroup of $G$ acting trivially on $Y_\alpha$ is the order $2$ cyclic 
subgroup generated by $s_{ij}^\alpha$. }

Let us consider the flat connection $\nabla_{a,b}^{(V)}$ and set 
$$
q=e^{-2\pi\i ab/n}\,,~~t_\alpha=e^{-2\pi\i k_\alpha ab}\,.
$$
Then the monodromy representation $\bar B_{1,n}^\Gamma\to GL(V)$ of $\nabla_{a,b}^{(V)}$ obviously gives a 
representation of $\mathcal H_n^\Gamma(q,\underline{t})$ either if $V$ is finite dimensional or if 
$a,b$ are formal parameters. In particular, taking $a=b$ a formal parameter and $V=H_n^\Gamma(\underline{k})$, 
one obtains an algebra morphism 
$$
\mathcal H_n^\Gamma(q,\underline{t})\longrightarrow H_n^\Gamma(\underline{k})[[a]]\,.
$$
We do not know if this morphism is an isomorphism upon inverting $a$.

\subsection{The modular extension of $\phi_{a,b}$.}

Now assume that $a,b\ne 0$.

\begin{proposition} The Lie algebra morphism $\phi_{a,b}$ can be extended 
to the algebra $U(\bar\t_{1,n}^\Gamma\rtimes
{\mathfrak d}^\Gamma)\rtimes G$ by the following formul\ae: 
$$
\phi_{a,b}(s_{ij}^\alpha)=s_{ij}^\alpha\,,
$$ 
$$
\phi_{a,b}(d)=\frac{1}{2}\sum_i ({\rm x}_i{\rm y}_i+{\rm y}_i{\rm x}_i)\,, \quad 
\phi_{a,b}(X)=-\frac{1}{2}ab^{-1}\sum_i {\rm x}_i^2\,,
$$
$$
\phi_{a,b}(\Delta_0)=\frac{1}{2}ba^{-1}\sum_i {\rm y}_i^2\,,
\quad 
\phi_{a,b}(\xi_{s,\gamma})=-a^{s-1}b^{-1}\sum_{i<j}(\gamma \cdot ({\rm x}_i-{\rm x}_j))^{s}\,.
$$
\end{proposition}
Thus, the flat connections $\nabla^{\Gamma}_{a,b}$ extend to flat connections
on ${\mathcal M}^{\Gamma}_{1,[n]}$.
\begin{proof}
The proof is a straightforward calculation.
\end{proof}

\appendix

\section{Conventions}

In this appendix we spell out our conventions regarding, fundamental groups, 
covering spaces, principal bundles, and monodromy morphisms. 

\subsection{Fundamental groups}

Our convention is that we read the concatenation of paths from left to right. 
For instance, if $X$ is a space, $p$ is a path from $x$ to $y$ in $X$, and $q$ 
is a path from $y$ to $z$ in $X$, then we write $pq$ for the concatenated path, 
going from $x$ to $z$. 

\subsection{Covering spaces and group actions}\label{app-cover}

Our convention is that the group of deck transformations acts from the left. 
Apart from the case of principal bundles (see next \S), group actions will always be 
from the left. We will often use $\cdot$ for such a left action. 

\medskip

The situation we are interested in is the one of a discrete group $H$ acting properly 
discontinuously from the left on a space $Y$, with quotient space $X=H\backslash Y$, 
so that the quotient map $Y\to X$ is a covering map. 

\medskip

We thus have a short exact sequence 
\[
1\to \pi_1(Y,y)\to \pi_1(X,x)\to H\to 1
\]
of groups, where $y\in Y$ and $x=H\cdot y\in X$ is its projection. Note that the surjective map 
$\pi_1(X,x)\to H$ sends (the class of) a loop $\gamma$ based at $x$ to $h_\gamma$, which is 
defined as follows: $\tilde{\gamma}(1)=h_\gamma\cdot\tilde{\gamma}(0)$, where $\tilde{\gamma}$ 
is a path lifting (uniquely) $\gamma$ to $Y$ and such that $\tilde{\gamma}(0)=y$. 
For the sake of completeness, let us check that this is indeed a group homomorphism. 
\begin{proof}
We have 
\[
h_{\gamma_1\gamma_2}\cdot y=\widetilde{\gamma_1\gamma_2}(1)=\tilde{\tilde{\gamma_2}}(1)\,,
\]
where $\tilde{\tilde{\gamma_2}}=h_{\gamma_1}\cdot\tilde{\gamma_2}$ is the (unique) lift of 
$\gamma_2$ such that $\tilde{\tilde{\gamma_2}}(0)=\tilde{\gamma_1}(1)=h_{\gamma_1}\cdot y$. 
Therefore, $h_{\gamma_1\gamma_2}=h_{\gamma_1}h_{\gamma_2}$. 
\end{proof}

\subsection{Principal bundles and descent}

Let $G$ be a group. 
All principal $G$-bundles (apart from covering spaces, see above) are right principal $G$-bundles. 
Let $\mathcal P$ be a principal $G$-bundle over $X$, so that $\mathcal{P}/G=X$. 

\medskip

Let us assume that $X=H\backslash Y$, where $H$ is a discrete group acting on $Y$. 
We now describe a way of constructing a $G$-bundle on the quotient space $X$ from the trivial 
$G$-bundle $\tilde{\mathcal{P}}:=Y\times G$ on $Y$, by means of non-abelian $1$-cocycles. 

\medskip

A left $H$-action on $\tilde{\mathcal{P}}$, compatible with the one on $Y$, is given as follows: 
\[
h\cdot(y,g)=\big(h\cdot y,c_{h}(y)g\big)\,,\quad c_{h}(y)\in G
\]
The property of being a left action is equivalent to the non-abelian $1$-cocycle identity
\[
c_{h_1h_2}(y)=c_{h_1}(h_2\cdot y)c_{h_2}(y)\,.
\]

\subsection{Monodromy and group actions}

Let us start with the monodromy in the case of a trivial principal $G$-bundle 
$\tilde{P}=Y\times G$ on a manifold $Y$ equipped with a flat connection $\nabla=d-\omega$. 
Here $\omega$ is a one-form on $Y$ with values in $\mathfrak{g}=\on{Lie}(G)$, and $G$ is 
assumed to be prounipotent. 

\medskip

Let $\gamma:[0,1]\to Y$ be a differentiable path, and consider its (unique) horizontal lift 
$\tilde{\gamma}=(\gamma,g):[0,1]\to\tilde{\mathcal{P}}$ such that $g(0)=1_G$. 
We define the monodromy $\mu(\gamma):=g(1)^{-1}$. 
\begin{remark}
Observe that if $(\gamma,\tilde{g})$ is another lift so that $\tilde{g}=g_0\in G$, then 
$\tilde{g}(t)=g(t)g_0$ (by unicity of horizontal lifts), and thus 
$\mu(\gamma)=\tilde{g}(0)\tilde{g}(1)^{-1}$. 
\end{remark}
Again, for the sake of completeness, we check that $\mu$ is a morphism, in the sense that 
it sends the concatenation of paths to the product in $G$. 
\begin{proof}
Let $\gamma_1,\gamma_2$ be composable paths in $Y$, and let $g_1,g_2$ determine composable 
horizontal lifts. Then 
\begin{eqnarray*}
\mu(\gamma_1\gamma_2) & = & (g_1g_2)(0)(g_1g_2)(1)^{-1}=g_1(0)g_2(1)^{-1} 										\\
											& = & g_1(0)g_1(1)^{-1}g_2(0)g_2(1)^{-1}=\mu(\gamma_1)\mu(\gamma_2)\,.
\end{eqnarray*}
\end{proof}

\medskip

Let us now assume that $Y$ is acted on properly discontinuously from the left by a discrete 
group $H$, that also acts in a compatible way on $\tilde{P}$ thanks to a non-abelian 
$1$-cocycle $c:H\times Y\to G$ (see previous \S~above). 
We borrow the notation from \S\ref{app-cover}, and assume that $\tilde{P}$ is equipped with 
an $H$-equivariant flat connection, that therefore descends to a flat connection on $\mathcal{P}$ 
We define a monodromy morphism 
\begin{eqnarray*}
\mu_x\,:\,\pi_1(X,x)	& \longrightarrow & G 			\\
				\gamma				& \longmapsto 		& \mu(\tilde{\gamma})c_{h_\gamma}(y)\,,
\end{eqnarray*}
where $\tilde{\gamma}$ is the lift of $\gamma$ along the quotient map $Y\to X$ such that 
$\tilde{\gamma}(0)=y$. 
Let us again check, for the sake of completeness, that $\mu_x$ is indeed a group morphism. 
\begin{proof}
Recall that for every loop $\gamma$ based at $x$, $\tilde{\gamma}(1)=h_\gamma\cdot y$. 
Hence, if $\gamma_1,\gamma_2$ are loops based at $x$, then 
$\widetilde{\gamma_1\gamma_2}=\tilde{\gamma_1}\tilde{\tilde{\gamma_2}}$, with 
$\tilde{\tilde{\gamma_2}}=h_{\gamma_1}\cdot \tilde{\gamma_2}$. 
Therefore
\begin{eqnarray*}
\mu_x(\gamma_1\gamma_2) 
& = & \mu(\widetilde{\gamma_1\gamma_2})c_{h_{\gamma_1\gamma_2}}(y) 			\\
& = & \mu(\widetilde{\gamma_1\gamma_2})c_{h_{\gamma_1}h_{\gamma_2}}(y)	\\
& = & \mu(\tilde{\gamma_1})\mu(h_{\gamma_1}\cdot \tilde{\gamma_2})
			c_{h_{\gamma_1}}(h_{\gamma_2}\cdot y)c_{h_{\gamma_2}}(y) \\
& = & \mu(\tilde{\gamma_1})c_{h_{\gamma_1}}(y)\mu(\tilde{\gamma_2})c_{h_{\gamma_1}}(h_{\gamma_2}\cdot y)^{-1}
			c_{h_{\gamma_1}}(h_{\gamma_2}\cdot y)c_{h_{\gamma_2}}(y) \\
& = & \mu_x(\gamma_1)\mu_x(\gamma_2)
\end{eqnarray*}
Here we made used of the (easy) fact that, if the flat connection is equivariant, 
then so is the monodromy map $\mu$: $\mu(h\cdot\gamma)=c_h(\gamma(0))\mu(\gamma)c_h(\gamma(1))^{-1}$. 
\end{proof}

\clearpage

\section*{List of notation}

\printnoidxglossaries


\clearpage

\begin{thebibliography}{EnGh}

\bibitem{BZ} R.~Bezrukavnikov, 
\textit{Koszul DG-algebras arising from configuration spaces}, 
Geom. and Funct. Analysis \textbf{4} (1992), no. 2, 119--135. 

\bibitem{B1} D.~Bernard, 
{\it On the Wess-Zumino-Witten model on the torus}, 
Nucl. Phys. {\bf B 303} (1988), 77--93. 



\bibitem{BE}
A.~Braverman, P.~Etingof, \& M.V.~Finkelberg,
\textit{Cyclotomic double affine Hecke algebras (with an appendix by Hiraku Nakajima and Daisuke Yamakawa)},
preprint \url{arXiv:1611.10216}. 

\bibitem{M2}
J.~Broedel, N.~Matthes, G.~Richter \& O.~Schlotterer,
\textit{Twisted elliptic multiple zeta values and non-planar one-loop open-string amplitudes}, 
Journal of Physics A: Mathematical and Theoretical \textbf{51} (2018), no. 28, 49pp. 

\bibitem{BL}
F.~Brown \& A.~Levin, 
\textit{Multiple elliptic polylogarithms},
preprint \url{arXiv:1110.6917}. 

\bibitem{CEE}
D.~Calaque, B.~Enriquez \& P.~Etingof, 
\textit{Universal KZB equations: the elliptic case}, 
in \textit{Algebra, Arithmetic and Geometry Vol I: In honor of Yu.~I.~Manin (Y.~Tschinkel \& Y.~Zarhin eds.)}, Progress in Mathematics \textbf{269} (2010), 165--266. 

\bibitem{CG-toappear}
D.~Calaque \& M.~Gonzalez, 
\textit{Ellipsitomic associators}, 
preprint \url{arXiv:2004.07271}.




\bibitem{DrGal}
V.~Drinfeld, 
\textit{On quasitriangular quasi-Hopf algebras and a group closely connected with $\on{Gal}(\bar\Q/\Q)$}, 
Leningrad Math. J. \textbf{2} (1991), 829--860.

\bibitem{En}
B.~Enriquez, 
\textit{Quasi-reflection algebras and cyclotomic associators}, 
Selecta Mathematica (NS) \textbf{13} (2008), no. 3, 391--463. 

\bibitem{En2}
B.~Enriquez, 
\textit{Elliptic associators}, 
Selecta Mathematica (NS) \textbf{20} (2014), no. 2, 491--584. 

\bibitem{En3}
B.~Enriquez, 
\textit{Analogues elliptiques des nombres multiz\'etas}, 
Bulletin de la SMF \textbf{144} (2016), no. 3, 395--427. 

\bibitem{En4}
B.~Enriquez, 
\textit{Flat connections on configuration spaces and formality of braid groups of surfaces}, 
Adv. in Math. \textbf{252} (2014), 204--226. 

\bibitem{EE}
B.~Enriquez \& P.~Etingof, 
\textit{Quantization of classical dynamical $r$-matrices with nonabelian base}, 
Comm. Math. Phys. \textbf{254} (2005), no. 3, 603--650. 


\bibitem{Et2}
P.~Etingof, 
\textit{Cherednik and Hecke algebras of varieties with a finite group action},
Moscow Math. J. \textbf{17} (2017), no. 4, 635--666. 

\bibitem{ES}
P.~Etingof \& O.~Schiffmann, 
\textit{Twisted traces of intertwiners for Kac-Moody algebras and classical dynamical $r$-matrices corresponding to generalized Belavin-Drinfeld triples},
Mathematical Research Letters \textbf{6} (1999), no. 6, 593--612. 

\bibitem{FP}
L.~Feher and B.~Pusztai, 
{\it Generalizations of Felder's elliptic dynamical r-matrices associated with twisted loop algebras of self-dual Lie algebras}, 
Nucl.Phys. \textbf{B621} (2002), 622--642. 

\bibitem{FICM}
G.~Felder, 
\textit{Conformal field theory and integrable systems associated to elliptic curves}, 
in \textit{Proceedings of the ICM, Vol. 1,2 (Z\"urich, 1994)}, 1247--1255, Birkh\"auser, Basel, 1995. 



\bibitem{Go}
A.B.~Goncharov, 
\textit{Multiple $\zeta$-values, Galois groups, and geometry of modular varieties}, 
in \textit{European Congress of Mathematics, Vol. 1 (Barcelona, 2000)}, 361--392, Birkh\"auser, Basel, 2001.

\bibitem{G1}
M.~Gonzalez, 
\textit{Contributions to the theory of associators}, 
PhD Thesis, available at \url{https://hal.archives-ouvertes.fr/tel-02541848}. 

\bibitem{HainMalcev}
R.~Hain, 
\textit{The Hodge de Rham theory of relative Malcev completion}, 
Annales scientifiques de l'\'Ecole Normale Sup\'erieure (S\'erie 4) \textbf{31} (1998), no. 1 , 47--92. 

\bibitem{Hain}
R.~Hain, 
\textit{Notes on the universal elliptic KZB equation}, 
Pure and Applied Mathematics Quarterly \textbf{16} (2020), no. 2, 229--312. 

\bibitem{HM} R.~Hain \& M.~Matsumoto, 
\textit{Relative pro-$\ell$ completions of mapping class groups},
Journal of Algebra \textbf{321} (2009), no. 11, 3335--3374. 



\bibitem{Kohno3}
T.~Kohno, 
\textit{Monodromy representations of braid groups and Yang-Baxter equations}, 
Ann. Inst. Fourier \textbf{37} (1987), 139--160. 

\bibitem{Kohno}
T.~Kohno, 
\textit{On the holonomy Lie algebra and the nilpotent completion of the fundamental group of the complement of hypersurfaces}, 
Nagoya Mathematical Journal \textbf{92} (1983), 21--37. 


\bibitem{LR}
A.~Levin \& G.~Racinet, 
\textit{Towards multiple elliptic polylogarithms}, 
preprint \url{arXiv:math/0703237}. 

\bibitem{Mor78}
J. W.~Morgan,
\textit{The algebraic topology of smooth algebraic varieties}, 
Publications Math\'ematiques de l'IHÉS \textbf{48} (1978), 137--204.   

\bibitem{Pri}J.~Pridham, 
\textit{On $\ell$-adic pro-algebraic and relative pro-$\ell$ fundamental groups}, 
in \textit{The Arithmetic of Fundamental Groups} (J.~Stix ed.), 
Contributions in Mathematical and Computational Sciences \textbf{2} (2012), 245--279. 

\bibitem{SW}
A.I.~Suciu \& H.~Wang, 
Formality properties of finitely generated groups and Lie algebras, 
Forum Mathematicum \textbf{31} (2019), no. 4, 867--905. 

\bibitem{Tol}
V.~Toledano-Laredo \& Y.~Yang, 
\textit{Universal KZB equations for arbitrary root systems}, 
preprint \url{arXiv:1801.10259}.

\bibitem{Ts}
H.~Tsunogai, 
\textit{The stable derivation algebra for higher genera}, 
Israel J. Math., \textbf{136} (2003), 221--250. 

\bibitem{Weil}
A. ~Weil,
\textit{Elliptic functions according to Eisenstein and Kronecker}, 
Ergebnisse der Mathematik und ihrer Grenzgebiete, 
\textbf{88}, Springer-Verlag, Berlin-New York, (1976), ii+93 pp.

\bibitem{Zag}
D.~Zagier, 
\textit{Periods of modular forms and Jacobi theta functions}, 
Invent. Math. \textbf{104} (1991), 449--465. 


\end{thebibliography}
\end{document}